\documentclass[12pt, letterpaper]{amsart} %reqno forces equation numbers to the right side

\usepackage{amsmath, amssymb, amsfonts}
\usepackage{float,longtable}
\usepackage[justification=centering]{caption}

\usepackage{tikz,pgf}
\usetikzlibrary{shapes}

\usepackage[colorlinks=true,urlcolor=blue,bookmarksopen=false,pdfpagemode=UseNone,pdfstartview=Fit]{hyperref}
\usepackage[all]{hypcap}

\usepackage{etoolbox}
\makeatletter
\patchcmd{\@settitle}{\uppercasenonmath\@title}{}{}{}
\patchcmd{\@setauthors}{\MakeUppercase}{}{}{}
\patchcmd{\section}{\normalfont\scshape}{\large}{}{}
\makeatother

\usepackage{fancyhdr}
\fancypagestyle{plain}{%
\fancyhf{} % clear all header and footer fields
\fancyfoot[C]{\vskip1\baselineskip \thepage} % except the center

}

\newtheorem{Thm}{Theorem}[section]
\newtheorem{Lem}[Thm]{Lemma}
\newtheorem{Prop}[Thm]{Proposition}

\newtheorem{Cor}[Thm]{Corollary}
\newtheorem*{BT}{\bfseries Theorem~\ref{muofo} \mdseries (\cite{aB93})}
\theoremstyle{remark}
\newtheorem*{Rem}{Remark}
\theoremstyle{definition}

\newcommand{\bfig}[1]{\begin{figure}[#1] \begin{center}
\addtocounter{Thm}{1}}
\newcommand{\efig}{\end{center} \end{figure}\vskip0\baselineskip}
\newcommand{\btab}[1]{\begin{table}[#1] \begin{center} \addtocounter{Thm}{1}}
\newcommand{\etab}{\end{center} \end{table}\vskip0\baselineskip}

\newcommand{\babs}{\begin{center}\vskip-2\baselineskip\bfseries Abstract\end{center}
\vskip-1\baselineskip \small \leftskip1cm \rightskip1cm}
\newcommand{\eabs}{\normalsize \leftskip0cm \rightskip0cm}

\usepackage{color}

\newcommand{\tw}{\textcolor{white}}
\definecolor{beamerblue}{rgb}{.2,.2,.7}
\newcommand{\tbb}{\textcolor{beamerblue}}
\definecolor{brightbblue}{rgb}{.6,.7,1}

\def\P{\mathbb{P}}

\title{\Large The M\"obius Function of Generalized Factor Order}
\author{\large Robert Willenbring\\
\normalsize Department of Mathematics, University of Mary\\
Bismarck, North Dakota, 58504\\
\href{mailto:rtwillenbring@umary.edu}{rtwillenbring@umary.edu}\\
\ \\
\large \today
}

\begin{document}
\pagestyle{plain}
\maketitle
\thispagestyle{plain}

%Custom Abstract formatting -leftskip/rightskip fixed after intro starts
\babs
We use discrete Morse theory to determine the M\"obius function of generalized factor order. Ordinary factor order on the Kleene closure $A^*$ of a set $A$ is the partial order defined by letting $u\leq w$ if $w$ contains $u$ as a subsequence of consecutive letters.  The M\"obius function of ordinary factor order was determined by Bj\"orner.  Using Babson and Hersh's application of Robin Forman's discrete Morse theory to lexicographically ordered chains, we are able to gain new understanding of Bj\"orner's result and its proof.  We generalize the notion of factor order to take into account a partial order on the alphabet $A$ and, relying heavily on discrete Morse theory, give a recursive formula in the case where each letter of the alphabet covers a unique letter.
%blank line required before \eabs as it is a "different paragraph"

\eabs

\section{Introduction}
\label{intro}
The M\"obius function of factor order was determined by Bj\"orner~\cite{aB93}.  This research was motivated by a desire to give a proof of Bj\"orner's result which provides a deeper explanation of the concepts he used to state his recursive formula, and to generalize these concepts so that they apply to a wider class of posets.  These investigations utilize discrete Morse theory, which also provides a useful context in which to consider the topology of posets.  The major result is a formula for the M\"obius function when factor order is generalized to include a partial ordering of letters, provided each letter covers a unique element.  Since it is clear this formula would have been nearly impossible to discover using other techniques of investigating M\"obius functions, this paper illustrates the ability of discrete Morse theory to simplify complex combinatorial problems of this nature.

Let $A$ be any set.  The \emph{Kleene closure}, $A^*$, is the set of all finite length words over $A$.  So if $w$ is a word and $w(i)$ is the $i^\text{th}$ letter in $w$, then
$$A^*=\{w=w(1)w(2)\ldots w(n): 0\leq n<\infty \text{ and } w(i)\in A \text{ for all } i\}.$$
The \emph{length} of $w$, denoted $|w|$, is the number of letters in $w$.  \emph{Ordinary factor order} on $A^*$ is the partial order on $A^*$ defined by letting $u\leq w$ if $w$ contains a subsequence of consecutive letters $w(i+1)w(i+2)\ldots w(i+n)$ such that $u(j)=w(i+j)$ for $1\leq j\leq n=|u|$.  When $u\leq w$, we call $u$ a \emph{factor} of $w$.  A word $u$ is \emph{flat} if $u(1)=\ldots =u(n)$, where $n=|u|$.

For example, if $A=\{a,b\}$, then $bbabb$ is an element of $A^*$ of length $|bbabb|=5$.  Factors of $bbabb$ include words such as $bbab$, $abb$, and $bb$.  Notice the word $bb$ is flat.

A closed interval $[u,w]$ in $A^*$ is the subposet consisting of all $v\in A^*$ satisfying $u\leq v\leq w$.  The open interval $(u,w)$ is defined similarly.  If the interval $[u,w]$ consists of the two elements $u$ and $w$, we say $w$ \emph{covers} $u$ and write $w\stackrel{ }{\rightarrow}u$.  The Hasse diagram of a poset $P$ is the graph whose vertices are the elements of $P$, and in which an edge is found between two vertices $w$ and $u$ if $w\stackrel{ }{\rightarrow}u$.  For example, if $A=\{a,b\}$, the Hasse diagram of $[b,bbabb]$ in $A^*$ is given in Figure~\ref{fig:b,bbabb}.

Let  $u<w$ be two elements in a poset $P$.  The M\"obius function $\mu$ is a map from $P\times P$ to the integers defined recursively as follows:
\vskip-2.5\baselineskip
\begin{align*}
\hskip.6cm\mu(u,u)&=1\\
\mu(u,w)&=\displaystyle-\sum_{u\leq v< w}\mu(u,v)\\
\mu(w,u)&=0.
\end{align*}
\vskip-.5\baselineskip

\bfig{t}

\begin{tikzpicture}
[every node/.style={ellipse,minimum size=5mm},
value/.style={yshift=-6mm,xshift=-4mm}]

\matrix[shape=rectangle,draw=none,row sep=5mm,column sep=4mm] {
% Row 1:
 & \node(bbabb){\tw{bbabb}}; & \\
% Row 2:
\node(babb){\tw{babb}}; &   & \node(bbab){\tw{bbab}}; \\
% Row 3:
\node(abb){\tw{abb}}; & \node(bab){\tw{bab}};  & \node(bba){\tw{bba}}; \\
% Row 4:
\node(bb){\tw{bb}}; & \node(ab){\tw{ab}};  & \node(ba){\tw{ba}}; \\
% Row 5:
& \node(b){\tw{b}}; & \\
};
\node at (bbabb)[draw]{bbabb};
\node at (babb)[draw]{babb} edge (bbabb);
\node at (bbab)[draw]{bbab} edge (bbabb);
\node at (abb)[draw]{abb} edge (babb);
\node at (bab)[draw]{bab} edge (babb);
\draw (bab) edge (bbab);
\node at (bba)[draw]{bba} edge (bbab);
\node at (bb)[draw]{bb} edge (abb);
\node at (ab)[draw]{ab} edge (abb);
\draw (ab) edge (bab);
\node at (ba)[draw]{ba} edge (bab);
\draw (ba) edge (bba);
\draw (bb) edge (bba);
\node at (b)[draw]{b} edge (ba) edge (bb) edge (ab);

\node[value] at (b){\tbb{$1$}};
\node[value] at (ab){\tbb{$-1$}};
\node[value] at (bb){\tbb{$-1$}};
\node[value] at (ba){\tbb{$-1$}};
\node[value] at (abb){\tbb{$1$}};
\node[value] at (bab){\tbb{$1$}};
\node[value] at (bba){\tbb{$1$}};
\node[value] at (babb){\tbb{$0$}};
\node[value] at (bbab){\tbb{$0$}};
\node[value] at (bbabb){\tbb{$-1$}};

\end{tikzpicture}
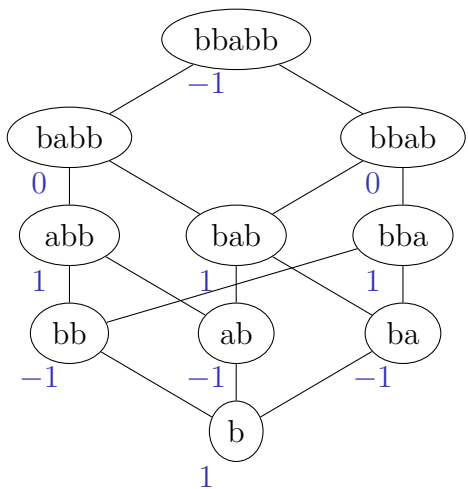
\caption{The interval $[b,bbabb]$. The M\"obius value $\mu(b,v)$ is given to the lower left of each word $v$.}
\label{fig:b,bbabb}
\efig

Figure~\ref{fig:b,bbabb} contains the M\"obius value $\mu(b,v)$ for each word $v$ in the interval $[b,bbabb]$.  Working from the bottom of the diagram to the top, we see that $\mu(b,b)=1$ by the first condition in the definition.  Using the recursive definition, we can calculate $\mu(b,bb)=-\mu(b,b)=-1$.  In higher rows, we again use the recursive definition to calculate $\mu(b,abb)=-(\mu(b,bb)+\mu(b,ab)+\mu(b,b))=1$, and $\mu(b,babb)=-(\mu(b,abb)+\mu(b,bab)+\mu(b,bb)+\mu(b,ab)+\mu(ba,b)+\mu(b,b))=0$.

To state Bj\"orner's formula, we need a few more definitions.  A \emph{prefix} of a word $w\in A^*$ is a factor of $w$ that includes the first letter of $w$.  Similarly, a \emph{suffix} of $w$ is a factor of $w$ that contains the last letter of $w$.  A prefix or suffix is \emph{proper} if it is not equal to $w$.  Define the \emph{outer word} $o(w)$ of $w$ to be the longest factor that appears as both a proper prefix and suffix in $w$. Notice that $o(w)$ can be the empty word.  Define the \emph{inner word} $i(w)$ of $w$ to be the factor $i(w)=w(2)...w(n-1)$, where $n=|w|$. 

For example, prefixes of $bbabb$ include $bb$ and $bba$, while suffixes include $bb$ and $b$.  Note $o(bbabb)=bb$ and $i(bbabb)=bab$.  The word $abb$ has the empty word as its outer word.

The following theorem of Bj\"orner gives a formula for the M\"obius function in ordinary factor order. 

\begin{Thm}[\cite{aB93}]
\label{muofo}
In ordinary factor order, if $u\leq w$ then
$$\mu(u,w)=\begin{cases}
\mu(u,o(w)) & \mbox{if } |w|-|u|>2 \mbox{ and } u\leq o(w)\not\leq i(w),\\
1 &\mbox{if } |w|-|u|=2\mbox{, } w \mbox{ is not flat, and } u=o(w) \mbox{ or } u=i(w),\\
(-1)^{|w|-|u|} &\mbox{if } |w|-|u|<2,\\
0 &\mbox{otherwise.}\end{cases}$$
\end{Thm}

Using this formula, we see that $\mu(b,bbabb)=\mu(b,bb)=-1$.  Since $b$ is the outer word of $bab$, we get $\mu(b,bab)=1$ by the second condition.  Notice that since $o(babb)=b<ab=i(w)$, $\mu(b,babb)=0$.  The reader is encouraged to verify the remaining values implied by the definition of the M\"obius function for the example in Figure~\ref{fig:b,bbabb} are consistent with those found by using this formula.  Notice this formula indicates the only possible values for the M\"obius function in $A^*$ are $-1$, $0$, and $1$.

To prove his result, Bj\"orner relied on successively removing irreducible elements in an interval, where an irreducible element is one covered by or covering exactly one element.  Bj\"orner also considered the M\"obius function of subword order~\cite{aB90}.  In~\cite{SV06}, Sagan and Vatter expanded upon his results in the subword order case.  In one proof of their result, they used the technique of critical maximal chains introduced by Babson and Hersh in~\cite{BH05}.  This technique uses discrete Morse theory, which was developed by Forman~\cite{rF95}.  Since our first goal is to apply the result of Babson and Hersh to reprove Bj\"orner's formula, we will introduce the notation and definitions we need to make use of the relevant theorem.  For a more complete introduction to discrete Morse theory, see Forman's primer~\cite{rF02}.

Discrete Morse theory is an adaptation of Morse theory that can be used to analyze the topology of a simplicial complex.  An \emph{abstract simplicial complex} is a set of vertices $V$ and a set $K$ of subsets of $V$ satisfying the following conditions:
\vskip-1.75\baselineskip

\begin{align*}\mbox{if } v\in V \mbox{, then } \{v\}\in K,\end{align*}
\vskip-1.25\baselineskip

and
\vskip-2.25\baselineskip

\begin{align*}
\mbox{if } \alpha\in K \mbox{ and } \gamma\subseteq\alpha\mbox{, then } \gamma\in K.\end{align*}
\vskip-.75\baselineskip

Each $\alpha$ is called a \emph{simplex}, and if $\alpha<\beta$, $\alpha$ is called a \emph{face} of $\beta$.  The \emph{dimension} of a simplex $\alpha$ is the number of vertices in $\alpha$ minus $1$.  Writing $\alpha^d$ will indicate the simplex $\alpha$ has dimension $d$.

Let $K$ be a simplicial complex and assume $K$ has an empty simplex of dimension $-1$ which is contained in every other simplex.  A function $f: K\longrightarrow\mathbb{R}$ is a \emph{discrete Morse function} if every simplex $\alpha^d$ satisfies the following conditions: 
\vskip-1.5\baselineskip

\begin{align}&\#\{\beta^{d+1}> \alpha | f(\beta)\leq f(\alpha)\}\leq 1\label{dmt1}\\
&\#\{\gamma^{d-1}< \alpha | f(\gamma)\geq f(\alpha)\}\leq 1\label{dmt2}
%\\ &\mbox{If } \beta\in \alpha^{\delta} \mbox{ and } f(\beta)\leq f(\alpha) \mbox{, then } \alpha \mbox{ is a regular face of } \beta.
\end{align}
\vskip-.5\baselineskip

We denote the set in (\ref{dmt1}) by $\alpha^+$ and the set in (\ref{dmt2}) by $\alpha_-$.  

A simplex $\alpha$ is \emph{critical} if $\#\alpha^+=\#\alpha_-=0$.  Note this definition states that with at most one local exception, a Morse function increases with respect to the dimension of a simplex.

\bfig{t}
\begin{tikzpicture}[scale=2,every node/.style={ellipse},
value/.style={yshift=-1mm,xshift=1mm}]
\filldraw[fill=white!80!black%, draw=white!80!black
]
(0,0) -- (1,1.4142) -- (2,0) -- cycle;
\draw (1,1.4142) -- (3,1.4142) -- (2,0);

\filldraw (0,0) circle (.5mm);
\filldraw (1,1.4142) circle (.5mm);
\filldraw (2,0) circle (.5mm);
\filldraw (3,1.4142) circle (.5mm);
\node[yshift=-3mm] at (0,0){\tbb{$0$}};
\node[yshift=3mm] at (1,1.4142){\tbb{$4$}};
\node[yshift=-3mm] at (2,0){\tbb{$2$}};
\node[yshift=3mm] at (3,1.4142){\tbb{$8$}};

\node[yshift=-3mm] at (1,0){\tbb{$1$}};
\node[yshift=3mm] at (2,1.4142){\tbb{$6$}};
\node[xshift=-3mm] at (.5,.7071){\tbb{$3$}};
\node[xshift=3mm] at (1.5,.7071){\tbb{$7$}};
\node[xshift=3mm] at (2.5,.7071){\tbb{$9$}};

\node at (1,.7071){\tbb{$5$}};

\end{tikzpicture}
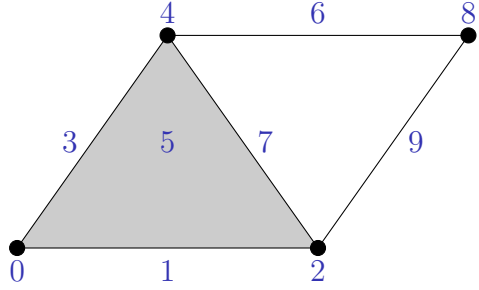
\caption{An example of a Morse function on a simplicial complex}
\label{fig:morse}
\efig

Figure~\ref{fig:morse} contains an example of a simplicial complex $K$ consisting of $4$ simplices of dimension $0$, $5$ simplices of dimension $1$, and $1$ simplex of dimension $2$.  The values of a Morse function $f$ appear next to each simplex.  Note that the edge labeled $9$ and vertex labeled $0$ are the only critical simplices because locally, they are the only simplices for which the Morse function increases with respect to dimension.

One can prove that at most one of the sets $\alpha^+$ and $\alpha_-$ has size $1$.  This result is crucial in proving the results concerning discrete Morse functions in this introduction.  Note that since $\beta\in\alpha^+$ implies $\alpha\in\beta_-$, it follows that simplices which are not critical come in pairs, and these pairs satisfy $f(\alpha^d)>f(\beta^{d+1})$.  A \emph{Morse matching} is a partition of the simplices of $K$ into sets of size one or two such that each one element set contains a critical simplex of a Morse function $f$ and each two element set consists of two non-critical simplices $\alpha^d<\beta^{d+1}$ satisfying $f(\alpha^d)>f(\beta^{d+1})$.

The Morse matching of the simplices for the function shown in Figure~\ref{fig:morse} is $\{0\},\{2,1\},$ $\{4,3\},\{7,5\},\{8,6\},\{9\}$, where, as an abuse of notation, each number refers to the simplex to which it is assigned.

Critical simplices are important from a topological viewpoint because Forman shows in~\cite{rF95} that $K$ is homotopy equivalent to a CW-complex with exactly one cell of dimension $d$ for each critical simplex $\alpha$ of dimension $d$.  Forman also proves the following Theorem in~\cite{rF95}.

\begin{Thm}[Weak Morse Inequalities]
\label{WMI}
Let $\tilde{m}_d$ be the number of critical simplices of dimension $d$, $\tilde{b}_d$ be the $d$-th reduced Betti number over the integers, and $\tilde{\chi}$ be the reduced Euler characteristic.  Then
\vskip-2\baselineskip

\begin{align*}\tilde{b}_d\leq\tilde{m}_d \mbox{ for } d\geq-1\nonumber\end{align*}
\vskip-1\baselineskip

and
\vskip-2\baselineskip

\begin{align*}\tilde{\chi}(W)=\sum_{d\geq-1}(-1)^d\tilde{m}_d.\nonumber\end{align*}
\vskip-1\baselineskip
\end{Thm}
%\vskip-1\baselineskip

Discrete Morse theory can be used to find information about the M\"obius function of a poset by using the connection between it and the reduced Euler characteristic of the order complex of a poset.  A chain in a poset is set of elements $\{v_0,v_1,\ldots,v_n\}$ such that $v_0>v_1>\ldots>v_n$.  For clarity, we write a chain $C$ as $C: v_0 >v_1>\ldots>v_n$.  Given two elements $u$ and $w$ of a poset $P$,  the order complex $\Delta(u,w)$ is the abstract simplicial complex whose simplices are the chains in the open interval $(u,w)$.  An important fact about the M\"obius function~\cite{rS96} is
$$\mu(u,w)=\tilde{\chi}(\Delta(u,w)).$$
Therefore, by using discrete Morse theory to investigate the chains of $(u,w)$, it is possible to calculate the M\"obius function and obtain additional information about the topology of the order complex.

Given two elements $u,w$ in a poset $P$, Babson and Hersh have developed a way of finding a Morse matching for $\Delta(u,w)$ which gives a relatively small number of critical simplices.  We need to develop a considerable amount of terminology to properly state their result.

Let $C: v_0 \stackrel{ }{\rightarrow}v_1\stackrel{ }{\rightarrow}\ldots\stackrel{ }{\rightarrow}v_n$ be a chain in a poset $P$.  Since each pair of adjacent elements are related by a cover, $C$ is called a \emph{saturated chain}.  The \emph{closed interval} of $C$ from $v_i$ to $v_j$ is the chain $C[v_i,v_j]:v_i \stackrel{ }{\rightarrow}v_{i+1}\stackrel{ }{\rightarrow}\ldots\stackrel{ }{\rightarrow}v_j$.  The \emph{open interval} of $C$ from $v_i$ to $v_j$, $C(v_i,v_j)$, and the half open intervals $C[v_i,v_j)$ and $C(v_i,v_j]$ are defined similarly.  The closed interval $C[v_{i},v_{i}]$ consisting of the single element $v_i$ will also be written $v_i$, but the context will always indicate whether we are referring to the element or the interval.  Notice that the interval $[u,w]$ is non-empty when $u\leq w$ in the poset $P$, while $C[v_i,v_j]$ is non-empty when $v_i\geq v_j$.  A  chain $C$ of the interval $[u,w]$ is a \emph{maximal chain} if $v_0=w$ and $v_n=u$.

Given two maximal chains  $C: v_0 \stackrel{ }{\rightarrow}v_1\stackrel{ }{\rightarrow}\ldots\stackrel{ }{\rightarrow}v_n$ and  $D: w_0 \stackrel{ }{\rightarrow}w_1\stackrel{ }{\rightarrow}\ldots\stackrel{ }{\rightarrow}w_n$ in an interval $[u,w]$, we say $C$ and $D$ \emph{agree to index $k$} if $v_i=w_i$ for all $i\leq k$.  We say $C$ and $D$ \emph{diverge from index $k$} if $C$ and $D$ agree to index $k$ and $v_{k+1}\neq w_{k+1}$.  

\btab{t}
\begin{longtable}[!h]{ccccccccc}
$v_0$ &  & $v_1$  &  & $v_2$  &  & $v_3$  &  & $v_4$\\
 $bbabb$ & $\longrightarrow$  & $babb$ & $\longrightarrow$  & $abb$ & $\longrightarrow$  & $bb$ & $\longrightarrow$  & $b$ \\
 $bbabb$ & $\longrightarrow$ & $babb$ & $\longrightarrow$ & $abb$ & $\longrightarrow$ & $ab$ & $\longrightarrow$ & $b$  \\
 $bbabb$ & $\longrightarrow$ & $babb$ & $\longrightarrow$ & $bab$ & $\longrightarrow$ & $ab$ & $\longrightarrow$ & $b$  \\
 $bbabb$ & $\longrightarrow$ & $babb$ & $\longrightarrow$ & $bab$ & $\longrightarrow$ & $ba$ & $\longrightarrow$ & $b$  \\
 $bbabb$ & $\longrightarrow$ & $bbab$ & $\longrightarrow$ & $bab$ & $\longrightarrow$ & $ab$ & $\longrightarrow$ & $b$ \\
 $bbabb$ & $\longrightarrow$ & $bbab$ & $\longrightarrow$ & $bab$ & $\longrightarrow$ & $ba$ & $\longrightarrow$ & $b$ \\
 $bbabb$ & $\longrightarrow$ & $bbab$ & $\longrightarrow$ & $bba$ & $\longrightarrow$ & $ba$ & $\longrightarrow$ & $b$ \\
 $bbabb$ & $\longrightarrow$ & $bbab$ & $\longrightarrow$ & $bba$ & $\longrightarrow$ & $bb$ & $\longrightarrow$ & $b$ \\
\end{longtable}
\caption{A poset lexicographic ordering of the eight maximal chains of $[b,bbabb]$}
\label{table:PLO}
\etab
\vskip0\baselineskip

In Table~\ref{table:PLO}, we have listed the eight maximal chains of $[b,bbabb]$.  The first two maximal chains agree to indices $0$, $1$, and $2$.  These two chains diverge from index $2$.

A total ordering $C_1<C_2<\ldots<C_n$ of the maximal chains of an interval is a \emph{poset lexicographic order} if it satisfies the following: suppose $C<D$ and $C$ and $D$ diverge from index $k$; if $C'$ and $D'$ agree to index $k+1$ with $C$ and $D$, respectively, then $C'<D'$.  

To illustrate the definition, let us investigate one case in Table~\ref{table:PLO}. Note the first chain and eighth chain agree to index $0$.  Since the second chain agrees with the first to index $1$, and the seventh chain agrees with the eighth to index $1$, we would need to have the second chain appear earlier than the seventh in order for this to be a poset lexicgoraphic order.  In fact, the ordering given in Table~\ref{table:PLO} is a poset lexicographic ordering of the maximal chains of $[b,bbabb]$.

To verify we have given a poset lexicographic ordering of the maximal chains of $[b,bbabb]$, it is easiest to informally discuss how it was created.  The formal construction appears in Section~\ref{ofo}.  To get this ordering, we record the positions $l_i$, relative to $w$, removed to get from $v_{i-1}$ to $v_i$.  The sequence $l_1-l_2-\ldots-l_n$ is called the chain id of the maximal chain.  By lexicographically ordering these chain ids, we get a poset lexicographic order (in terms of Babson and Hersh, this is a chain labeling).  To clearly indicate where each $l_i$ comes from, we can zero out the positions removed in each word, creating an embedding of the word $v_i$ into $w$.  If a word consists entirely of one letter (such as $bb$), removing any letter gives the same word, so by convention, we only allow the first (non-zero) position to be removed.  

In Table~\ref{table:CID}, we give same poset lexicographic order as above with this extra information about the order.  Notice that when only one digit numbers are used for each label $l_i$, lexicographically ordering the chain ids is equivalent to ordering $n$ digit numbers by size.

Recall we need to consider the chains of the open interval $(u,w)$ because the order complex $\Delta(u,w)$ is defined in terms of the open interval. Note that the maximal chains of $(u,w)$ are in a one to one correspondence with those of $[u,w]$: by removing the first and last vertex from a maximal chain in $[u,w]$, we get a maximal chain in $(u,w)$.  Thus, we will continue to list $v_0$ and $v_n$ in our examples to make it easier to identify the labels $l_1$ and $l_n$.

\btab{t}
\begin{longtable}[!h]{rccccccccc}
Chain Id & $v_0$ & $l_1$ & $v_1$  & $l_2$ & $v_2$  & $l_3$ & $v_3$  & $l_4$ & $v_4$ \\
1-2-3-4  & $bbabb$ & $1$  & $0babb$ & $2$  & $00abb$ & $3$  & $000bb$ & $4$  & $0000b$ \\
1-2-5-3 & $bbabb$ & $1$ & $0babb$ & $2$ & $00abb$ & $5$ & $00ab0$ & $3$ & $000b0$  \\
1-5-2-3 & $bbabb$ & $1$ & $0babb$ & $5$ & $0bab0$ & $2$ & $00ab0$ & $3$ & $000b0$  \\
1-5-4-3 & $bbabb$ & $1$ & $0babb$ & $5$ & $0bab0$ & $4$ & $0ba00$ & $3$ & $0b000$  \\
5-1-2-3 & $bbabb$ & $5$ & $bbab0$ & $1$ & $0bab0$ & $2$ & $00ab0$ & $3$ & $000b0$ \\
5-1-4-3 & $bbabb$ & $5$ & $bbab0$ & $1$ & $0bab0$ & $4$ & $0ba00$ & $3$ & $0b000$ \\
5-4-1-3 & $bbabb$ & $5$ & $bbab0$ & $4$ & $bba00$ & $1$ & $0ba00$ & $3$ & $0b000$ \\
5-4-3-1 & $bbabb$ & $5$ & $bbab0$ & $4$ & $bba00$ & $3$ & $bb000$ & $1$ & $0b000$ \\
\end{longtable}
\caption{The same ordering of the maximal chains of $[b,bbabb]$ with information\\ about the chain ids.}
\label{table:CID}
\etab
\vskip0\baselineskip

Suppose $C_1<C_2<\ldots<C_n$ is a lexicographic ordering of the maximal chains of $(u,w)$.   Recall that the chains of $(u,w)$ are the simplices in the order complex $\Delta(u,w)$.  Call a simplex contained in $C$ \emph{new} if it is not contained in any $C'$ for $C'<C$.  We can inductively define a Morse matching on $\Delta(u,w)$ by extending the matching at the $k^{\text{th}}$ step to the new simplices of $C$. In fact, Babson and Hersh show in~\cite{BH05} that a matching can be constructed in this manner so that during each step $k$, at most one new simplex is a critical simplex.  So, using this process, adding each maximal chain adds at most one critical simplex.  We will refer to a maximal chain that contributes a critical simplex as a \emph{critical chain}.

To motivate the next set of definitions, we consider which subchains of a maximal chain $C$ appear in a lexicographically earlier chain, and which ones are new simplices.  In a maximal chain $C$ of a poset lexicographic ordering, Babson and Hersh prove that each maximal subchain of $C$ which appears in a lexicographically earlier chain consists of a subchain of $C$ given by a skipping a single interval of consecutive ranks.  This is because the poset lexicographic ordering assures that if $C$ and $C'$ diverge from index $k$, but later are the same at index $\ell$, there is some chain $D'<C$, possibly equal to $C'$, such that $C$ and $D'$ are the same outside the interval $C(v_k,v_\ell)$.  These skipped intervals are referred to as minimally skipped intervals.  Therefore, a simplex in $C$ entirely belongs to a lexicographically earlier chain if it entirely misses any of the minimally skipped intervals.  Equivalently, the new simplices of $C$ are those subchains of $C$ which intersect every minimally skipped interval of $C$ non-trivially.

Formally, given an ordering of the maximal chains of $[u,w]$, a non-empty interval $(v_i,v_j)$ is a \emph{skipped interval} of a maximal chain $C$ if
$$C-C(v_i,v_j) \subset C' \mbox{ for some } C' < C.$$
It is a \emph{minimally skipped interval (MSI)} if it does not properly contain another skipped interval.  We write $I(C)$ for the set of all MSIs of a chain $C$.  To find the set $I(C)$, consider each interval $I\subseteq C$ and see if $C-I\subset C'$ for any $C'\subset C$, then throw out any such interval that is nonminimal.  

\btab{t}
\begin{longtable}[!h]{@{\extracolsep{-4pt}}rcccccccccl}
Chain Id & $v_0$ & $l_1$ & $v_1$  & $l_2$ & $v_2$  & $l_3$ & $v_3$  & $l_4$ & $v_4$  & MSIs \\
1-2-3-4  & $bbabb$ & $1$  & $0babb$ & $2$  & $00abb$ & $3$  & $000bb$ & $4$  & $0000b$  & \\
1-2-5[-]3 & $bbabb$ & $1$ & $0babb$ & $2$ & $00abb$ & $5$ & $[00ab0]$ & $3$ & $000b0$   & $[ab, ab]$ \\
1-5[-]2-3 & $bbabb$ & $1$ & $0babb$ & $5$ & $[0bab0]$ & $2$ & $00ab0$ & $3$ & $000b0$   & $[bab, bab]$ \\
1-5-4[-]3 & $bbabb$ & $1$ & $0babb$ & $5$ & $0bab0$ & $4$ & $[0ba00]$ & $3$ & $0b000$   & $[ba, ba]$ \\
5[-]1-2-3 & $bbabb$ & $5$ & $[bbab0]$ & $1$ & $0bab0$ & $2$ & $00ab0$ & $3$ & $000b0$   & $[bbab, bbab]$ \\
5[-]1-4[-]3 & $bbabb$ & $5$ & $[bbab0]$ & $1$ & $0bab0$ & $4$ & $[0ba00]$ & $3$ & $0b000$   & $[bbab, bbab]$ $[ba, ba]$ \\
5-4[-]1-3 & $bbabb$ & $5$ & $bbab0$ & $4$ & $[bba00]$ & $1$ & $0ba00$ & $3$ & $0b000$   & $[bba, bba]$ \\
5[-4-]3[-]1 & $bbabb$ & $5$ & $[bbab0$ & $4$ & $bba00]$ & $3$ & $[bb000]$ & $1$ & $0b000$ &  $[bbab, bba]$ $[bb, bb]$ \\
\end{longtable}
\vskip0\baselineskip

The new simplices, which are in $\displaystyle C\setminus(\cup_{C'<C} C')$, for selected maximal chains $C$:
\vskip.5\baselineskip

$1-5-2-3:\quad bab, \quad babb-bab,\quad bab-ab,\quad babb-bab-ab$\\
$5-4-3-1:\quad bba-bb,\quad bbab-bb,\quad bbab-bba-bb$
\vskip0\baselineskip

\caption{The MSIs of the maximal chains of $[b,bbabb]$.}
\label{table:MSIs}
\etab

For an example of both MSIs and new simplices, see Table~\ref{table:MSIs}.  In this table, we have placed brackets around each interval which is minimally skipped, and placed the corresponding brackets into the chain id.  For example, the intervals $D[bb,bb]$ and $D[bbab,bba]$ are minimally skipped intervals of chain $D$ with chain id $5-4-3-1$.  To see how one determines the minimally skipped intervals, consider the third chain $C$ with chain id $1-5-2-3$.  The only intervals $I$ satisfying $C-I\subset C'$ for some $C'<C$ are $C[bab,ab]$, since $C-C[bab,ab]$ is in the first chain, and $C[bab,bab]$, since $C-C[bab,bab]$ is in the second chain.  Since the second interval is contained in the first, the only MSI in the third chain is $C[bab,bab]$.  So the new simplices in this chain, which intersect this MSI, are the subchains $bab$, $babb-bab$, $bab-ab$, and $babb-bab-ab$.

The set of MSIs $I(C)$ \emph{covers} $C$ if its union equals the open interval $C(v_0,v_n)$.  This last definition reflects the fact that the order complex of an interval is constructed without the maximum and minimum elements.  Notice the set $I(D)$ for the maximal chain $D$ with chain id $5-4-3-1$ covers $D$, and that the set of MSIs $I(C)$ for any other maximal chain $C$ in the interval $[b,bbabb]$ does not cover $C$.

Notice $I(C)$ could contain intervals that overlap, that is, intervals with non-empty intersection. We will need to produce a set of disjoint intervals from $I(C)$, which we will call $J(C)$.  We construct $J(C)=\{J_1,J_2,\ldots\}$ as follows.  Order the intervals of $I(C)$ based on when they are first encountered in $C$.  Thus, $I_1$ will contain the word $v_i$ of smallest index that appears in any interval in $I(C)$, $I_2$ will contain the word $v_j$ of smallest index that appears in any interval in $I(C)$ not equal to $I_1$, etc.  Let $J_1=I_1$.  Then consider the intervals $I'_2=I_2-J_1$, $I'_3=I_3-J_1$, and so forth. Throw out any that are not containment minimal, and pick the first one that remains to be $J_2$.  Continue this process until no intervals remain to add to $J(C)$.

For an example of the difference between $I(C)$ and $J(C)$, we need to consider a new interval.  In the interval $[a,abbabb]$, whose maximal chains are found in Table~\ref{table:JC}, the chain $C$ with chain id $6-5-4-3-2$ has a set of MSIs $I(C)$ in which there is overlap.  To construct $J(C)$, we first add the MSI $I_1=C[abbab,abba]$ of $I(C)$ to $J(C)$.  Then we truncate the remaining intervals in $I(C)$.  In particular, $I_2=C[abba,abb]$ becomes $I'_2=C[abb,abb]$, while the chain  $I_3=C[abb,ab]$ does not overlap with $I_1$, giving $I'_3=I_3$.  Now we see that $I'_2\subset I'_3$.  So we remove $I'_3$ from the set of intervals under consideration, and add $I'_2$ to $J(C)$.  At this point, no intervals of the original set $I(C)$ remain to be considered, so $J(C)=\{I_1,I'_2\}$.

The following theorem of Babson and Hersh gives the connection between $J(C)$ and $\mu(u,w)$.  For a description, and example, of the Morse matching of simplices which leads to this Theorem, please see Appendix~\ref{BHmatch}.

%\begin{landscape}
%\pagestyle{empty}
\btab{t}

 $I(C)$ intervals for  $[a,abbabb]$
\begin{longtable}[!h]{@{\extracolsep{-4pt}}rccccccccccc}
Chain Id & $v_0$ & $l_1$ & $v_1$  & $l_2$ & $v_2$  & $l_3$ & $v_3$  & $l_4$ & $v_4$  & $l_5$ & $v_5$    \\
1-2-3-6-5  & $abbabb$ & $1$  & $0bbabb$ & $2$  & $00babb$ & $3$  & $000abb$ & $6$  & $000ab0$ & $5$  & $000a00$   \\
1-2-6[-]3-5 & $abbabb$ & $1$ & $0bbabb$ & $2$ & $00babb$ & $6$ & $[00bab0]$ & $3$ & $000ab0$ & $5$ & $000a00$    \\
1-2-6-5[-]3 & $abbabb$ & $1$ & $0bbabb$ & $2$ & $00babb$ & $6$ & $00bab0$ & $5$ & $[00ba00]$ & $3$ & $000a00$    \\
1-6[-]2-3-5 & $abbabb$ & $1$ & $0bbabb$ & $6$ & $[0bbab0]$ & $2$ & $00bab0$ & $3$ & $000ab0$ & $5$ & $000a00$    \\
1-6[-]2-5[-]3 & $abbabb$ & $1$ & $0bbabb$ & $6$ & $[0bbab0]$ & $2$ & $00bab0$ & $5$ & $[00ba00]$ & $3$ & $000a00$    \\
1-6-5[-]2-3 & $abbabb$ & $1$ & $0bbabb$ & $6$ & $0bbab0$ & $5$ & $[0bba00]$ & $2$ & $00ba00$ & $3$ & $000a00$    \\
6[-]1-2-3-5 & $abbabb$ & $6$ & $[abbab0]$ & $1$ & $0bbab0$ & $2$ & $00bab0$ & $3$ & $000ab0$ & $5$ & $000a00$    \\
6[-]1-2-5[-]3 & $abbabb$ & $6$ & $[abbab0]$ & $1$ & $0bbab0$ & $2$ & $00bab0$ & $5$ & $[00ba00]$ & $3$ & $000a00$    \\
6[-]1-5[-]2-3 & $abbabb$ & $6$ & $[abbab0]$ & $1$ & $0bbab0$ & $5$ & $[0bba00]$ & $2$ & $00ba00$ & $3$ & $000a00$    \\
6-5[-]1-2-3 & $abbabb$ & $6$ & $abbab0$ & $5$ & $[abba00]$ & $1$ & $0bba00$ & $2$ & $00ba00$ & $3$ & $000a00$   \\
6[-5[-]4[-]3-]2 & $abbabb$ & $6$ & $[abbab0$ & $5$ & $[abba00]$ & $4$ & $[abb000]$ & $3$ & $ab0000]$ & $2$ & $a00000$    \\
\end{longtable}

$J(C)$ intervals for  $[a,abbabb]$

The set $I(C)$ only changes for the last chain:

\begin{longtable}[!h]{@{\extracolsep{-4pt}}rccccccccccc}
Chain Id & $v_0$ & $l_1$ & $v_1$  & $l_2$ & $v_2$  & $l_3$ & $v_3$  & $l_4$ & $v_4$  & $l_5$ & $v_5$ \\
6[-5-]4[-]3-2 & $abbabb$ & $6$ & $[abbab0$ & $5$ & $abba00]$ & $4$ & $[abb000]$ & $3$ & $ab0000$ & $2$ & $a00000$  \\
\end{longtable}
\caption{Comparing $I(C)$, $J(C)$ for the maximal chains of $[a,abbabb]$.}
\label{table:JC}
\etab
%\ 
%\vskip1in

%\begin{center}13\end{center}

%\end{landscape}
%\end{rotate}

\begin{Thm}[\cite{BH05}]
\label{chainmu}
Let $P$ be a poset and $[u,w]$ be a finite interval in $P$.  For any poset lexicographic order on the maximal chains of $[u,w]$, the above construction can be used to produce a Morse matching which matches all the new simplices of a chain, except possibly one which is critical.  The set $J(C)$ for each $C$ has the following properties:

\begin{enumerate}
\item\vskip-.25\baselineskip A maximal chain $C$ is critical if and only if $J(C)$ covers $C$.
\vskip-.35\baselineskip

\item If $C$ is critical, then the dimension of its critical simplex is
\vskip-1.25\baselineskip

$$d(C)=\#J(C)-1.$$

\item \vskip-.5\baselineskip The M\"obius value from $u$ to $w$ is
\vskip-1.45\baselineskip

$$\mu(u,w)=\sum_{C}(-1)^{d(C)}$$
\vskip-.1\baselineskip

where the sum is over all critical chains $C$ in $[u,w]$.
%\qed
\end{enumerate}
\end{Thm}

The rest of the paper is organized as follows.  In Section~\ref{ofo}, we use Theorem~\ref{chainmu} to reprove Bj\"orner's formula in Theorem~\ref{muofo}.  We also give a simple proof of his characterization of the topology of posets ordered by ordinary factor order using discrete Morse theory. In Section~\ref{gfo}, we will consider generalized factor order on the integers.  Section~\ref{gfoTF} considers generalized factor order on rooted forests, and gives a formula for the M\"obius function which contains the formula of Section~\ref{gfo} and Bj\"orner's formula as subcases.  This is not obvious and in fact two independent proofs are needed to establish the connection.  Section~\ref{open} discusses open problems related to this work.

\section{Ordinary Factor Order}
\label{ofo}

Let $A$ be any set.  Partially order $A^*$ using ordinary factor order.  To get a sense of the structure of the poset $A^*$, we first consider the covering relations in this poset.

\begin{Lem}
\label{cover}
A word $w=w(1)\ldots w(n)$ in $A^*$ can only cover the words $w(2)\ldots w(n)$ and $w(1)\ldots w(n-1)$.  These two words are distinct unless $w$ is flat, in which case they are equal and flat.
\end{Lem}

\begin{proof}
Note $w$ covers words of length $|w|-1$.  Since the letters of a factor of $w$ must appear consecutively in $w$, the longest proper suffix $w(2)\ldots w(n)$ and the longest proper prefix $w(1)\ldots w(n-1)$ are the only two words $w$ covers.  Should these two words be equal, we have $w(1)=w(2)$, $w(2)=w(3)$, $\ldots$, and $w(n-1)=w(n)$ so that $w(1)=w(2)=\ldots=w(n)$. This implies that $w$ is flat when it covers a single word.
\end{proof}

We will now formalize some of the concepts introduced informally in the introduction.  Suppose there is an element $0\notin A$.  A word $\eta\in(A\cup 0)^*$ is an \emph{expansion} of $u \in A^*$ if $\eta\in0^*u0^*$, where adjacency denotes concatenation of words and letters.  An \emph{embedding} of $u$ into $w$ is an expansion $\eta$ of $u$ with length $|w|$ such that for all $i$, either $\eta(i)=w(i)$ or $\eta(i)=0$.  In the latter case, we say $w(i)$ is \emph{reduced to $0$}. 

Note that the words appearing in Tables~\ref{table:CID} through~\ref{table:JC} are actually the embeddings of the words $v_i$.  For example, we note that $bb000$ is the prefix embedding of $bb$ into $bbabb$, while $000bb$ is the suffix embedding.

Let $[u,w]$ be an interval in $A^*$.  Let $w=v_0\stackrel{l_1}{\rightarrow}v_1\stackrel{l_2}{\rightarrow}\ldots\stackrel{l_{n-1}}{\rightarrow} v_{n-1}\stackrel{l_n}{\rightarrow}v_n=u$ be a maximal chain in $[u,w]$, where the $l_i$ are defined by the corresponding sequence of embeddings of $v_i$ into $w$, $\eta_{v_i}$, in the sense that
\vskip-.5\baselineskip

$$\eta_{v_i}(l_i)=0 \text{ and } \eta_{v_i}(j)=\eta_{v_{i-1}}(j) \text{ when } j\neq l_i.$$
In the case where $v_{i-1}$ is flat, we require $l_{i}$ to be the smaller of the two possible values.  Note that this gives each maximal chain its own unique sequence $l_1\ldots l_n$ which we can use to identify it.  We call this sequence a maximal chain's \emph{chain id}.  By lexicographically ordering these chain ids, we produce a poset lexicographic order on the maximal chains of $[u,w]$.  This is the order we will use to find the MSIs of the maximal chains $C$, and ultimately the sets $J(C)$ which will allow us to apply Theorem~\ref{chainmu}.  Examples of this ordering were given in Tables~\ref{table:CID} through~\ref{table:JC}.

To facilitate the exposition, we make the following definitions. A \emph{descent} in a maximal chain is a word $v_i$ where $l_i>l_{i+1}$.  We say $v_i$ is a \emph{strong descent} if $l_i>l_{i+1}+1$, and a \emph{weak descent} if $l_i=l_{i+1}+1$.  An \emph{ascent} in a maximal chain is a word $v_i$ where $l_i<l_{i+1}$.  In Table~\ref{table:CID}, the chain $5-4-1-3$ has $v_1=bbab$ as a weak descent, $v_2=bba$ as a strong descent, and $v_3=ba$ as an ascent.

\begin{Lem}
\label{descent}
Let $C:v_0\stackrel{l_1}{\rightarrow}v_1\stackrel{l_2}{\rightarrow}\ldots\stackrel{l_{n-1}}{\rightarrow} v_{n-1}\stackrel{l_n}{\rightarrow}v_n$ be a maximal chain in $[u,w]$.  If $v_i$ is a strong descent, then $v_i$ is an MSI in $C$.
\end{Lem}

\begin{proof}
Since $l_{i}-l_{i+1}>1$, $v_i$ is $v_{i-1}$ with the last letter reduced and $i(v_{i-1})=v_{i+1}$.  Let $v'$ be $v_{i-1}$ with the first letter reduced. 
 
Suppose $v'$ is not flat.  Then
$$C':v_0\stackrel{l_1}{\rightarrow}\ldots\stackrel{l_2}{\rightarrow} v_{i-2}\stackrel{l_{i-1}}{\rightarrow} v_{i-1}\stackrel{l_{i+1}}{\rightarrow}v'\stackrel{l_{i}}{\rightarrow}v_{i+1}\stackrel{l_{i+2}}{\rightarrow}\ldots\stackrel{l_{n}}{\rightarrow} v_n$$ 
is a lexicographically earlier chain than $C$.  Hence, $v_i$ is a skipped interval in $C$.  Since $v_i$ is an interval consisting of a single element, $v_i$ is an MSI.

Suppose $v'$ is not flat.  Then $l_i$ cannot be reduced in $v'$.  However, the chain
$$D':v_0\stackrel{l_1}{\rightarrow}\ldots\stackrel{l_2}{\rightarrow} v_{i-2}\stackrel{l_{i-1}}{\rightarrow} v_{i-1}\stackrel{l_{i+1}}{\rightarrow}v'\stackrel{l_{i+2}}{\rightarrow}v_{i+1}\stackrel{l_{i+3}}{\rightarrow}\ldots\stackrel{l_{n}}{\rightarrow} v_{n-1}\stackrel{l_{n}+1}{\rightarrow} v_{n}$$
is a lexicographically earlier chain than $C$ because each $v_j$ for $j>i$ is flat.  Hence, $v_i$ is an MSI in this case as well.\end{proof}

Considering this result in Tables~\ref{table:MSIs} and~\ref{table:JC} shows that it accounts for a small proportion of the MSIs in ordinary factor order.  The next result, however, gives a great deal of information about the critical chains of ordinary factor order.

\begin{Lem}
\label{ascent}
Let $C:v_0\stackrel{l_1}{\rightarrow}v_1\stackrel{l_2}{\rightarrow}\ldots\stackrel{l_{n-1}}{\rightarrow} v_{n-1}\stackrel{l_n}{\rightarrow}v_n$ be a maximal chain in $[u,w]$.  If $v_i$ is an ascent, then it is not contained in any MSI.
\end{Lem}

\begin{proof}
We will prove this lemma by contradiction.  Suppose $C[v_r,v_s]$ is an MSI that contains $v_i$. Notice that $v_i$ may only be preceded by ascents in this interval because if there are descents, the last one that occurs before $v_i$ would be a strong descent.  By Lemma~\ref{descent}, this would be an MSI, contradicting the minimality of $C[v_r,v_s]$.  Thus, it suffices to derive a contradiction for $v_i=v_{r}$, the first ascent in the interval $C[v_r,v_s]$.

Since $C[v_r,v_s]$ is an MSI of $C[w,u]$ if and only if it is an MSI of $C[v_{r-1},v_{s+1}]$, it suffices to consider the case $r=1$.   However, if $r=1$ then $v_1$ being an ascent forces $l_1=1$.  This implies $v_1$ appears in all chains preceding $C$.  Therefore, $v_r$ can be removed from any skipped interval in which it appears and that interval will still be skipped, contradicting the fact that $C[v_r,v_s]$ is minimal.
\end{proof}

Notice that Lemma~\ref{ascent} implies that all MSIs of a chain $C$ consist entirely of descents.  Thus, only the lexicographically last chain in an interval can possibly be critical.  The next lemma covers the two basic cases of MSIs in a chain $C$.  We have already encountered the first case, while the second case is new.

\begin{Lem}
\label{inout} Suppose $w$ is not flat and $|w|\geq2$:\begin{enumerate}
\item There are two maximal chains in the interval $[i(w),w]$, and if $C$ is the second chain, then it has a unique MSI, $C(w,i(w))$.
\item If $o(w)\not< i(w)$, there are two maximal chains in the interval $[o(w),w]$, and if $C$ is the second chain, then it has a unique MSI, $C(w,o(w))$.
\end{enumerate}
\end{Lem}

\begin{proof}
The first case follows from Lemma~\ref{descent}.  For the second case,   once we reduce the first or last letter, there is only one copy of $o(w)$ left in $w$.  Therefore, there are two maximal chains in the interval $[o(w),w]$: the first chain, which ends at the suffix embedding, and the last chain, which ends at the prefix embedding.  Since these two chains share only $o(w)$ and $w$ in common, $C(w,o(w))$ is an MSI, completing the proof.
\end{proof}

We can illustrate this lemma using the intervals $[aa,baab]$ and $[b,baab]$.  Note the inner word of $baab$ is $aa$, so that the maximal chains of $[aa,baab]$ are $baab-aab-aa$ and $baab-baa-aa$, giving $C(baab,aa)$ as an MSI in the second chain.  The outer word of $baab$ is $b$, so that the maximal chains of $[b,baab]$ are $baab-aab-ab-b$ and $baab-baa-ba-b$, giving $C(baab,b)$ as an MSI in the second chain.

There is another way to think about these two types of MSIs that will prove useful moving forward.  In the first type, the embedding of $i(w)$ in the critical chain and first chain are the same when $i(w)$ is not flat.  In the second type, the embeddings of $o(w)$ in the critical chain and the first chain are different.  Notice that if both $i(w)$ and $o(w)$ are flat, then  $w$ is flat, $o(w)<i(w)$, and there is a unique maximal chain in every non-empty interval $[u,w]$.  Thus, there can be no overlap between these two types of MSI.  We will see this observation about same and different embeddings provides a very useful way of determining how MSIs arise, even though it does not apply to MSIs that end at flat words.

Proposition~\ref{outerword} generalizes Lemma~\ref{inout}(2), and the theorem that follows shows that we have identified all cases of MSIs.  It will be convenient to adopt the convention that a sequence $l_{i+1}$ consisting of a single label is not decreasing, corresponding to the fact that the interval $C(v_i,v_{i+1})$ is empty and so not an MSI.

\begin{Prop}
\label{outerword}
Let $C:w=v_0\stackrel{l_1}{\rightarrow}v_1\stackrel{l_2}{\rightarrow}\ldots\stackrel{l_{n-1}}{\rightarrow} v_{n-1}\stackrel{l_n}{\rightarrow}v_n=u$ be a maximal chain in the interval $[u,w]$. Suppose there are $i$ and $j$ such that $v_j=o(v_i)\not< i(v_i)$, and such that the sequence $l_{i+1},\ldots, l_j$ is decreasing.  Then $C(v_{i},v_{j})$ is an MSI in $C$.
\end{Prop}

\begin{proof}
Since the sequence $l_{i+1},\ldots, l_j$ is decreasing, $v_i$ can not be flat and $j\geq i+2$.  So Lemma~\ref{inout}(2) implies that $C(v_{i},v_{j})$ is an MSI in the subchain of $C$ that is its intersection with $[v_j,v_i]$.  So there is a lexicographically earlier maximal chain $D$ in $[v_j,v_i]$.  Thus, the chain $C'$ formed by replacing the subchain $C[v_i,v_j]$ in $C$ with $D$ yields $C-C(v_{i},v_{j})\subset C'$.  Therefore, $C(v_{i},v_{j})$ is a skipped interval of $C$. 

To see that it must also be an MSI, note that $o(v_i)\not< i(v_i)$ so that $[v_j,v_i]$ has only two maximal chains and they intersect only at $v_i$ and $v_j$.  So the same must be true of any maximal chain in $[u,w]$ containing $v_i$ and $v_j$.  This forces minimality.
\end{proof}

\begin{Thm}
\label{MSIchar}
The interval $C(v_i,v_j)$ is an MSI of a maximal chain $C$ of $[u,w]$ if and only if $C(v_i,v_j)=v_{i+1}$ and $v_{i+1}$ is a strong descent, or $v_j=o(v_i)\not<i(v_i)$, where the sequence $l_{i+1},\ldots, l_j$ is decreasing.
\end{Thm}

\begin{proof}
The reverse implication follows from Lemma~\ref{descent} and Proposition~\ref{outerword}.

Suppose $C(v_i,v_j)$ is an MSI in $C$.  Note by Lemma~\ref{ascent} the sequence $l_{i+1},\ldots,l_{j}$ is a decreasing sequence. If $v_{i+1}$ is a strong descent, then $v_{i+1}$ is an MSI by Lemma~\ref{descent}.  This implies $C(v_i,v_j)=v_{i+1}$.  If $v_{i+1}$ is not a strong descent then, by containment minimality, none of the descents are strong descents.  Also, our sequence is decreasing, so we conclude that $v_i=v_jw({l_j})\ldots w(l_{i+1})$.  Thus, $v_j$ is a prefix of the word $v_i$.  

Since $C(v_{i},v_{j})$ is an MSI of $C$ and $C[v_i,v_j]$ is the only maximal chain of $[v_{j},v_{i}]$ ending at the prefix embedding, there must be at least one more embedding of $v_j$ into $v_i$ .  Let $k$ be the largest index so that $v_k$ contains exactly two copies of $v_j$.  Then $v_{k+1}$ contains only one embedding of $v_j$, implying that $v_j$ is a suffix of $v_k$.  By the previous paragraph, $v_j$ is also a prefix of $v_k$.  Thus $o(v_k)=v_j$ because if a word longer than $v_j$ was $o(v_k)$, $v_{k+1}$ would have more than one copy of $v_j$. Similarly, $o(v_k)\not<i(v_k)$. So by Proposition~\ref{outerword}, $C(v_{k},v_{j})$ is an MSI of $C$.  Thus, by containment minimality, it must be the case that $k=i$.
\end{proof}

Theorem~\ref{MSIchar} completes the characterization of the MSIs in an interval $[u,w]$ in ordinary factor order.  Notice the definitions of the inner and outer word, which Bj\"orner used to state his formula, naturally arise when determining the MSIs.  Also, the inequality $u\leq o(w)\not\leq i(w)$ is forced upon us by the poset lexicographic ordering of the maximal chains.

We are now ready to prove Bj\"orner's formula using discrete Morse theory.  We have broken the proof up into several cases to make it easier to follow.

\begin{BT}
In ordinary factor order, if $u\leq w$ then
$$\mu(u,w)=\begin{cases}
\mu(u,o(w)) & \mbox{if } |w|-|u|>2 \mbox{ and } u\leq o(w)\not\leq i(w),\\
1 &\mbox{if } |w|-|u|=2\mbox{, } w \mbox{ is not flat, and } u=o(w) \mbox{ or } u=i(w),\\
(-1)^{|w|-|u|} &\mbox{if } |w|-|u|<2,\\
0 &\mbox{otherwise.}\end{cases}$$
\end{BT}

\begin{proof}
Let $[u,w]$ be an interval in ordinary factor order.  Suppose first that $|w|-|u|<2$. Then $u=w$ or $|u|=|w|-1$.  By the definition of the M\"obius function, we have $\mu(u,w)=1$ in the first case and $\mu(u,w)=-1$ in the second case.  Thus, the formula for $\mu(u,w)$ holds when $|w|-|u|<2$.

Now suppose $|w|-|u|=2$.  Then by the M\"obius recursion $\mu(u,w)=0$ if there is one element in the interval $(u,w)$ and $\mu(u,w)=1$ when there are 2 elements in the interval $(u,w)$.  Since $w$ covers at most two elements, these are the only possibilities.  If $w$ is flat, then $\mu(u,w)=0$ since $(u,w)$ contains a single element. If $w$ is not flat and $u=i(w)$, then removing either the first or last letter of $w$ gives us an element in $(u,w)$, implying $\mu(u,w)=1$.  If $w$ is not flat and $u=o(w)$, then removing either the first two letters or last two letters of $w$ gives us $u$.  Thus, $(u,w)$ has 2 elements implying $\mu(u,w)=1$.  If the above cases do not hold, then $u$ is either a prefix or a suffix of $w$, but not both.  In these cases, $(u,w)$ has 1 element implying $\mu(u,w)=0$.  Thus, the formula for $\mu(u,w)$ holds when $|w|-|u|=2$.  

We now turn to the case $|w|-|u|>2$.  We will use Theorem~\ref{chainmu} to calculate $\mu(u,w)$ from the critical chains in $[u,w]$.  By Lemma~\ref{ascent}, the chain id of a critical chain must be decreasing. Since a strong descent is followed by an ascent unless it is the last element in a chain, all the descents must be weak descents except possibly the last one.  Also, the only maximal chain in $[u,w]$ that could be critical is the one that is lexicographically last. Call this chain $C$.

Suppose first that $u\leq o(w)\nleq i(w)$.  We need to show that $o(w)$ is an element in the chain $C$. Let $k=|w|-|o(w)|$.  Since $v_k=o(w)$ is not contained in $i(w)$, the word $o(w)w(k+1)$ cannot be flat even if $o(w)$ is flat.  This observation, along with the fact that $u\leq o(w)$, allows us to conclude that $|w|,|w|-1,\ldots,|w|-(k-1)$ is a valid beginning for the chain id of a maximal chain $D$. Notice that each of these entries is the largest possible entry that does not already appear in the sequence.  Thus, any chain whose chain id differs from the chain id of $D$ in the first $k$ entries is lexicographically earlier than $D$.  So in the chain $C$, $l_1=|w|$, $l_2=|w|-1$,$\ldots$, $l_k=|w|-(k-1)$, and $v_k=o(w)$.

If $u=o(w)\nleq i(w)$, the previous paragraph implies that the sequence $l_1,\ldots,l_k$ is decreasing.  Thus, Theorem~\ref{MSIchar} implies $C(w,o(w))$ is the only interval in $J(C)$.  So by Theorem~\ref{chainmu}, $\mu(u,w)=1$.  Of course, in this case $\mu(u,o(w))=\mu(u,u)=1$ as well, so the formula holds.

Next we consider $u<o(w)\nleq i(w)$.  Since $l_k$ was the largest possible entry remaining, $l_{k+1}<l_{k}$, implying that $o(w)$ is a descent. Let $C'$ be the restriction of $C$ to the interval $[u,o(w)]$. We will show that $\#J(C)=2+\#J(C')$, allowing us to apply Theorem~\ref{chainmu} to complete the case $u\leq o(w)\nleq i(w)$.  Since the sequence $l_1,\ldots,l_{k}$ is decreasing, Theorem~\ref{MSIchar} implies $C(w,o(w))$ is the first interval in $J(C)$.   We claim $o(w)$ is the second interval in $J(C)$.  If $o(w)$ is a strong descent, this follows immediately.  If $o(w)$ is a weak descent, $v_{k+1}=w(1)w(2)\ldots w(|w|-k-1)<o(w)$, implying there are at least two copies of $v_{k+1}$ contained in $w$.  Let $j$ be the the largest value such that $v_j$ contains two copies of $v_{k+1}$.  Since in this case $v_1,\ldots,v_{j}$ are weak descents, $v_{k+1}\nleq i(v_j)$ because the two copies of $v_{k+1}$ in $v_j$ must be the prefix and suffix embeddings.  Furthermore, $o(v_j)=v_{k+1}$ because the prefix with one additional letter, $o(w)$, appears only once in $v_j$.  Since the sequence $l_{j+1},\ldots,l_{k+1}$ is decreasing, Theorem~\ref{MSIchar} implies $C(v_j,v_{k+1})$ is a skipped interval in $C$.  By the process of constructing $J(C)$ from $I(C)$, $o(w)=C(v_j,v_{k+1})-C(v_0,v_k)$ is the second MSI in $J(C)$, proving the claim. Since $o(w)$ is an MSI consisting of one element, all the remaining intervals in $J(C)$ are contained in the interval $(u,o(w))$.   Therefore, $J(C)=J(C')\cup\{C(w,o(w)),o(w)\}$ and $\#J(C)=2+\#J(C')$.  So by Theorem~\ref{chainmu}, $\mu(u,w)=\mu(u,o(w))$, proving the formula for $|w|-|u|>2$ and $u\leq o(w)\nleq i(w)$.

It remains to consider what happens when $|w|-|u|>2$ and $u\leq o(w)\nleq i(w)$ does not hold.  To show $\mu(u,w)=0$, we proceed by contradiction.  If $\mu(u,w)\neq0$ then, by Theorem~\ref{chainmu}, $J(C)$ must cover $C$. This implies that $J_1=C(v_0,v_j)$ is an MSI for some $v_j$.  Recall that Theorem~\ref{MSIchar} gives two possibilities for MSIs.  If $J_1=v_1$ and $v_1$ is a strong descent, then since $|w|-|u|>2$, $v_2$ is an ascent. This contradicts the fact that $C$ has a decreasing chain id.  Alternatively, we must have $v_j=o(w)\nleq i(w)$.  However, since $v_j\geq u$, $u\leq o(w)\nleq i(w)$, contradicting our assumption that this inequality does not hold. So $\mu(u,w)=0$, completing the proof. 
\end{proof}

By reproving Bj\"{o}rner's formula using this technique, it is easy to verify Bj\"orner's description of the homotopy type of a poset ordered by ordinary factor order.  The following result is a direct corollary of Lemma~\ref{ascent}.

\begin{Thm}
\label{homtype}
Let $[u,w]$ be an interval in $A^*$.  Then $\Delta(u,w)$ is homotopic to a sphere or is contractible
\end{Thm}
\begin{proof}
By Forman's fundamental theorem of discrete Morse Theory~\cite{rF95}, a simplicial complex with a discrete Morse function is homotopy equivalent to a CW complex with exactly one cell of dimension $d$ for each critical simplex of dimension $d$ (as well as a dimension $0$ cell).  By Babson and Hersh's theorem for poset lexicographic orders (Theorem~\ref{chainmu} and~\cite{BH05}), $\Delta(u,w)$ has a discrete Morse function in which a maximal chain is critical (contributes a critical simplex) if and only if $J(C)$ covers $C$.

Thus, if $[u,w]$ has no critical chains, it is homotopy equivalent to a CW complex with only the $0$-cell.  By Lemma~\ref{ascent}, a critical chain must have a decreasing chain id, which means only the lexicographically last chain can be critical.  So there can be at most one critical chain.  This gives us a CW complex with a $0$-cell and one other cell, which by Theorem~\ref{chainmu} has dimension $\#J(C)-1$.  The unique way to attach this cell to the $0$ cell is through a map which is constant on the boundary, resulting in a sphere of dimension $\#J(C)-1$.
\end{proof}

As a final note on the homotopy type, notice that a critical chain contains at most one MSI caused by a strong descent in $J(C)$.  Thus, the dimension of the sphere grows larger as the number of recursive calls to the formula (because of outerwords) increases.  For example, if $A=\{a,b\}$, $u=a$, and $w$ is an alternating word of $a$'s and $b$'s with $|w|=n>2$, the sphere has dimension $n-3$.

\section{Generalized Factor Order on the Positive Integers}
\label{gfo}

Let $P$ be a partially ordered set.  Partially order $P^*$ using \emph{generalized factor order}, which is the partial order on $P^*$ defined by letting $u\leq w$ if $w$ contains a subsequence $w(i+1)w(i+2)\ldots w(i+n)$ such that $u(j)\leq w(i+j)$ in $P$ for $1\leq j\leq n=|u|$.  If $u\leq w$, we will say that $u$ is a \emph{factor} of $w$.  Although we now have two partial orders to work with, the use of inequalities will never be ambiguous because the partial order induced by generalized factor order on words of one letter is the same as the order in $P$.  Notice that when $P$ is an antichain, generalized factor order and ordinary factor order are the same partial order.

\bfig{t}
\begin{tikzpicture}
[every node/.style={ellipse,
minimum size=5mm},
value/.style={yshift=-6.1mm,xshift=4.2mm}]

\matrix[shape=rectangle,
draw=none,row sep=5mm,column sep=4mm] {
% Row 1:
& & \node(1221){\tw{1221}}; & & \\
% Row 2:
\node(221){\tw{221}}; & \node(1121){\tw{1121}};& &\node(1211){\tw{1211}}; & \node(122){\tw{122}};; \\
% Row 3:
& & \node(121){\tw{121}};  & & \\
};
\node at (1221)[draw]{1221};
\node at (221)[draw]{221} edge (1221);
\node at (1121)[draw]{1121} edge (1221);
\node at (1211)[draw]{1211} edge (1221);
\node at (122)[draw]{122} edge (1221);
\node at (121)[draw]{121} edge (221) edge (1121) edge (1211) edge (122);

\node[value] at (121){\tbb{$1$}};
\node[value] at (221){\tbb{$-1$}};
\node[value] at (1121){\tbb{$-1$}};
\node[value] at (1211){\tbb{$-1$}};
\node[value] at (122){\tbb{$-1$}};
\node[value] at (1221){\tbb{$3$}};
\end{tikzpicture}

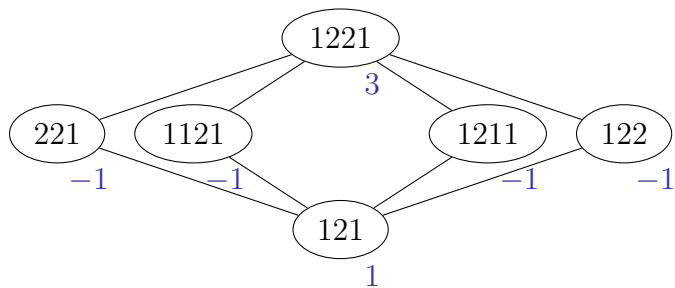
\caption{The interval $[121,1221]$. The M\"obius value $\mu(121,v)$ is given to the lower right of each word $v$.}
\label{fig:121,1221}
\efig

In this section, we will consider generalized factor order on the positive integers $\P$.  The Hasse diagram of the interval $[121,1221]$, ordered by generalized factor order on $\P^*$, is in Figure~\ref{fig:121,1221}.   To see that $121<1221$, let $w=1221$ and note that $1\leq w(1)$, $2\leq w(2)$, and $1\leq w(3)$.  In fact, $121$ appears twice in $1221$ because $1\leq w(2)$, $2\leq w(3)$, and $1\leq w(4)$.  This simple interval will useful in illustrating some of the ideas in this section.  The M\"obius values are listed for convenience - unlike ordinary factor order, we are able to have M\"obius values outside of the set $\{-1,0,1\}$.

Our definition of length for ordinary factor order on $A^*$ can be used in the context of generalized factor order on $\P^*$.  An expansion of a word $u$ is a word $\eta$ in $(\P\cup 0)^*$ satisfying $\eta=0^*u0^*$.  An \emph{embedding} of a word $u$ into $w$ is an expansion $\eta$ of $u$ satisfying $\eta\leq w$ for generalized factor order on $(\P\cup 0)^*$.  Thus, $0121$ and $1210$ are the embeddings of $121$ into $1221$.

A word is \emph{flat} in $\P^*$ if it is a sequence of $1$'s.  This is a natural refinement of the definition from ordinary factor order because in $\P$, $1$ is the unique minimal element. If we \emph{reduce} a letter $w(i)>1$, we are replacing it with $w(i)-1$.  To be precise, \emph{reducing} a letter $w(i)=1$ means removing it when considering words in $\P^*$, or replacing it with $0$ when considering embeddings in $(\P\cup 0)^*$.   We record the possible covering relations in $\P^*$ in a lemma for convenience - the proof is left to the reader.

\bfig{t}
\begin{tikzpicture}
[every node/.style={ellipse,
minimum size=5mm},
value/.style={yshift=-6mm,xshift=3mm}]

\matrix[shape=rectangle,
draw=none,row sep=5mm,column sep=4mm] {
% Row 1:
& & \node(2212){\tw{2212}}; && \\
% Row 2:
\node(1212){\tw{1212}}; &  & \node(2112){\tw{2112}}; &  & \node(2211){\tw{2211}}; \\
% Row 3:
\node(212){\tw{212}}; & \node(1112){\tw{1112}}; & \node(1211){\tw{1211}}; & \node(2111){\tw{2111}}; & \node(221){\tw{221}}; \\
% Row 4:
\node(112){\tw{112}}; & \node(211){\tw{211}}; &  & \node(121){\tw{121}}; & \node(22){\tw{22}}; \\
% Row 5:
 & \node(12){\tw{12}}; &  & \node(21){\tw{21}}; &  \\
% Row 6:
 &  & \node(2){\tw{2}}; &  &  \\
};
\node at (2212)[draw]{2212};

\node at (1212)[draw]{1212} edge (2212);
\node at (2112)[draw]{2112} edge (2212);
\node at (2211)[draw]{2211} edge (2212);

\node at (212)[draw]{212} edge (1212);
\node at (1112)[draw]{1112} edge (1212) edge (2112);
\node at (1211)[draw]{1211} edge (1212) edge (2211);
\node at (2111)[draw]{2111} edge (2112) edge (2211);
\node at (221)[draw]{221} edge (2211);

\node at (112)[draw]{112} edge (212) edge (1112);
\node at (211)[draw]{211} edge (212) edge (1211) edge (2111) edge (221);
\node at (121)[draw]{121} edge (1211) edge (221);
\node at (22)[draw]{22} edge (221);

\node at (12)[draw]{12} edge (112) edge (121) edge (22);
\node at (21)[draw]{21} edge (211) edge (121) edge (22);

\node at (2)[draw]{2} edge (12) edge (21);

\node[value] at (2){\tbb{$1$}};
\node[value] at (12){\tbb{$-1$}};
\node[value] at (21){\tbb{$-1$}};

\node[value] at (112){\tbb{$0$}};
\node[value] at (211){\tbb{$0$}};
\node[value] at (121){\tbb{$1$}};
\node[value] at (22){\tbb{$1$}};

\node[value] at (212){\tbb{$1$}};
\node[value] at (1112){\tbb{$0$}};
\node[value] at (1211){\tbb{$0$}};
\node[value] at (2111){\tbb{$0$}};
\node[value] at (221){\tbb{$-1$}};

\node[value] at (1212){\tbb{$-1$}};
\node[value] at (2112){\tbb{$1$}};
\node[value] at (2211){\tbb{$0$}};

\node[value] at (2212){\tbb{$-1$}};
\end{tikzpicture}

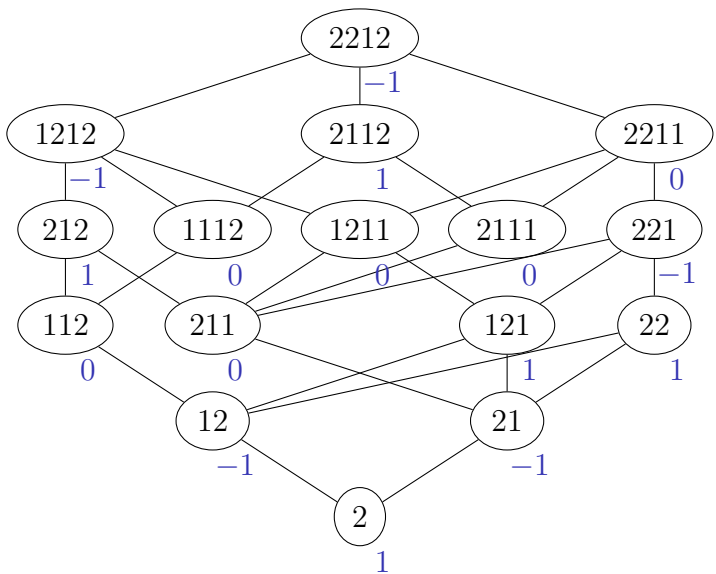
\caption{The interval $[2,2212]$. The M\"obius value $\mu(2,v)$ is given to the lower right of each word $v$.}
\label{fig:2,2212}
\efig

\begin{Lem}
\label{gcover}
A word $w=w(1)\ldots w(n)$ in $\P^*$ can cover up to $n$ words, each formed by reducing a letter in $w$ by $1$.  Reducing the letters $w(1)$ and $w(n)$ will always produce a factor, while reducing $w(i)$ for $1< i< n$ can only produce a new factor if $w(i)\geq2$. These words are distinct unless $w$ is flat, in which case $w$ only covers one word which is flat. \qed
\end{Lem}

It should be noted that whenever a word begins or ends with a sequence of $1$'s, we could produce the same factor by reducing any 1 from the sequence.  However, since such a word is always a prefix or suffix, to assure our embeddings respect generalized factor order on $(\P\cup 0)^*$, $1$'s can only be reduced if they are at the beginning or end of a word.  If a word is flat, our convention is that only the first $1$ can be reduced.  For ease, we will say the letter $w(i)$ is \emph{reducible} in the word $w$ if $w(i)>1$ or $w(i)=1$ and its position $i$ is consistent with the preceding discussion. Similarly, we say the letter $\eta(i)$ in the embedding of a word $v$ into $w$ is \emph{reducible} if $\eta(i)>1$, if $\eta(i)=1$ is the first non-zero letter, or if $v$ is not flat and $\eta(i)=1$ is the last non-zero letter.

For example, the word $1211$, found in the center of Figure~\ref{fig:2,2212}, has three reducible positions: $1$, $2$ and $4$.  Notice reducing position $2$ results in the word $1111$, which is not in the interval $[2,2212]$.  Also, in a flat word, such as $1111$, only the first position is reducible.  Similarly, if a flat word is given as an embedding, such as $0110$ in $2212$, then only the first non-zero position, in this case position 2, is reducible.

Let $[u,w]$ be an interval in $\P^*$.  Let $C:w=v_0\stackrel{l_1}{\rightarrow}v_1\stackrel{l_2}{\rightarrow}\ldots\stackrel{l_{n-1}}{\rightarrow} v_{n-1}\stackrel{l_n}{\rightarrow}v_n=u$ be a maximal chain in $[u,w]$, where the $l_i$ are defined by the corresponding sequence of embeddings $\eta_{v_i}$ in the sense that
$$\eta_{v_i}(l_i)=\eta_{v_{i-1}}(l_i)-1 \text{ and } \eta_{v_i}(j)=\eta_{v_{i-1}}(j) \text{ when } j\neq l_i.$$
This gives each maximal chain $C$ a chain id $l_1\ldots l_n$.  By lexicographically ordering these chain ids, we produce a poset lexicographic order on the maximal chains of $[u,w]$.  We will use this order to find the MSIs of the maximal chains.  Examples of this poset lexicographic order and MSIs are given in Tables~\ref{table:121,1221} and~\ref{table:2,2212}.

Suppose $|u|=|w|$.  Let $m_i=w(i)-u(i)$.  By Lemma~\ref{gcover}, every permutation of the multiset $M_{uw}=\{1^{m_1},2^{m_2},\ldots\}$ is the chain id for a maximal chain in $[u,w]$.  Since there is a single embedding of $u$ into $w$, these permutations account for every maximal chain in $[u,w]$.  This implies the same-length case is really a direct product of chains.  If $[n]$ is the poset consisting of the integers $1,\ldots, n$, partially ordered by size, a direct product of chains is the well-known poset defined by $P=[n_1]\times [n_2]\times \ldots \times [n_m]$, in which $(i_1,i_2,\ldots,i_m)\leq (j_1,j_2,\ldots,j_m)$ if $i_k\leq j_k$ for all $1\leq k\leq m$.  We record some relevant results about this poset here for later use.  Recall that the \emph{rank function} of any poset $P$ in which every maximal chain has the same length is a map $\rho$ from the elements of $P$ to the non-negative integers.  It is recursively defined by setting $\rho(x)=0$ if $x$ is minimal, and $\rho(y)=\rho(x)+1$ if $y$ covers $x$.

\btab{t}
\begin{longtable}[!h]{rcccccccl}
Chain Id & $v_0$ & $l_1$ & $v_1$  & $l_2$ & $v_2$ \\
1-2  & $1221$ & $1$  & $0221$ & $2$  & $0121$ \\
2[-]1 & $1221$ & $2$ & $[1121]$ & $1$ & $0121$ \\
3[-]4 & $1221$ & $3$ & $[1211]$ & $4$ & $1210$ \\
4[-]3 & $1221$ & $4$ & $[1220]$ & $3$ & $1210$ \\
\end{longtable}
\caption{The MSIs of the maximal chains of $[121,1221]$ in $\P^*$.}
\label{table:121,1221}
\etab
\vskip0\baselineskip

\begin{Prop}
\label{samelength}
If $[u,w]\subset\P^*$ with $|u|=|w|$ and $C:w=v_0\stackrel{l_1}{\rightarrow}v_1\stackrel{l_2}{\rightarrow}\ldots\stackrel{l_{n-1}}{\rightarrow} v_{n-1}\stackrel{l_n}{\rightarrow}v_n=u$ is a maximal chain in $[u,w]$, then each descent $v_i$ is an MSI of $C$ and this accounts for all MSIs of $C$.
\end{Prop}

\begin{Cor}
\label{musame}
Let $[u,w]\subset\P^*$ with $|u|=|w|$. Then $$\mu(u,w)=\begin{cases}
{(-1)^{\rho(u,w)}} & \mbox{if } w(i)-u(i)\leq1 \mbox{ for all } 1\leq i\leq|w|,\\
0 &\mbox{otherwise,}\end{cases}$$ where $\rho$ denotes the rank function in $\P^*$ and $\rho(u,w)=\rho(w)-\rho(u)$.
\end{Cor}

\begin{Cor}
\label{msisame}
Let $[u,w]\subset\P^*$ and $C:w=v_0\stackrel{l_1}{\rightarrow}v_1\stackrel{l_2}{\rightarrow}\ldots\stackrel{l_{n-1}}{\rightarrow} v_{n-1}\stackrel{l_n}{\rightarrow}v_n=u$ be a maximal chain in $[u,w]$.  If $|v_i|=|v_j|$ and $v_k$ is a descent with $i<k<j$, then $v_k$ is an MSI in $C$.\qed
\end{Cor}

Let $C:w=v_0\stackrel{l_1}{\rightarrow}v_1\stackrel{l_2}{\rightarrow} \ldots\stackrel{l_{n-1}}{\rightarrow} v_{n-1}\stackrel{l_n}{\rightarrow} v_n=u$ be an arbitrary maximal chain in an interval $[u,w]$ of $\P^*$.  We can determine precisely which sequences $l_1\ldots l_n$ correspond to chain ids by considering what conditions imply the letters $\eta_{v_{i}}(l_{i+1})$ are reducible.

Let $\eta$ be an embedding of $u$ into $w$.  Let $m_i=w(i)-\eta(i)$.  Let $f$ denote the position of the first non-zero letter in the embedding $\eta$, and $\ell$ denote the position of the last non-zero letter in the embedding $\eta$. We say a permutation of the multiset $M_{\eta}=\{1^{m_1},2^{m_2},\ldots\}$ is \emph{admissible} if the last $i$ appears before the last $i+1$ for all $1\leq i\leq f-2$ and the last $j$ before the last $j-1$ for all $\ell+2\leq j\leq |w|$.  An admissible permutation is \emph{strongly admissible} if the last $\ell+1$ appears before either one value from the set $\{f,f+1,\ldots, \ell\}$, or $2$ copies of a value from the set $\{1,\dots, f-1\}$.

\btab{t}
 I(C) for $[2,2212]$

\begin{longtable}[!h]{rcccccccccccccl}
Chain Id & $v_0$ & $l_1$ & $v_1$  & $l_2$ & $v_2$  & $l_3$ & $v_3$  & $l_4$ & $v_4$  & $l_5$ & $v_5$ \\
1-1-2-2-3  & $2212$ & $1$  & $1212$ & $1$  & $0212$ & $2$  & $0112$ & $2$  & $0012$ & $3$  & $0002$ \\
1-1-4[-4-]3 & $2212$ & $1$ & $1212$ & $1$ & $0212$ & $4$ & $[0211$ & $4$ & $0210]$ & $3$ & $0200$ \\
1-2[-]1-2-3 & $2212$ & $1$ & $1212$ & $2$ & $[1112]$ & $1$ & $0112$ & $2$ & $0012$ & $3$ & $0002$ \\
1-4[-]1-4-3 & $2212$ & $1$ & $1212$ & $4$ & $[1211]$ & $1$ & $0211$ & $4$ & $0210$ & $3$ & $0200$ \\
1-4-4[-]1-3 & $2212$ & $1$ & $1212$ & $4$ & $1211$ & $4$ & $[1210]$ & $1$ & $0210$ & $3$ & $0200$ \\
1-4[-4-]3[-]1 & $2212$ & $1$ & $1212$ & $4$ & $[1211$ & $4$ & $1210]$ & $3$ & $[1200]$ & $1$ & $0200$ \\
2[-]1-1-2-3 & $2212$ & $2$ & $[2112]$ & $1$ & $1112$ & $1$ & $0112$ & $2$ & $0012$ & $3$ & $0002$ \\
2[-4[-]4-3-]2 & $2212$ & $2$ & $[2112$ & $4$ & $[2111]$ & $4$ & $2110$ & $3$ & $2100]$ & $2$ & $2000$ \\
4[-]1-1-4-3 & $2212$ & $4$ & $[2211]$ & $1$ & $1211$ & $1$ & $0211$ & $4$ & $0210$ & $3$ & $0200$ \\
4[-]1-4[-]1-3 & $2212$ & $4$ & $[2211]$ & $1$ & $1211$ & $4$ & $[1210]$ & $1$ & $0210$ & $3$ & $0200$ \\
4[-]1-4-3[-]1 & $2212$ & $4$ & $[2211]$ & $1$ & $1211$ & $4$ & $1210$ & $3$ & $[1200]$ & $1$ & $0200$ \\
4[-]2[-]4-3-2 & $2212$ & $4$ & $[2211]$ & $2$ & $[2111]$ & $4$ & $2110$ & $3$ & $2100$ & $2$ & $2000$ \\
4-4[-]1-1-3 & $2212$ & $4$ & $2211$ & $4$ & $[2210]$ & $1$ & $1210$ & $1$ & $0210$ & $3$ & $0200$ \\
4-4[-]1-3[-]1 & $2212$ & $4$ & $2211$ & $4$ & $[2210]$ & $1$ & $1210$ & $3$ & $[1200]$ & $1$ & $0200$ \\
4-4[-]2[-]3-2 & $2212$ & $4$ & $2211$ & $4$ & $[2210]$ & $2$ & $[2110]$ & $3$ & $2100$ & $2$ & $2000$ \\
4-4-3[-]1-1 & $2212$ & $4$ & $2211$ & $4$ & $2210$ & $3$ & $[2200]$ & $1$ & $1200$ & $1$ & $0200$ \\
4-4-3[-]2[-]2 & $2212$ & $4$ & $2211$ & $4$ & $2210$ & $3$ & $[2200]$ & $2$ & $[2100]$ & $2$ & $2000$ \\
\end{longtable}
$J(C)$ intervals for  $[a,abbabb]$

The set $I(C)$ only changes for the chain $2-4-4-3-2$:

\begin{longtable}[!h]{rcccccccccccccl}
Chain Id & $v_0$ & $l_1$ & $v_1$  & $l_2$ & $v_2$  & $l_3$ & $v_3$  & $l_4$ & $v_4$  & $l_5$ & $v_5$ \\
2[-4-]4[-3-]2 & $2212$ & $2$ & $[2112$ & $4$ & $2111]$ & $4$ & $[2110$ & $3$ & $2100]$ & $2$ & $2000$ \\
\end{longtable}

\caption{The MSIs of the maximal chains of $[2,2212]$ in $\P^*$.}
\label{table:2,2212}
\etab

Examples of these definitions can be found in Table~\ref{table:2,2212}.  For an admissible permutation, we consider the embedding $\eta=0200$.  Here, $M_\eta=\{1^2,3^1,4^2\}$, and $f=\ell=2$.  So as long as the $3$ appears after the last $4$, as in the sequence $1-4-4-1-3$, we get an admissible permutation.  For a strongly admissible permutation, we consider the embedding $\eta=00110$ in $22122$.  Here, $M_\eta=\{1^2,2^2,4^1,5^2\}$, $f=3$, and $\ell=4$.  So we need the last $2$ to occur after the last $1$, and the last $5$ to appear before either the last $4$, two $1$'s, or two $2$'s.  Examples include  $2-4-5-5-1-1-2$ and $1-1-2-2-5-5-4$.

\begin{Prop}
\label{permchar}
Let $\eta$ be an embedding of $u$ into $w$, $m_i=w(i)-\eta(i)$, and $M_{\eta}=\{1^{m_1},2^{m_2},\ldots\}$.

If $u$ is not flat, then a sequence of numbers is the chain id for a maximal chain in $[u,w]$ ending at $\eta$ if and only if it is an admissible permutation of the multiset $M_{\eta}$.

If $u$ is flat, then a sequence of numbers is the chain id for a maximal chain in $[u,w]$ ending at $\eta$ if and only if it is a strongly admissible permutation of the multiset $M_{\eta}$.
\end{Prop}

\begin{proof}
For a maximal chain in the interval $[u,w]$ to finish at the embedding $\eta$, its chain id must be a permutation of $M_{\eta}$.  The sequence $l_1,\ldots,l_n$ is a chain id in $[u,w]$ if and only if every word in the corresponding sequence $v_1,\ldots v_n$ is a factor of $w$ that contains $u$ as a factor.

If $v_i$ is not flat, $1$'s can only be reduced at the beginning or end of the word.  Thus, when $u$ is not flat, each $v_i$ will be a factor of $w$ that contains $u$ in the embedding corresponding to $\eta$ if and only if the positions in $\eta$ smaller than $f$ are reduced to $0$ from left to right, and the positions greater than $\ell$ are reduced to $0$ from right to left.  The permutations in $M_{\eta}$ that satisfy this requirement are the admissible permutations.

If $v_i$ is flat, only the first position of the word can be reduced.  Thus, when $u$ is flat, each $v_i$ will be a factor of $w$ that contains $u$ in the embedding corresponding to $\eta$ if and only if the positions in $\eta$ smaller than $f$ are reduced to $0$ from left to right, the positions greater than $\ell$ are reduced to $0$ from right to left, and  all positions greater than $\ell$ are reduced to $0$ before the last $2$ in the word is reduced. The permutations in $M_{\eta}$ that satisfy this requirement are the strongly admissible permutations.  
\end{proof}

Now that we know which permutations of a chain id produce maximal chains, we are ready to consider the MSIs.  We begin by looking at the maximal chains in $[121,1221]$ and $[2,2212]$.  

In $[121,1221]$, whose maximal chains are in Table~\ref{table:121,1221}, descents $v_i$ satisfying $l_{i+1}=l_i-1$ are MSIs.  Also, there is an MSI containing an ascent, which is not possible in the ordinary factor order (antichain) case.

In $[2,2212]$, whose maximal chains are in Table~\ref{table:2,2212}, we see that every descent that does not remove two $1$'s from the back of a word consecutively is an MSI.  Furthermore, for every MSI $C(v_i,v_j)$ of length greater than $1$, the embedding of $v_j$ into $v_i$ is not found in any previous chain.  This is consistent with the antichain case, in which MSIs consist of descents satisfying $l_{i+1}<l_i-1$ or weakly decreasing sequences starting at $v$ and ending at $o(v)$.

Our first goal is to determine when a single descent creates an MSI.  Suppose $v_i$ is a descent of a chain $C$. Notice that whenever interchanging $l_{i+1}$ and $l_i$ in the label sequence of $C$ produces another maximal chain in $[u,w]$, the new chain is lexicographically earlier than $C$.  This implies $v_i$ is an MSI of $C$.  We will invoke this line of reasoning by saying ``$l_{i+1}$ and $l_i$ can be interchanged.''

To simplify the language, we will say a descent $v_i$ is a \emph{strong descent} if $v_{i-1}\neq v_{i+1}11$ and we will say $v_i$ is a \emph{weak descent} if $v_{i-1}= v_{i+1}11$.

\begin{Prop}
\label{gdescent}
Suppose $[u,w]\subset\P^*$ and $C:w=v_0\stackrel{l_1}{\rightarrow}v_1\stackrel{l_2}{\rightarrow}\ldots\stackrel{l_{n-1}}{\rightarrow} v_{n-1}\stackrel{l_n}{\rightarrow}v_n=u$ is a maximal chain in $[u,w]$. Suppose $v_i$ is a descent.  Then $v_i$ is a length 1 MSI if and only if $v_i$ is strong.
\end{Prop}

\begin{proof}
We will prove this proposition by considering the difference in length between $v_{i-1}$ and $v_{i+1}$.

If $|v_{i-1}|=|v_{i+1}|$,  $v_i$ is an MSI by Corollary~\ref{msisame}.

Suppose $|v_{i-1}|=|v_{i+1}|+1$.  Then either $\eta_{v_{i-1}}(l_{i+1})=1$ and $l_{i+1}$ corresponds to the first letter of $v_{i-1}$ or $\eta_{v_{i-1}}(l_{i})=1$ and $l_i$ corresponds to the last letter of $v_{i-1}$.  In the first case, by Proposition~\ref{permchar}, $l_{i+1}$ and $l_i$ can be interchanged.  In the second case, when $v_{i+1}$ is not flat, Proposition~\ref{permchar} implies $l_{i+1}$ and $l_i$ can be interchanged. If $v_{i+1}$ is flat,  $\eta_{v_{i-1}}(l_{i+1})=2$ because otherwise $v_{i-1}$ would be flat and $l_i$ could not correspond to the last letter of $v_{i-1}$.  Therefore, if $v$ is a flat sequence of $1$'s with length $|v_{i-1}|$, then the chain $$D:w=v_0\stackrel{l_1}{\rightarrow}\ldots\stackrel{l_{i-1}}{\rightarrow} v_{i-1}\stackrel{l_{i+1}}{\rightarrow}v\stackrel{l_{i+2}}{\rightarrow}v_{i+1}\stackrel{l_{i+3}}{\rightarrow}\ldots\stackrel{l_{n}}{\rightarrow}v_{n-1}\stackrel{l_{n}+1}{\rightarrow}v_n=u$$
is a lexicographically earlier chain than $C$ in $[u,w]$.  Since $C-v_i\subset D$, $v_i$ is a skipped interval of $C$, implying it is an MSI of $C$.

Suppose $|v_{i-1}|=|v_{i+1}|+2$.  Since $v_i$ is a descent, our restrictions on reducing $1$'s imply $v_{i+1}$ is not flat and either $v_{i-1}= 1v_{i+1}1$ or $v_{i-1}= v_{i+1}11$.  In the first case, by Proposition~\ref{permchar}, $l_{i+1}$ and $l_i$ can be interchanged.  In the second case, $v_i$ cannot be a length $1$ MSI because there is a unique maximal chain in the interval $[v_{i+1},v_{i-1}].$
\end{proof}

In the proof of the above result, we showed that unless the word $v_{i+1}$ is flat, all descents $v_i$ which are MSIs have $v_{i+1}$ in the same embedding into $v_{i-1}$ as a previous chain.  So suppose $C(v_i,v_j)$ is an MSI of the chain $C$, and $D<C$ is a chain satisfying $C-C(v_i,v_j)=D-D(v_i,v_j)$.  Then if there is a chain $D<C$ satisfying $C-C(v_i,v_j)=D-D(v_i,v_j)$ such that embedding of $v_j$ into $v_i$ in $D$ is the same as $C$, we should have an MSI which is a strong descent.  But, if every chain $D<C$ satisfying $C-C(v_i,v_j)=D-D(v_i,v_j)$ has an embedding of $v_j$ into $v_i$ different than $C$, unless $v_j$ is flat, we should have a different type of MSI.  As in the antichain case, these are the only two types of MSIs in generalized factor order on $\P^*$.

To describe the second type of MSI, we need to develop the notion of a ``principal factor.''  If $|u|\leq|w|$ and $u(i)\leq w(i)$ for $1\leq i\leq|u|$, we call $u$ a {\emph{prefix}} of $w$.  A {\emph {suffix}} of $w$ is defined analogously.  If $|u|<|w|$, a prefix or suffix is \emph{proper}.  If $u$ is both a proper prefix and a proper suffix of $w$, we say it is an \emph{outer factor} of $w$.  To simplify the language, we will call an outer factor of $w$ not contained in a longer outer factor a \emph{maximal outer factor of $w$}.

Using these definitions, we see that $21$ is a prefix of $2212$, $11$ is a suffix of $2212$, and that $211$ and $111$ are maximal outer factors of $2212$.

From this point forward, if $u$ is a prefix of $w$, we will often be dealing with the corresponding embedding.  If this is the case, will abuse notation and write $u(i)=0$ in place of introducing $\eta$ and writing $\eta(i)=0$.  As an example, if $u=22$ and $w=2212$, we may write $w(3)>u(3)$, assuming the third position of $u$ is $0$.

Let $p$ be a maximal outer factor of $w$.  Suppose $p$ is not flat.  Then the \emph{principal index} $i$ of $p$ in $w$ is the smallest index such that $w(i)> p(i)$ and $w(i)$ is reducible.  We say $p$ is a \emph{principal factor} of $w$ if the word produced by reducing $w(i)$ by $1$ no longer contains $p$ as a suffix.  

Intuitively, the principal index $i$ of an outer factor $p$ is the first position of $w$ that can be reduced without removing the prefix embedding of $p$. The factor becomes a principal factor if reducing $i$ removes the suffix embedding of $p$ from $w$.  Thus, the principal index of a principal factor satisfies $i>1$.  Also, since $p$ is not flat, $w(i)>1$ because the letters in the suffix embedding of $p$ greater than $1$ necessarily occur later in $w$ than the corresponding letters in the prefix embedding of $p$.

For our first examples, we consider the principal factors in the intervals $[121,1221]$ and $[2,2212]$ given in Tables~\ref{table:121,1221} and~\ref{table:2,2212}.  Note $121$ is a prefix of $1221$, and its principal index is $3$.  Our results below will show the MSI in the chain $3-4$ results from the fact that $121$ is a principal factor of $1221$.  For the second example, the only principal factor of $2212$ is $211$, as the flat maximal outer factor $111$ is excluded from the definition.  As for the other words in the interval $[2,2212]$, $1212$ has $12$ as a principal factor, $2112$ has $2$ as a principal factor, $2211$ has $211$ as a principal factor, $212$ has $2$ as a principal factor, and $221$ has $21$ as a principal factor.  Note that when one considers the entire set of maximal chains of $[2,2212]$, each of these principal factors immediately follows exactly one MSI not caused by a strong descent.

Some additional examples may help to further clarify the definition. The principal factors of $12222$ are $1211$, $1212$, $1221$, and $1222$.  The principal factors of $33133$ are $3111$, $3112$, $3113$ and $33$.  The words $3121$ and $11211$ have no principal factors.

It is important to note that a principal factor has exactly two embeddings in $w$.  If there were a third embedding of a principal factor $p$, we could extend it by a sequence of $1$'s to create a suffix of $w$.  This new suffix would also be a prefix since $1$ is the minimum of $\P$, implying that $p$ would be contained in a longer outer factor.

Using principal factors, we can identify the second type of MSI in generalized factor order on $\P$.

\begin{Prop}
\label{gprincipal}
Suppose $u$ is a principal factor of $w$ with principal index $i$.  Let $C:w=v_0\stackrel{i}{\rightarrow}v_1\stackrel{l_2}{\rightarrow}\ldots\stackrel{l_{n-1}}{\rightarrow} v_{n-1}\stackrel{l_n}{\rightarrow}v_n=u$ be the lexicographically first chain in $[u,w]$ with $l_1=i$.  Then $C(u,w)$ is an MSI of $C$.
\end{Prop}

\begin{proof}
From the definition of principal factor, we conclude the word $v_1$ contains only the prefix embedding of $u$.  By the definition of outer factor, $w$ contains two embeddings of $u$, the prefix embedding and the suffix embedding. So by Proposition~\ref{permchar}, there is a maximal chain in $[u,w]$ with a chain id beginning with $1$.  Since the principal index $i$ of $u$ is greater than $1$, there exist chains that are lexicographically earlier than $C$.

Let $C'$ be an arbitrary maximal chain that is lexicographically earlier than $C$.  Then $l'_1<i$ because $C$ is the lexicographically first chain in $[u,w]$ with $l_1=i$.  From the definition of principal index, we conclude that the word $v'_1$ in the chain $C'$ does not contain the prefix embedding of $u$.  So $v'_1$ must contain the suffix embedding of $u$.  Furthermore, since $C$ is the lexicographically first chain with the prefix embedding of $u$, we reduce the letters at the end last, implying $v_{n-1}=u1$.  Since $v_1'$ does not contain $u1$, the only words common to $C$ and $C'$ are $w$ and $u$.   Since $C'$ was an arbitrary maximal chain in $[u,w]$ with $l'_1<i$, and $C$ is the first maximal chain in $[u,w]$ with $l_1=i$, we conclude $C(u,w)$ is an MSI.
\end{proof}

Note that this proposition shows that since $121$ is a principal factor of $1221$, the chain in the interval $[121,1221]$ with chain id $3-4$ has $C(1221,121)$ as an MSI (see Table~\ref{table:121,1221}).  Also, since $211$ is a principal factor of $2212$, the chain in the interval $[2,2212]$ with chain id $2-4-4-3-2$ has $C(211,2212)$ as an MSI (see Table~\ref{table:2,2212}).  As a last example, since $21$ is a principal factor of $221$, the chain in the interval $[2,2212]$ with chain id $4-4-2-3-2$ has $C(21,221)$ as an MSI. 

To complete the characterization of the MSIs, we will need a precise description of the lexicographically first chain in an interval $[u,w]$ that contains an embedding $\eta$ of $u$.  The chain id of this chain, $C_{\eta}$, is the lexicographically first permutation of $M_{\eta}$ that is the chain id of a maximal chain.  Using Proposition~\ref{permchar}, we will describe the structure of $C_{\eta}$ when $u$ is not flat.  First, it reduces all the letters before the support of the embedding to $0$ in order from left to right.  Next, it reduces all the letters in the support of the embedding down to the corresponding $u$-value in order from left to right. In the third step, $C_{\eta}$ reduces all letters beyond the support of the embedding to $1$ from left to right.  Finally, once we reach the end of $w$,  all the $1$'s beyond the support of the embedding are reduced to $0$ from right to left.

For example, the first admissible chain of $[121,1221]$ ending at the prefix embedding has chain id $3-4$.  The first admissible chain of $[2,2212]$ ending at the prefix embedding has chain id $2-4-4-3-2$.

If $u$ is not flat, call $C_{\eta}$ the \emph{first admissible} chain ending at $\eta$.  Recall that a sequence is \emph{unimodal} if it consists of a weakly increasing sequence followed by a weakly decreasing sequence.  So a sequence $l_1\ldots l_n$ is unimodal if $l_1\leq\ldots\leq l_i\geq\ldots\geq l_n$ for some index $i$.  This discussion implies the following lemma.

\begin{Lem}
\label{flc}
Suppose $[u,w]\subset\P^*$ and $\eta$ is an embedding of $u$ into $w$.  Let $\ell$ be the index of the largest non-zero number in $\eta$.  If $u$ is not flat, $C_{\eta}$ has as its chain id the unique unimodal permutation of $M_{\eta}$ with decreasing suffix $|w|, |w|-1, \ldots, \ell+1$.\qed
\end{Lem}

To make better use of the Lemma, we introduce the notation $1^m$ to represent a sequence of $m$ $1$'s.  For example, if considering the prefix embedding of $u$ into $w$, then $u1^m$, where $m=|w|-|u|$, is the word in $C_{u}$ which precedes the decreasing suffix $|w|, |w|-1, \ldots, \ell+1$.

The following Theorem completes the characterization of the MSIs.

\begin{Thm}
\label{gMSIchar}
Suppose $[u,w]\subset\P^*$ and $C:w=v_0\stackrel{l_1}{\rightarrow}v_1\stackrel{l_2}{\rightarrow}\ldots\stackrel{l_{n-1}}{\rightarrow} v_{n-1}\stackrel{l_n}{\rightarrow}v_n=u$ is a maximal chain in $[u,w]$.  Then $C(v_i,v_j)$ is an MSI of $C$ if and only if $C(v_i,v_j)$ consists of a single strong descent, or $v_{j}$ is a principal factor of $v_i$, $l_{i+1}$ is the principal index of $v_j$ with respect to the embedding $\eta_{v_j}$ of $v_j$ into $w$,  and $C[v_i,v_j]$ is the first admissible chain in $[v_j,v_i]$ ending at the prefix-embedding of $v_j$.
\end{Thm}

\begin{proof}
The reverse implication follows from Propositions~\ref{gdescent} and~\ref{gprincipal}.

Suppose $C(v_i,v_j)$ is an MSI of $C$.  By the definition of a poset lexicographic order, $C(v_i,v_j)$ is an MSI of $C$ if and only if it is an MSI of $C[v_i,v_j]$ in the interval $[v_j,v_i]$. Thus, it suffices to consider the case $w=v_i$ and $u=v_j$. Corollary~\ref{msisame} implies that when $|u|=|w|$, any MSI consists of a single descent.  Proposition~\ref{gdescent} states that if $C(w,u)$ is an MSI and it consists of a single descent, then $w\neq u11$.  

To finish the proof, it suffices to consider the case when $|u|<|w|$ and $C(w,u)$ is an MSI of $C$ that does not consist of a strong descent.  Our first goals are to establish that $w$ has a prefix embedding of $u$ and   $C[w,u]$ is the first admissible chain ending at the prefix embedding. Let $\eta$ be the embedding of $u$ into $w$ at the end of the chain $C$.  Note that any descent $v_k$ in $C$ is a weak descent satisfying $v_{k-1}=v_{k}1=v_{k+1}11$ since otherwise, by Proposition~\ref{gdescent}, $C(v_{k-1},v_{k+1})$ would be an MSI, contradicting the minimality of $C(w,u)$.  If $k+1\neq n$, this forces $l_{k+2}<l_{k+1}=l_k-1$,  implying that $v_{k+1}$ is also a weak descent.  By continually applying this idea, we find that any descents contained in $C(w,u)$ occur in a single sequence of weak descents at the end of this interval, and the corresponding labels form a decreasing sequence of consecutive numbers.  To see $\eta$ is the prefix embedding of $u$ into $w$, suppose for a contradiction that the set $M_{\eta}$ contains a $1$.  Since any descents occur in a sequence at the end of the interval, it follows that $l_1=1$ or $l_n=1$.  If $l_n=1$, then $v_{n-1}$ is descent satisfying $u11=v_{n-2}$, implying $u$ is the empty word.  However, this contradicts Proposition~\ref{permchar} because $v_{n-1}$ is obtained by trimming a $1$ from the back of a flat word.  Suppose $l_1=1$.  Then every chain $C'$ lexicographically earlier than $C$ has $l'_1=1$ and thus contains $v_1$, contradicting the fact that $C(w,u)$ is an MSI.  Therefore, $\eta$ is the prefix embedding of $u$ into $w$, allowing us to write $\eta=u$.  Since $|u|<|w|$, we must reduce position $|w|$ by the end of the chain. Note that the label $|w|$ can only be followed by another $|w|$ or the sequence of labels $|w|-1,\ldots,|u|+1$, which leads to the sequence of weak descents. It follows that $C$ must contain the word $u1^m$, where $m=|w|-|u|$.  This implies $u$ is not flat.  So by Lemma~\ref{flc}, $C$ is the first admissible permutation of $M_{u}$.  Thus, we have shown $C$ is the first admissible chain ending at the prefix-embedding of $u$.  

Next, we will show $u$ is an outer factor of $w$ by showing it is a suffix of $w$.   Recall that $C$ is not the lexicographically first chain in $[u,w]$ because $C(w,u)$ is an MSI.  Let $C'$ be the lexicographically first chain in the interval $[u,w]$.  Lemma~\ref{flc} implies $C'$ contains the word $v_{n-1}=u1$ unless it ends at the suffix embedding of $u$ into $w$.  However, $C'$ cannot contain $v_{n-1}$ as otherwise $C-C(w,v_{n-1})\subset C'$, contradicting the assumption that $C(w,u)$ is an MSI.  Therefore, $C'$ must end at the suffix embedding of $u$ into $w$.  Thus, $u$ is an outer factor of $w$ which is not flat.

To establish that $u$ is a principal factor of $w$, it remains to show that $u$ is not contained in a longer outer factor, that $l_1$ satisfies the definition of a principal index, and that reducing $w(l_1)$ removes the suffix embedding.

Suppose $u$ were contained in a longer outer factor of $w$.  Then $u$ would be a prefix of this larger factor, which is a suffix of $w$.  Therefore, there exists an outer factor $v$ of $w$ such that $v=u1^m$, where $m=|v|-|u|$.  By Lemma~\ref{flc}, $C$ must contain this word.  Since $C$ is the first admissible chain ending at the prefix-embedding of $u$,  $C[w,v]$ is the first admissible chain ending at the prefix-embedding of $v$.  Let $D$ be the lexicographically first chain in the interval $[v,w]$.  Since $v$ contains $u$, it is not flat.  So the structure of $D$ is determined by Lemma~\ref{flc}, which implies $D$ ends at the embedding of $v$ into $w$ that is farthest to the right.  So $D$ ends at the suffix embedding of $v$ into $w$.  Let $\hat{C}=D[w,v]\cup C[v,u]$.  Then $\hat{C}$ is a maximal chain in $[u,w]$ lexicographically earlier than $C$. But $C-C(w,v)\subset \hat{C}$, contradicting the fact that $C(w,u)$ is an MSI.  So $u$ is not contained in a longer outer factor of $w$.

Lemma~\ref{flc} implies that $l_1$ is the first index such that $w(l_1)>u(l_1)$ and $w(l_1)$ is reducible.  Suppose $v_{1}$, the word produced by reducing $w(l_1)$ by 1, contains the suffix embedding of $u$.
Then $u$ is an outer factor of $v_1$ and $|v_1|=|w|$.  By Lemma~\ref{flc}, $C[v_{1},u]$ is the first admissible chain ending at the prefix-embedding of $u$.  Since $v_{1}$ also contains the suffix embedding of $u$, Lemma~\ref{flc} implies that $C[v_{1},u]$ is not the lexicographically first chain in $[u,v_{1}]$. Let $D$ be the lexicographically first chain in the interval $[u,v_{1}]$, and $\hat{C}=w\cup D$.  Then $\hat{C}$ is a maximal chain in $[u,w]$ lexicographically earlier than $C$. But $C-C(v_1,u)\subset \hat{C}$, contradicting the fact that $C(w,u)$ is an MSI.  Therefore, $v_{1}$ does not contain the suffix embedding of $u$, and we have established that $u$ is a principal factor of $w$ with principal index $l_1$, completing the proof.
\end{proof}

Theorem~\ref{gMSIchar} shows there are precisely two types of MSIs. We will refer to the first type by saying ``an MSI caused by a strong descent,'' or just by referring to a word $v_i$ as a strong descent. We will refer to the second type by saying ``an MSI caused by the principal factor $p_{v_i}$.''  This language implies that $p_{v_i}$ is a principal factor of the word $v_i$, $l_{i+1}$ is its principal index with respect to the embedding $\eta_{v_i}$ of $v_i$ into $w$, the interval $C(v_i,p_{v_i})$ is an MSI in the related chain $C$, and  $C[v_i,p_{v_i}]$ is the first admissible chain in $[p_{v_i},v_i]$ ending at the prefix embedding of $p_{v_i}$.

Our next goal is to determine precisely which maximal chains are critical chains in an interval in $\P^*$.  Our first objective is to consider those critical chains that consist entirely of strong descents.

\begin{Prop}
\label{alldes}
Suppose $[u,w]\subset\P^*$, $(u,w)$ is non-empty, and $w$ is not flat.  Suppose $\eta$ is an embedding of $u$ into $w$.  Then there is a critical chain $C$ in $[u,w]$ ending at $\eta$ that consists entirely of (strong) descents if and only if $w(i)-\eta(i)\leq 1$ for all $i$, $\eta(2)\neq0$, and $\eta(|w|-1)\neq0$.  Furthermore, these conditions imply $|w|-|u|\leq2$ and that $[u,w]$ has at most two critical chains consisting entirely of (strong) descents.
\end{Prop}
\begin{proof}
Suppose $C$ is a critical chain in $[u,w]$ that ends at $\eta$ and consists entirely of descents.  Then $l_1>l_2>\ldots>l_{n-1}>l_{n}$.  This implies $w(i)-\eta(i)\leq 1$ for all $i$ as each letter can be reduced at most once.  Since $1$'s may only be reduced at the beginning or end of a word, the decreasing label sequence also implies $\eta(2)\neq0$ since it is not possible to reduce position $1$ before position $2$.  To finish this implication, suppose for a contradiction that $\eta(|w|-1)=0$.  Since each $l_i$ is distinct and the sequence is decreasing, it follows that $w=v_211$ as the last letter must be reduced to zero before the second to last can be.  So by Theorem~\ref{gMSIchar}, the MSI containing $v_1$ is not caused by a descent and therefore must be caused by a principal factor $p(w)$.  Since the chain id is decreasing, the corresponding principal index $l_1$ must be $|w|$. By the definition of principal index, $w(|w|)>1$.  This contradicts $w=v_211$, implying $\eta(|w|-1)\neq0$.

Suppose $w(i)-\eta(i)\leq 1$ for all $i$, $\eta(2)\neq0$, and $\eta(|w|-1)\neq0$.  The first assumption implies each entry in $M_{\eta}$ is distinct.  If $u$ is not flat, then the decreasing and $\eta(2)\neq0$ conditions imply the two conditions in the definition of admissibility. If $u$ is flat, having $w(i)-\eta(i)\leq 1$ implies $w$ consists only of $1$'s and $2$'s.  So the decreasing permutation of $M_{\eta}$ is strongly admissible because once the last $2$ has been reduced to a $1$, at most an initial $1$ remains to be reduced since $\eta(2)\neq0$.  Therefore, Proposition~\ref{permchar} implies that the decreasing permutation of $M_{\eta}$ is the chain id of a maximal chain $C$ in $[u,w]$.  Since each entry in the chain id is distinct, $C$ consists entirely of descents.  Furthermore, since $\eta(|w|-1)\neq0$, $w\neq v_211$.  So no descent has the property that $v_{i-1}\neq v_{i+1}11$ as a decreasing chain id only permits the removal of consecutive $1$'s from the end of a word at the beginning of the chain.  Therefore, by Theorem~\ref{gMSIchar} each is a length 1 MSI in $C$, implying that $J(C)$ covers $C(w,u)$.  Thus, $C$ is a critical chain consisting entirely of descents and ending at $\eta$.

To see each interval $[u,w]$ has at most two critical chains consisting entirely of descents, note that $|u|\geq |w|-2$ because the first and last positions in an embedding satisfying the necessary restrictions are the only ones which can be zero.  Whenever $|u|=|w|$, there is also only one such embedding.  If $|u|=|w|-2$ there can only be one such embedding as only the first and last letters of $w$ can be reduced, implying $0u0$ is the only embedding of $u$ into $w$.  If $|u|=|w|-1$, then it is possible the prefix embedding and the suffix embedding satisfy the necessary restrictions.  Thus, there are at most two embeddings of $u$ that satisfy the restrictions for a critical chain consisting entirely of descents.  Since there is only one weakly decreasing permutation of any set $M_{\eta}$, this completes the proof.
\end{proof}

We note the example $[121,1221]$ from Table~\ref{table:121,1221} provides a nice illustration of this proposition, as the critical chains whose chain ids are $2-1$ and $4-3$ have MSIs caused by strong descents.  As stated in the theorem, we can have at most $2$ critical chains consisting entirely of strong descents.

The remaining critical chains must contain at least one MSI caused by a principal factor.  So these chains contain a principal factor of $w$ or a principal factor of some $v_i$ with the property that $C(w,v_{i+1})$ consists entirely of descents $v_k$ such that $v_{k-1}\neq v_{k+1}11$.  The second possibility is easier to work with if we consider the relationship between the principal factor of $v_i$ and $w$. This is the content of the next proposition.  

\begin{Prop}
\label{ofstruc}
If a critical maximal chain $C$ in $[u,w]$ does not consist entirely of descents, then it contains a $v_i\in (u,w]$ with a principal factor $p_{v_i}\geq u$ such that $p_{v_i}1^m$ is a maximal outer factor of $w$ for some $m\geq0$.
\end{Prop}

\begin{proof}
By Theorem~\ref{gMSIchar}, $C$ contains at least one MSI caused by a principal factor.  Suppose $C(v_i,p_{v_i})$ is the first such MSI.  If $p_{v_i}$ is a principal factor of $w$, there is nothing to show.  If not, then by Theorem~\ref{gMSIchar}, $w\neq v_211$ and the chain id of $C$ decreases through $l_{i+1}$.  This implies only the last letter of $w$ can be reduced to $0$ in $\eta_{v_i}$. So $v_i$ and thus $p_{v_i}$ are prefixes of $w$.  

To complete the proof, it suffices to show that $p_{v_i}1^m$ is an outer factor for some $m\geq0$ because then the theorem holds for $p_{v_i}1^\ell$, where $\ell=|v|-|p_{v_i}|$ and $v$ is a maximal outer factor of $w$ containing $p_{v_i}$.  Recall that $p_{v_i}$ is an outer factor of $v_i$, which is a prefix of $w$ such that only the last letter of $w$ can be reduced to $0$ in $\eta_{v_i}$.  Therefore, $p_{v_i}$ or $p_{v_i}1$ is an outer factor of $w$ so that desired statement holds for $m=0$ or $m=1$. 
\end{proof}

In order to have a critical chain in $[u,w]$ involving an MSI resulting from a principal factor $p_{v_i}$, $p_{v_i}$ needs to be contained in a different MSI or $p_{v_i}$ must equal $u$. By Theorem~\ref{gMSIchar}, $p_{v_i}$ could be contained in one of three types of MSI: an MSI caused by a strong descent, an overlapping MSI caused by a principal factor, or an adjacent MSI caused by a principal factor.  We will show the last possibility cannot occur. 

\begin{Prop}
\label{pfover}
Suppose $C(v_i,p_{v_i})$ is an MSI of a critical chain $C$ caused by the principal factor $p_{v_i}$.  Then $p_{v_i}$ is contained in an MSI caused by a descent or an overlapping MSI caused by a principal factor.
\end{Prop}
\begin{proof}
Let $v_{j+1}=p_{v_i}$.  Suppose for a contradiction that $C(v_j,p_{v_j})$ is an MSI caused by a principal factor.  Since $C(v_i,v_{j+1})$ is an MSI caused by a principal factor, Theorem~\ref{gMSIchar} implies $C[v_i,v_{j+1}]$ is the first admissible chain ending at the corresponding embedding.  So by Lemma~\ref{flc}, $v_{j}=v_{j+1}1$.  However, this implies the value at the principal index $l_{j+1}$ of $p_{v_j}$ in $v_j$ is $v_j(l_{j+1})=1$.  This contradicts the fact that $v_j(l_{j+1})>1$, as $p_{v_j}$ was assumed to be a principal factor of $v_j$.  So $C(v_j,p_{v_j})$ is not an MSI caused by a principal factor, completing the proof.
\end{proof}

While refining the set of MSIs $I(C)$ to create the set $J(C)$,  we reduce any intervals remaining in $I(C)$ that overlap with an earlier MSI and remove any that are no longer containment minimal after this reduction.  The next proposition states that overlapping MSIs in the set $I(C)$ of a critical chain must come in pairs.

\begin{Prop}
\label{overlap}
An MSI of a critical chain can overlap with at most one other MSI.
\end{Prop}
\begin{proof}
Let $C$ be maximal chain of $[u,w]$.  Suppose $C(v_{i_1},v_{j_1}), \ldots, C(v_{i_k},v_{j_k})$ is a maximal sequence of three or more overlapping MSIs satisfying ${i_1}<{i_2}<\ldots<{i_k}$.  Since MSIs caused by strong descents always have length one, they cannot be involved in overlapping MSIs.  So by Theorem~\ref{gMSIchar}, each $v_{j_\ell}$ is a principal factor of $v_{i_\ell}$, $l_{i_\ell+1}$ is the principal index of $v_{j_\ell}$ with respect to the embedding $\eta_{v_{i_\ell}}$ of $v_{i_\ell}$ into $w$, and $C[v_{j_\ell},v_{j_\ell}]$ is the first admissible chain in $[v_{j_\ell},v_{j_\ell}]$ ending at the prefix embedding of $v_{j_\ell}$.  By Lemma~\ref{flc}, any two consecutive MSIs must have unimodal label sequences that overlap.  Note each $\eta_{v_{i_\ell}}(l_{i_\ell+1})>1$ because $l_{i_\ell+1}$ corresponds to a principal index. So each label sequence must contain at least one entry before reaching its decreasing suffix.  Therefore, the overlap between two intervals must start on the weakly increasing portion of the label sequences.  Since no letters are removed before reaching the decreasing suffix, this implies  $|v_{i_1}|=|v_{i_2}|=\ldots=|v_{i_k}|$. So each MSI $C(v_{i_\ell},v_{j_\ell})$ includes all of the decreasing suffix of $C(v_{i_1},v_{j_1})$, and each MSI except the first contains all of the decreasing suffix of $C(v_{i_2},v_{j_2})$. 

Consider the process of refining the set of intervals $I(C)$ to $J(C)$.  The first interval in the sequence of overlapping intervals, $C(v_{i_1},v_{j_1})$, is added to $J(C)$.  Any other interval in $I(C)$ has its overlap with this interval removed, and any interval that is no longer containment minimal is discarded.  However, each interval $C(v_{i_\ell},v_{j_\ell})$ left in $I(C)$ contains the truncated portion of $C(v_{i_2},v_{j_2})$.  So $C[v_{j_1},v_{j_2})$ is added to $J(C)$ and the rest of the intervals are discarded.  Thus, $v_{j_2}$ is not in an interval in $J(C)$ and $C$ is not a critical chain.
\end{proof}

We now have enough information to give a nice description of the structure of a critical chain.

\begin{Thm}
\label{JCstruc}
Let $C:w=v_0\stackrel{l_1}{\rightarrow}v_1\stackrel{l_2}{\rightarrow}\ldots\stackrel{l_{n-1}}{\rightarrow} v_{n-1}\stackrel{l_n}{\rightarrow}v_n=u$ be a maximal chain of $[u,w]$.  Then $C$ is a critical chain if and only if $C(w,u)$ can be written as a sequence of intervals $$C[v_{i_1}=v_1,v_{i_2})\cup C[v_{i_2},v_{i_3}) \cup \ldots \cup C[v_{i_{k-1}},v_{i_k}=u)$$ where each interval $C[v_{i_j},v_{i_{j+1}})$ is one of the following three types:
\begin{enumerate}
\item $C[v_{i_j},v_{i_{j+1}})$ is an MSI caused by the strong descent $v_{i_{j}}$.
\item $C[v_{i_j},v_{i_{j+1}})$ is an MSI caused by the principal factor $v_{i_{j+1}}$ of the word $v_{i_j-1}$.
\item The word $v_{i_{j+1}}$ is a principal factor of a word in  $C[v_{i_{j-1}},v_{i_{j}})$ and satisfies $v_{i_{j+1}}1^m=v_{i_{j}}$ for $m=|v_{i_{j}}|-|v_{i_{j+1}}|>0$. The value $m$ is unique in the sense that no other word satisfies the description of $v_{i_{j+1}}$ for another value $m$.
\end{enumerate}
Furthermore, type (1) intervals are followed by intervals of type (1) or (2), type (2) intervals are followed by intervals of type (1) or (3), and type (3) intervals are followed by intervals of type (1).  Finally, only intervals of type (1) or (2) can begin the decomposition.
\end{Thm}

\begin{proof}
For the forward implication, we need to show that in a critical chain, each interval in the set $J(C)$ is one of the three interval types listed and the order of the intervals respects the ordering restrictions of the proposition.  By Theorem~\ref{gMSIchar}, each MSI is caused by a strong descent or a principal factor.  So intervals of type (1) and type (2) can occur in $J(C)$. Next, since Proposition~\ref{overlap} states that overlapping intervals must occur in pairs, we need to show that the second interval in an overlapping pair is an interval of type (3).  Let $C[v_{i_j},v_{i_{j+1}})$ be the remainder of an interval from $I(C)$ that was reduced in $J(C)$.  From the proof of Proposition~\ref{overlap},  $C[v_{i_j},v_{i_{j+1}})$ must be the remainder of the MSI $C(v,v_{i_{j+1}})$ caused by the principal factor $v_{i_{j+1}}$, where the preceding interval $C[v_{i_{j-1}},v_{i_{j}})$ is of type (2) and $|v|=|v_{i_{j-1}}|$. Since $1$'s in these MSIs can only be trimmed from the back, $v_{i_j}=v_{i_{j+1}}1^m$ for $m=|v_{i_j}|-|v_{i_{j+1}}|$.  To assure this is the second interval in a pair of overlapping intervals in $I(C)$, there cannot be another word in $C[v_{i_j},u]$ that is a principal factor of a word in $C[v_{i_{j-1}},v_{i_{j}})$.  In particular, this means there can be no word in $C[v_{i_j},u]$ satisfying the previous description of $v_{i_{j+1}}$ for another value $m$.   Therefore, $C[v_{i_j},v_{i_{j+1}})$ is an interval of type (3) and all intervals reduced from overlapping MSIs in $I(C)$ are of this type. Thus, there are no other types of intervals that can occur in $J(C)$.  Finally, we need to verify the ordering restrictions.  Type (1) intervals cannot be followed by intervals of type (3) because the half-open interval under consideration is empty.  Intervals of type (2) cannot be followed by type (2) intervals by Proposition~\ref{pfover}.  Type (3) intervals cannot be followed by type (2) intervals by Proposition~\ref{pfover}, and cannot be followed by type (3) intervals because the value $m$ is unique.  And type (3) intervals cannot begin the decomposition because the half-open interval under consideration does not exist.  This completes the proof of this implication.

For the backwards implication, since the given set of intervals covers $C$, we need to show the interval types are always reductions of MSIs in $I(C)$.  Type (1) and type (2) intervals are MSIs by Theorem~\ref{gMSIchar}.  Type (3) intervals must occur after a type (2) interval by definition.  Suppose $C[v_{i_{j}},v_{i_{j+1}})$ is a type (3) interval so that the word $v_{i_{j+1}}$ is a principal factor of a word $v$ in  $[v_{i_{j-1}},v_{i_{j}})$.  Lemma~\ref{flc} implies that the label sequence of $C[v_{i_{j-1}},v_{i_{j}})$ is unimodal.  So Lemma~\ref{flc} also implies that $v$, the word which has $v_{i_{j+1}}$ as a principal factor, occurs before the decreasing suffix of the label sequence of $C[v_{i_{j-1}-1},v_{i_{j}}]$.   Thus, $C[v,v_{i_{j+1}}]$ is the first admissible chain in the corresponding interval $[v_{i_{j+1}},v]$ with $v_{i_{j+1}}$ in the prefix embedding.  So by Theorem~\ref{gMSIchar}, $C(v,v_{i_{j+1}})$ is the MSI in $I(C)$ that is reduced to the interval $C[v_{i_{j}},v_{i_{j+1}})$.  Furthermore, this reduced interval appears in $J(C)$ because $v_{i_{j+1}}1^m=v_{i_{j}}$ for a unique value of $m$.   Indeed, if $m$ were not unique, a different reduced interval would either be contained inside this one or have this interval contained within it; both cases contradict the fact that $J(C)$ covers $C$.  Since the ordering restrictions respect Propositions~\ref{pfover} and~\ref{overlap}, this completes the proof of this implication.
\end{proof}

In the interval $[121,1221]$, the critical chain with chain id $2-1$ consists of a single type (1) interval, $C[1121,121)=C(1221,121)$ (see Table~\ref{table:121,1221}).  In the interval $[2,2212]$, the critical chain with chain id $2-4-4-3-2$ consists of the type (2) interval $C[2112,211)=C(2212,211)$ followed by the type (3) interval $C[211,2)$, which is truncated from the MSI $C(2112,2)$ (see Table~\ref{table:2,2212}).

Using Theorem~\ref{JCstruc}, we can separate the critical chains of $[u,w]$ into three groups based on what happens after the first type (2) interval in the chain: critical chains with no type (2) intervals, critical chains whose first type (2) interval is the last interval or is followed by a type (1) interval, and critical chains whose first type (2) interval is followed by a type (3) interval.  Notice the first group is investigated in Proposition~\ref{alldes}.  So if we can find the total contribution of all critical chains whose first interval is a given type (2) interval, we will have enough information to write down a recursive formula for the M\"obius value.

Let $C$ be a critical chain.  If $J(C)$ contains any type (2) intervals, then by Theorem~\ref{JCstruc} the first type (2) interval in the set must either be the first interval or occur after a sequence of type (1) intervals. Furthermore, if the first type (2) interval is followed by a type (3) interval, Theorem~\ref{JCstruc} implies the principal factor causing the type (3) interval must be a certain prefix of the principal factor causing the type (2) interval. So to facilitate the exposition, we need to make several new definitions.

Define $w\setminus1^m$ to be the word that results when $m$ $1$'s are removed from the suffix of $w$, or as undefined if $w$ ends in less than $m$ $1$'s.

Define a word $v$ to be a \emph{base of $w$} if $v(j)=w(j)$ or $v(j)=w(j)-1$ for all $j$ and $|v|=|w|$ or $|w|-1$.  Define the \emph{degree} of a base to be the number of indices $j$ for which $v(j)=w(j)-1$.  This way, if a word $v_i$ in a chain $C$ is a base of $w$, then it is a base of degree $i$.  When $i>0$, the language ``$v_i$ is based in $C$'' will indicate the word $v_i$ of the chain $C$ is a base of $w$ and the labels $l_1,\ldots,l_i$ of $C$ form a decreasing sequence.  In Proposition~\ref{alldes}, we showed that this condition forces each word $v_1,\ldots,v_{i-1}$ to be a strong descent.

Suppose $v_i$ is a base of $w$ of degree $i$.  Let $l$ be the index of the smallest position satisfying $v_i(l)=w(l)-1$, or $|w|+1$ if $i=0$. We define any word $p_{v_i}$ that is a principal factor of $v_i$ and whose principal index takes a value less than $l$ to be a \emph{principal factor of $w$ of degree $i$}. This definition is an extension of the definition of a principal factor because a principal factor of $w$ satisfies the new definition for $i=0$.  So to maintain consistency in the language, when a degree is not noted in the language or the notation, the assumption will be that the principal factor has degree $0$.

In the example $[2,2212]$ found in Table~\ref{table:2,2212}, the bases of $2212$ of degree $1$ are $1212$, $2112$, and $2211$.  The base $2211$ of admits $211$ as a principal factor of degree $1$, but $2112$ and $1212$ have principal factors with principal indices greater than the smallest position satisfying $v_i(l)=w(l)-1$.  Returning to an example first given after the definition of a principal factor, the degree $2$ base $12211$ of the word $12222$ has $1211$ as a principal factor, making $1211$ a principal factor of degree $2$.  In fact, $1211$ is a principal factor of $4$ bases of $12222$: $12211$, $12212$, $12221$. and $12222$ itself.

Let $p_{v_i}$ be a principal factor of $w$ of degree $i$.  By Theorem~\ref{JCstruc}, $p_{v_i}$ could cause the first type (2) interval $C(v_i,p_{v_i})$ in some critical chain $C$ because when $v_i$ is based in $C$, the condition on the principal index of $p_{v_i}$ implies $v_i$ is a strong descent.  Notice $p_{v_i}$ must be an outer factor of $w$ or $w\setminus1$.   The latter is a possibility when $w\setminus1$ is defined because, by the definition of a base, $v_i(|w|)$ could be $0$. Furthermore, the proof of Proposition~\ref{ofstruc} implies $p_{v_i}1^m$ is a maximal outer factor of $w$ for some $m\geq0$. Finally, $p_{v_i}$ could be a principal factor of other bases $v_j$ of $w$.

Suppose $p_{v_i}$ is a principal factor of $v_i$.  Let $v_i'$ be the word that results when the letter in $v_i$ at the principal index of $p_{v_i}$ is reduced by $1$.  That is, $v_i'(j)=v_i(j)-1$ when $j$ is the principal index of $p_{v_i}$ and $v_i'(j)=v_i(j)$ for all other indices $j$.  Define the \emph{primary prefix} $x({p_{v_i}})$ of a principal factor $p_{v_i}$ to be the factor $p_{v_i}\setminus 1^m$, where $m$ is the smallest positive integer such that $p_{v_i}\setminus 1^m$ is an outer factor of $v_i'$.  So if no $m>0$ satisfies the restriction, the primary prefix is undefined.

Intuitively, the primary prefix of a principal factor $p_{v_i}$ is the longest proper prefix of $p_{v_i}$ that has two embeddings in $v_i'$ and only differs from $p_{v_i}$ by some number of $1$'s removed from the back.  Notice the primary prefix depends on the word $v_i'$, or perhaps more intuitively, the word $v_i$ and the principal index of $p_{v_i}$ in $v_i$.  In the example $[2,2212]$, $2$ is the primary prefix of the principal factor $211$ of $2212$.  However, in the interval $[2,2222]$, which contains the previous interval as a subinterval, $21$ is the primary prefix of the principal factor $211$ of $2222$.  And in the interval $[2,2211]$, the principal factor $211$ does not have a primary prefix.

The following proposition asserts that the primary prefix is the only word that can cause a type (3) interval after the type (2) interval $C(v_i,p_{v_i})$.

\begin{Prop}
\label{pmpfx}
Let $C$ be a maximal chain of $[u,w]$ and suppose $C(v_i,p_{v_i})$ is a type (2) interval in the set $J(C)$.  Then $C(v_i,p_{v_i})$ is followed by a type (3) interval $C[p_{v_i},x)$ in $J(C)$ if and only if $x$ is the primary prefix of $p_{v_i}$ and $x$ appears in $C$.
\end{Prop}
\begin{proof} First suppose $x$ is the primary prefix of $p_{v_i}$ and that it appears in $C$.  By definition, $x$ is a maximal outer factor of $v_i'$.  Thus, it is a maximal outer factor of any word in $C(v_i,p_{v_i})$ of which it is an outer factor.   Let $v_m$ be the last word in the interval $C(v_i,p_{v_i})$ that contains $x$ as an outer factor.  We will show $x$ is a principal factor of $v_m$ with principal index $l_{m+1}$.  Since $p_{v_i}$ is a prefix of $v_{m+1}$, $x$ is as well.  So $v_{m+1}$ no longer contains the suffix embedding of $p_{v_i}$.  Furthermore, since  $C[v_i,p_{v_i}]$ is the first admissible chain in $[p_{v_i},v_i]$ ending at the prefix embedding of $p_{v_i}$, for all $l_{i+1}<k<l_{m+1}$, either we have $v_m(k)=p_{v_i}(k)$ or we have $v_m(k)=1$ and $p_{v_i}(k)=0$.  Therefore, for all $l_{i+1}<k<l_{m+1}$, either we have $v_m(k)=x(k)$ or we have $v_m(k)=1$ and $x(k)=0$.  This implies $C[v_m,x]$ is the first admissible chain in $[x,v_m]$ ending at the prefix embedding of $x$ and $v_m(l_{m+1})$ is the first reducible letter in $v_m$ greater than the corresponding position in the prefix embedding of $x$.  Thus, $l_{m+1}$ satisfies the definition of a principal index, which implies $x$ is a principal factor of $v_m$ and $C(v_m,x)$ is an MSI of $C$.  Since $x$ is a maximal outer factor of $v_i'$, and $C(v_i,p_{v_i})$ is a type (2) interval, no word between $p_{v_i}$ and $x$ can be a principal factor of a word $v_k$ in $C$.  Therefore, $C(v_m,x)$ is reduced to the type (3) interval $C[p_{v_i},x)$ in $J(C)$, completing the reverse implication.

Now suppose that $C(v_i,p_{v_i})$ is followed by a type (3) interval $C[p_{v_i},x)$ in the set $J(C)$.  Then by Theorem~\ref{JCstruc}, $x$ is a principal factor of a word $v$ in $C(v_i,p_{v_i})$ and $x=p_{v_i}\setminus1^m$ for some $m$.  From the proof of Theorem~\ref{JCstruc}, we know $|v|=|v_i'|$ and $v\leq v_i'$, implying $x$ is an outer factor of $v_i'$.  By Theorem~\ref{JCstruc} part (3), it suffices to show that $x$ is a maximal outer factor of $v_i'$.  For a contradiction, suppose $x$ is not a maximal outer factor of $v_i'$.  Then a word of the form $x1^k$ would be a maximal outer factor of $v_i'$ and by the argument in the paragraph above, a principal factor of some word in $C(v_i,p_{v_i})$.  Thus, $C(v,x1^k)$ would be an MSI in $I(C)$.  This would be reduced to the interval $C[p_{v_i},x1^k)$, which is contained in $C[p_{v_i},x)$, implying that $C[p_{v_i},x)$ could not be in $J(C)$.  This is a contradiction.  Thus, $x$ is a maximal outer factor of $v_i'$, implying it is the primary prefix of the word $p_{v_i}$. 
\end{proof}

Let $\mu(u,v)$ be the normal M\"obius function if $u$ and $v$ are both elements of $\P^*$, or zero if either is undefined.  Define the function $\nu(u,v)$ to be 
$$\nu(u,v)=\displaystyle\sum_{i\geq0} \mu(u,v\setminus 1^i).$$
Notice all the terms in the summation will be zero beyond the largest value $i=m$ for which $v\setminus1^m$ is defined, or the smallest value $i=m$ for which $v\setminus1^m\leq u$.

We are now ready to consider the contribution of critical chains whose first type (2) interval is a specific interval.  This proof is very technical, and we have broken it up into several cases to make it easier to follow.
 
\begin{Prop}
\label{type2}
Suppose $C(v_i,p_{v_i})$ is the first type (2) interval of some critical chain $C$ of $[u,w]$.  Then $p_{v_i}$ is a principal factor of $w$ of degree $i$ and $v_i$ is based in $C$.  Furthermore, the contribution to the M\"obius value $\mu(u,w)$ of all critical chains in $[u,w]$ that have $C(v_i,p_{v_i})$ as the first type (2) interval  is
$$\displaystyle (-1)^i\left(\nu(u,p_{v_i})-  \nu(u,x(p_{v_i}))\right),$$
where $x(p_{v_i})$ is the primary prefix of $p_{v_i}$.
\end{Prop}
\begin{proof} Since $C(v_i,p_{v_i})$ is a type (2) interval, Theorem~\ref{JCstruc} implies $p_{v_i}$ must be a principal factor of $v_i$.  If $i=0$, $p_{v_i}$ is a principal factor of $w$ of degree $0$.  If $i\neq 0$, then since $C$ is a critical chain and $C(v_i,p_{v_i})$ is the first type (2) interval, Theorem~\ref{JCstruc} implies $v_i$ and the words $v_{1}, \ldots, v_{i-1}$ which precede it must each be contained in a type (1) interval. Since $v_211\neq v_0$, $|v_i|=|w|$ or $|w|-1$.  Also, the $i$ labels preceding $v_i$ are distinct and form a decreasing sequence.  Therefore, $v_i(j)=w(j)$ or $v_i(j)=w(j)-1$.    This implies $v_i$ is a base of $w$, which allows us to conclude $v_i$ is based in $C$.  Since $v_i$ is in a type (1) interval, $l_{i+1}<l_i$, implying $p_{v_i}$ is a principal factor of $w$ of degree $i$.

To facilitate the discussion, we will first consider the case when $p_{v_i}$ ends with a letter other than $1$.  Then $p_{v_i}\setminus1$ is undefined, so $\nu(u,p_{v_i})=\mu(u,p_{v_i})$.  Note that the set $J(C)$ must cover the entire chain $C(w,u)$ for $C$ to be critical, and we already know $J(C)$ covers $C(w,p_{v_i})$ when $C(v_i,p_{v_i})$ is the first type (2) interval.  Furthermore, $p_{v_i}$ does not have a primary prefix, so by Theorem~\ref{JCstruc}, $C(v_i,p_{v_i})$ can only be followed by an interval of type (1) in $J(C)$.    Recall from the proof of Theorem~\ref{gMSIchar} that the last label in $C(v_i,p_{v_i})$, $l_k$, is the last non-zero position in $p_{v_i}1$.  So $l_k=|p_{v_i}1|$ by Lemma~\ref{flc}.  Thus, in any critical chain containing the type (2) interval $C(v_i,p_{v_i})$, $l_{k+1}<l_k$.  We now consider two subcases depending on whether $u=p_{v_i}$.  

First suppose $u<p_{v_i}$.  Since $p_{v_i}$ does not end in a $1$, it must be a strong descent. So in any maximal chain containing $C[v_i,p_{v_i}]$, $p_{v_i}$ is contained in a type (1) interval.   It follows that $J(C[p_{v_i},u])$ must cover $C(p_{v_i},u)$, whose first label corresponds to $l_{k+1}$ in $C$.  Since any choice of label $l_{k+1}$ puts $p_{v_i}$ in a type (1) interval, the set of critical chains in $[u,w]$ that contain $C(v_i,p_{v_i})$ are in a one-to-one correspondence with the critical chains of $[u,p_{v_i}]$.  The corresponding map between the sets $J(C)$ and $J(C[p_{v_i},u])$ is given by the addition or subtraction of the $i+2$ intervals $v_1, \ldots, v_i, C(v_i,p_{v_i}),$ and $p_{v_i}$.  Therefore, the number of intervals in $J(C)$ is $i+2+|J(C[p_{v_i},u])|$.  So by Theorem~\ref{chainmu}, the total contribution of all these chains to $\mu(u,w)$ is
$$(-1)^{i+2}\mu(u,p_{v_i})=(-1)^i\mu(u,p_{v_i}).$$

Now suppose $u=p_{v_i}$. The above formula follows from Theorem~\ref{chainmu}, as $\mu(p_{v_i},p_{v_i})=1$ and there are precisely $i+1$ intervals in $J(C)=J(C([w,p_{v_i}]))$.  So in this case, $\mu(u,w)=(-1)^i$.  Since $p_{v_i}$ does not have a primary prefix and $\nu(u,p_{v_i})=\mu(u,p_{v_i})$, this completes the proof for this case.

Now we consider the case $p_{v_i}$ ends with a $1$.  If $u=p_{v_i}$, the argument does not change from the one above.  If $u<p_{v_i}$, there are two significant differences from the previous discussion.  First, while $p_{v_i}$ is still a descent $v_k$ in any maximal chain containing $C(v_i,p_{v_i})$, it is possible it could be a weak descent.  The second difference is $p_{v_i}$ could have a primary prefix, implying that $C(v_i,p_{v_i})$ could be followed by a type (3) interval in a critical chain.  Fortunately, because of the required conditions on the descent $p_{v_i}$,  $C(v_i,p_{v_i})$ can only be followed by a type (1) interval if $l_{k+1}<|p_{v_i}|$, and a type (3) interval if $l_{k+1}=|p_{v_i}|$.  This allows us to consider these two cases separately.

First, suppose $p_{v_i}=v_k$ ends with a $1$ and a critical chain $C$ containing the type (2) interval $C(v_i,p_{v_i})$ has $l_{k+1}<|p_{v_i}|$.  Then $C$ can only be critical if $C(v_i,p_{v_i})$ is followed by a type (1) interval and $C(w,v_i]$ consists of $i$ type (1) intervals.  So as above, it follows that $C$ is critical if and only if $J(C[p_{v_i},u])$ covers $C(p_{v_i},u)$, whose first label is $l_{k+1}$ in $C$.  Since $l_{k+1}<|p_{v_i}|$, we have a one-to-one correspondence between the critical chains of $[u,w]$ that contain $C(v_i,p_{v_i})$ and the critical chains of $[u,p_{v_i}]$ whose first label is not $|p_{v_i}|$.  Notice that $\mu(u,p_{v_i})$ could count critical chains of $[u,p_{v_i}]$ whose first label is $|p_{v_i}|$ and first word is $p_{v_i}\setminus1$. So we must subtract out these critical chains when calculating the contribution to the M\"obius value.  Since the corresponding map between the sets of critical chains is given by the addition or subtraction of the $i+2$ intervals $v_1, \ldots, v_i, C(v_i,p_{v_i}),$ and $p_{v_i}$, it follows from Theorem~\ref{chainmu} that the total contribution of all these critical chains to $\mu(u,w)$ is
$$\displaystyle(-1)^{i+2}\left(\mu(u,p_{v_i})-O(u,p_{v_i})\right)=(-1)^i(\mu(u,p_{v_i})-O(u,p_{v_i})),$$
where $O(u,p_{v_i})$ is the total contribution to $\mu(u,p_{v_i})$ of the critical chains of $[u,p_{v_i}]$ for which $v_1'=p_{v_i}\setminus1$.  Since $|p_{v_i}|$ does not satisfy the definition of a principal index, Theorem~\ref{JCstruc} implies that $p_{v_i}\setminus1$ can only be in a critical chain of $[u,p_{v_i}]$ if it is in a type (1) interval.  If $p_{v_i}\setminus1^2$ is undefined, we are looking for all critical chains of $[u,p_{v_i}\setminus1]$.  However, if $p_{v_i}\setminus1^2$ is defined, $O(u,p_{v_i})$ is the contribution of all critical chains of $[u,p_{v_i}\setminus1]$ that do not start with the values $|p_{v_i}\setminus1|$ and $|p_{v_i}\setminus1^2|$, as this would put $p_{v_i}\setminus1$ in a non-MSI-causing descent in $[u,p_{v_i}]$.    Furthermore, these critical chains contain one less interval than those of $[u,p_{v_i}]$ because they do not contain the interval $p_{v_i}$.  So we need to account for this by taking the negative of the values of the chains in $[u,p_{v_i}\setminus1]$.  Thus,
$$\displaystyle O(u,p_{v_i})=-\left(\mu(u,p_{v_i}\setminus1)-O(u,p_{v_i}\setminus1)\right),$$
since $O(u,p_{v_i}\setminus1)$ is the total contribution to $\mu(u,p_{v_i}\setminus1)$ of those critical chains of $[u,p_{v_i}\setminus1]$ for which $v_1''=p_{v_i}\setminus1^2$.  This recursive definition for $O(u,p_{v_i})$ terminates when we find an $\ell$ for which $p_{v_i}\setminus1^\ell$ is undefined or $p_{v_i}\setminus1^\ell<u$ because both of these cases imply $O(u, p_{v_i}\setminus1^\ell)=0$.  Therefore,
$$\displaystyle O(u,p_{v_i})=-\left(\mu(u,p_{v_i}\setminus1)+\mu(u,p_{v_i}\setminus1^2)+\ldots+\mu(u,p_{v_i}\setminus1^{\ell-1})\right).$$
This implies that when $p_{v_i}=v_k$ ends in a $1$, the total contribution of all critical chains containing the type (2) interval $C(v_i,p_{v_i})$ and satisfying $l_{k+1}<|p_{v_i}|$ is $$\displaystyle(-1)^i\left(\mu(u,p_{v_i})+\nu(u,p_{v_i}\setminus1)\right)=(-1)^i\nu(u,p_{v_i}).$$

Our last goal is to consider the case when $p_{v_i}=v_k$ ends with a $1$ and a critical chain $C$ containing the type (2) interval $C(v_i,p_{v_i})$ has $l_{k+1}=|p_{v_i}|$.  In this case, the set $J(C)$ must contain $i$ type (1) intervals before $C(v_i,p_{v_i})$, and a type (3) interval immediately following $C(v_i,p_{v_i})$.  By Proposition~\ref{pmpfx}, the type (3) interval must be caused by the primary prefix of $p_{v_i}$, $x(p_{v_i})$.  So the type (3) interval is $C[p_{v_i},x(p_{v_i}))$.  Notice that Theorem~\ref{JCstruc} implies $x(p_{v_i})$ must be contained in a type (1) interval.  In the previous case, note that $p_{v_i}$ is the first word in the interval $[u,w]$ which is not in the intervals of $J(C)$ required for the type (2) interval $C(v_i,p_{v_i})$, and $p_{v_i}$ must be contained in a type (1) interval.  In this case, $x(p_{v_i})$ is the first word in $[u,w]$ that is not contained in the intervals of $J(C)$ required for the type (3) interval $C[p_{v_i},x(p_{v_i}))$, and $x(p_{v_i})$ must be contained in a type (1) interval for $C$ to be critical. Therefore, the argument is very similar, except $x(p_{v_i})$ takes the place of $p_{v_i}$.  If $u<x(p_{v_i})$, the major difference is $C$ is now critical if and only if $J(C[x(p_{v_i}),u])$ covers $C(x(p_{v_i}),u)$.  So the $i+3$ intervals $v_1, \ldots, v_i, C(v_i,p_{v_i}), C[p_{v_i},x(p_{v_i}))$, and $x(p_{v_i})$ precede the elements of $J(C[x(p_{v_i}),u])$ in the set $J(C)$.  Beyond this detail, the argument leading to $(-1)^i\nu(u,p_{v_i})$ is the same, with $x(p_{v_i})$ taking the place of $p_{v_i}$.  Therefore, the contribution to $\mu(u,w)$ of the critical chains of $[u,w]$ that contain the type (3) interval $C[p_{v_i},x(p_{v_i}))$ is
$$\displaystyle(-1)^{i+3}\nu(u,x(v_i))=(-1)^{i+1}\nu(u,x(p_{v_i})).$$
Note this formula also holds if $u=x(p_{v_i})$ because then $\nu(u,x(p_{v_i}))=\mu(u,u)=1$ and there are $i+2$ intervals in $J(C)$.  

Thus, if $x(p_{v_i})$ is the primary prefix of $p_{v_i}$, the contribution to the M\"obius value of all critical chains that contain the type (2) interval $C(v_i,p_{v_i})$ as their first type (2) interval is
$$\displaystyle (-1)^i\left(\nu(u,p_{v_i})- \nu(u,x(p_{v_i}))\right)$$
because if $x(p_{v_i})$ is not defined, $\nu(u,x(p_{v_i}))=0$.  This completes the proof.
\end{proof}

Using Propositions~\ref{alldes} and~\ref{type2}, we are able to write down a formula for $\mu(u,w)$.  Recall that $\rho$ denotes the rank function in $\P^*$, and $\rho(u,w)=\rho(w)-\rho(u)$.  For simplicity, let $0\leq t\leq2$ be the number of critical chains in $[u,w]$ that consist entirely of strong descents, and define
$$d(u,w)=\begin{cases}
t(-1)^{\rho(u,w)} & \mbox{if } \rho(u,w)>1\\
(-1)^{\rho(u,w)} & \mbox{if } \rho(u,w)\leq1.\\
\end{cases}$$

\begin{Thm}
\label{Mu2}
Suppose $u\leq w$ in the poset $\P^*$.  Then 
$$\displaystyle\mu(u,w)=d(u,w)+\sum(-1)^i\left(\nu(u,p_{v_i})-\nu(u,x(p_{v_i}))\right),$$
where the  sum is over all triples $v_i,p_{v_i},x(p_{v_i})$ such that $p_{v_i}$ is a principal factor of $w$ of degree $i$ with base $v_i$ and primary prefix $x(p_{v_i})$.
\end{Thm}

\begin{proof} First suppose ${\rho(u,w)}\leq1$.  Then $[u,w]$ cannot have any maximal chains because the open interval is empty or undefined.  So both summations contribute $0$ to the M\"obius value.  By definition of the M\"obius value, when ${\rho(u,w)}=1$, $\mu(u,w)=-1$ and when ${\rho(u,w)}=0$, $u=w$ and $\mu(u,u)=1$.  These are the values that are defined in the formula $d(u,w)$.

Now suppose ${\rho(u,w)}>1$.  If $C$ is a critical chain of $[u,w]$ that consists entirely of strong descents, then $J(C)$ consists of $\rho(u,w)-1$ intervals of length $1$.  Therefore, by Theorem~\ref{chainmu}, $C$ contributes $$(-1)^{\rho(u,w)-2}=(-1)^{\rho(u,w)}$$ to the M\"obius value $\mu(u,w)$.  By Proposition~\ref{alldes}, there are $0$, $1$, or $2$ such chains in any interval $[u,w]$, making the total contribution of all such chains $$t(-1)^{\rho(u,w)},$$ where $t$ is the total number of these chains.  This is the value in the formula for $d(u,w)$.  Thus, by adding the contribution of all critical chains that do not consist entirely of strong descents to $d(u,w)$, we will get a formula for the M\"obius value.

If a critical chain $C$ in $[u,w]$ does not consist entirely of strong descents, then by Theorem~\ref{gMSIchar}, $C$ must contain an interval caused by a principal factor.  By Proposition~\ref{type2}, the first such interval is caused by a principal factor of some degree $i$.  Let $C(v_i,p_{v_i})$ be this interval, where $i$ is the degree of the principal factor $p_{v_i}$, and consider the set $S$ of all critical chains that contain this interval as the first type (2) interval.  Then Proposition~\ref{type2} implies the contribution of the chains in $S$ to the M\"obius value is $$(-1)^i\left(\nu(u,p_{v_i})- \nu(u,x(p_{v_i}))\right),$$ where $x(p_{v_i})$ is the primary prefix of $p_{v_i}$.  This is precisely the term that appears in the given formula for the triple $v_i,p_{v_i},x(p_{v_i})$.  By Theorem~\ref{JCstruc}, summing over all such triples yields a summation which gives the contribution of all critical chains containing at least one type (2) interval.  Thus, the M\"obius value can be found by adding $d(u,w)$ to this last summation, completing the proof.
\end{proof}

We close this section with several example M\"obius function calculations using the above formula.
\vskip-1\baselineskip

\begin{align*}
\mu(121,1221)&=d(121,1221)+\nu(121,121)\\
&=2+1\\
&=3\\
&\\
\mu(2,2212)&=d(2,2212)+\nu(2,211)-\nu(2,2)\\
&=0+\mu(2,211)+\mu(2,21)+\mu(2,2)-\mu(2,2)\\
&=0+0-1+1-1\\
&=-1\\
&\\
\mu(2,3121)&=d(2,3121)+0\\
&=0
\end{align*}

For larger examples, it is helpful to collect like terms and eliminate coefficients.  In the example below, $3111$ has $31$ as its primary prefix for one degree $0$ and one degree $1$ base, while it has $3$ as its primary prefix for one degree $1$ base.

\begin{align*}
\mu(3,33133)=&0+(1-2+1)\nu(3,3111) -(1-1)\nu(3,31)-(-1)\nu(3,3)\\
& +(1-2+1)\nu(3,3112)+\nu(3,3113)+\nu(3,33)\\
=&\mu(3,3) +\mu(3,3113) +\mu(3,33)\\
=&1+\mu(3,3)+\mu(3,3)\\
=&3\\
\end{align*}

A quick investigation of the formula yields the coefficient for each term $\nu(u,v)$ is dependent on how many times $v$ occurs as a principal factor of odd versus even degree, or how many times $v$ occurs as the primary prefix of principal factors of odd versus even degree.  It is not possible for the primary prefix $x=x(p_{v_i})$ to be a principal factor of any degree.  Indeed, its principal index is the same as the principal factor $p_{v_i}=x1^m$ of $w$ so that there are at least $3$ embeddings of $x$ in $w$.  So we would need to remove the middle embedding from $w$ via strong descents before removing the suffix embedding.  This is impossible because the suffix embedding starts at a later index than any other embedding. 

\begin{Rem}
We have noticed that there is often a clear relationship between the odd and even counts, such as a binomial sum, resulting in an unusually high number of terms $\nu(u,v)$ with coefficient zero.  However, a precise description of when the coefficients are zero has eluded us.  For more comments on this, see the last section on open problems.\end{Rem}

\section{Generalized Factor Order on Trees and Forests}
\label{gfoTF}

A \emph{tree} is a poset for which the undirected graph underlying its Hasse diagram is connected and has no cycles, and a \emph{forest} is a disjoint union of trees.  A \emph{rooted tree} $T$ is a tree with a unique minimal element.  We will show in this section that our results generalize to the case of a \emph{rooted forest} $F$, which is a disjoint union of rooted trees.

Suppose $T$ is a rooted tree, and let $r$ be its root.  Since $T$ has no cycles and every element $s$ satisfies $s\geq r$, it follows that every element except the root covers a unique element.  This is why we are considering generalized factor order on these posets.  It can be shown the M\"obius function of generalized factor order on $T^*$ is similar to that of $\P^*$, and the proofs leading to the result are nearly identical.  For this reason, we have chosen not to consider this case separately from the rooted forest case.

Suppose $F$ is a rooted forest.  Like the rooted tree case, in a rooted forest, every nonminimal element covers a unique element.  However, there are multiple minimal elements in this poset.  This leads us to suspect that we need to combine the results of Section 2 with Theorem~\ref{Mu2}.  While this is largely true, we will see that the definition of principal factor does not translate quite as expected, and that having multiple minimal elements complicates several results from Section 3.

To be consist with Section 2, we say a flat word in the Kleene closure $F^*$ is a sequence of $m$'s, where $m$ is a minimal element.  Since each minimal element is the root of a tree, we say a word is \emph{rooted} if it consists entirely of minimal elements .  Note that a flat word is also rooted.  The following lemma states that the covering relations of $F^*$ are analogous to those of $\P^*$.  It's proof is similar to that of Lemmas~\ref{cover} and~\ref{gcover}. 

\begin{Lem}
\label{FTcover}
A word $w=w(1)\ldots w(n)$ in $F^*$ can cover up to $n$ words, each formed by reducing a letter in $w$ to the unique letter it covers, where reducing a minimal element means removing it from the word.  Reducing $w(1)$ and $w(n)$ will always produce a factor, while reducing $w(i)$ for $1< i< n$ can only produce a new factor if $w(i)$ is nonminimal. These words are distinct unless $w$ is flat, in which case $w$ only covers one word which is flat.\qed
\end{Lem}

Note that a minimal element $m$ cannot be reduced unless it is at the beginning or end of a word.  We maintain the convention that if a word is flat, only the first $m$ can be reduced.  This allows us to maintain the notion of a reducible letter $w(i)$.

Suppose we have a distinguished symbol $\hat{0}$ and $\hat{0}\notin F$.  Define $\hat{F}$ to be the poset $F$ with $\hat{0}$ added as the unique minimal element.  This allows us to maintain the definition of expansion from the previous sections, that is, a word $\eta\in \hat{F}$ is an expansion of $u\in F$ if $\eta\in \hat{0}^*u\hat{0}^*$.

Let $[u,w]$ be an interval in $F^*$.  Let $C:w=v_0\stackrel{l_1}{\rightarrow}v_1\stackrel{l_2}{\rightarrow}\ldots\stackrel{l_{n-1}}{\rightarrow} v_{n-1}\stackrel{l_n}{\rightarrow}v_n=u$ be a maximal chain in $[u,w]$, where the $l_i$ are defined by the corresponding sequence of embeddings $\eta_{{v_i}}$ in the sense that
$$\eta_{v_i}(l_i)=s \text{ where } \eta_{v_{i-1}}(l_i)\rightarrow s\text{ and } \eta_{v_i}(j)=\eta_{v_{i-1}}(j) \text{ when } j\neq l_i.$$
This gives each maximal chain $C$ a chain id $l_1\ldots l_n$ .  This chain id is unique because every element in $\hat{F}$ except $\hat{0}$ covers a unique element.  Notice this is the most general class of posets in which every interval ordered by generalized factor order has maximal chains with unique chain ids of this form.  By lexicographically ordering the chain ids, we get a poset lexicographic order on the maximal chains of $[u,w]$ which we will use to find the MSIs in the rooted forest case.

Let $\eta$ be an embedding of $u$ into $w$.  Let $m_i=\rho(\eta(i))-\rho(w(i))$, where $\rho$ is the rank function in $\hat{F}$.  This allows us to maintain the idea of an admissible permutation of the multiset $M_{\eta}=\{1^{m_1},2^{m_2},\ldots\}$ from the previous section.

By Lemma~\ref{FTcover}, the characterization of chain ids ending at nonflat words does not change.  Unfortunately, since both rooted and unrooted words can cover flat words, it is not possibly to write down a useful characterization of chain ids ending at flat words.  However, since flat words can only be reduced to flat words, we will be able to deal with this case separately.

\begin{Prop}
\label{Fpermchar}
Suppose $F$ is a rooted forest, and $u$ and $w$ are two elements in $F^*$ satisfying $u\leq w$.  Let $\eta$ be an embedding of $u$ into $w$, $m_i=\rho(w(i))-\rho(\eta(i))$, and $M_{\eta}=\{1^{m_1},2^{m_2},\ldots\}$.

If $u$ is not flat, then a sequence of numbers is the chain id for a maximal chain in $[u,w]$ ending at $\eta$ if and only if it is an admissible permutation of the multiset $M_{\eta}$.\qed
\end{Prop}

Since there are multiple minimal elements, we need to reconcile our previous classifications of MSIs in the positive integer and antichain cases.  To begin considering MSIs in the new setting, we need to once again identify all intervals $[u,w]$ in which a chain $C$ has $C(w,u)$ as an MSI.

Recall from Section 2 that descents always caused MSIs in the antichain case when they were strong descents, that is, $l_{i+1}<l_i-1$.  In the context of $\P^*$, a strong descent satisfied $v_{n-1}\neq v_{n+1}11$.  These conditions are analogous.  So we call any descent $v_n$ that does not remove two minimal elements from the back of a word a \emph{strong descent}, that is, $v_{n-1}\neq v_{n+1}mn$ for any minimal elements $m$ and $n$.  The next proposition states that a strong descent causes a length $1$ MSI.  Since its proof is essentially the same as that of Proposition~\ref{gdescent}, we omit it.

\begin{Prop}
\label{Fdescent}
Suppose $[u,w]\subset F^*$ and $C:w=v_0\stackrel{l_1}{\rightarrow}v_1\stackrel{l_2}{\rightarrow}\ldots\stackrel{l_{n-1}}{\rightarrow} v_{n-1}\stackrel{l_n}{\rightarrow}v_n=u$ is a maximal chain in $[u,w]$. If $v_i$ is a strong descent, then $v_i$ is a length 1 MSI.\qed
\end{Prop}

Notice this statement is not an if and only if, as in the antichain case outer word MSIs can reduce consecutive letters from the back.  Indeed, if $1$ and $a$ are minimal elements in $F$, the interval $[1a,1a1a]$ contains a length $1$ MSI.  The presence of length $1$ MSIs containing weak descents is one of the key differences between the antichain and positive integer cases.

While the idea of a maximal outer factor still holds in the new context, simple examples reveal that the previous definition of a principal factor will not be sufficient for the forest case.  For example, suppose $F$ is disjoint union of two chains, $b\rightarrow a $ and $2\rightarrow 1$.  Then in the interval $[b,bb1]$, the chain $C$ with chain id $232$ has $C(bb1,b)$ as an MSI, even though $b$ is not a suffix of $bb1$. A quick analysis reveals that $b$ behaves like the principal factors from the previous section because the MSI results from a unimodal chain id which is the lexicographically first id leading to the prefix embedding of the word $b$.

Deeper analysis shows that words causing MSIs in this manner can have multiple embeddings in $w$.  The smallest example we could find of this is in the interval $[21a22, 21a221a22a22]$. In spite of the fact $21a22$ has three embeddings in the larger word, there is a maximal chain in which $C(21a221a22a22,21a22)$ is an MSI.  Fortunately, the definition of a principal factor can still be generalized to fit the new context so that once again, we will have exactly two types of MSIs in $F^*$.

Let $p$ be a word in $F^*$ and let $w$ be a word in $F^*$ that is not flat.  Suppose that $p$ is a prefix of $w$ with other embeddings in $w$, and no longer prefix containing $p$ has multiple embeddings in $w$.  Then there is a smallest index $i$, called the principal index of $p$ in $w$, such that $w(i)>p(i)$ and $w(i)$ is reducible.  We say $p$ is a \emph{principal factor} of $w$ if the word produced by reducing $w(i)$ contains only the prefix embedding of $p$.  

Notice this definition accounts for both of the examples given before it.  As in the previous cases, the principal index of a principal factor must take a value greater than $1$.  Before proceeding, it is important to understand why this definition includes both outer words and the Section 3 definition of a principal factor as special cases.  In the antichain case, an outer word $o(w)$ of a nonflat word is a principal factor when its principal index is $|w|$ because only $1$ and $|w|$ are reducible in the case of an antichain. Since reducing $|w|$ can only remove the suffix embedding of a word, we must have $o(w)\not\leq i(w)$ in order for it to be a principal factor, where $i(w)$ is the inner word.  This provides further insight into this condition of Bj\"orner's formula.  

To see that the definition of a principal factor from Section 3 is generalized by the new one, first note that a principal factor of an unrooted word $w$ cannot be rooted.  Indeed, if $w$ contains a nonminimal element, then the principal index of any rooted prefix $p$ will point to the first nonminimal element.  So reducing this element cannot possibly remove any embeddings of $p$.  Furthermore, the new definition guarantees $p$ is a suffix in the case of $\P$.  If $p$ has an embedding in $w$ which is neither prefix nor suffix, the word $p1$ would have a prefix embedding and another embedding, contradicting the maximality of $p$.  While it is more difficult to identify principal factors when they do not have a suffix embedding, this generalization clearly shows which properties of principal factors cause MSIs.

\begin{Prop}
\label{Fprincipal}
Suppose $u$ is a principal factor of $w$ with principal index $i$.  Let $C:w=v_0\stackrel{i}{\rightarrow}v_1\stackrel{l_2}{\rightarrow}\ldots\stackrel{l_{n-1}}{\rightarrow} v_{n-1}\stackrel{l_n}{\rightarrow}v_n=u$ be the lexicographically first chain in $[u,w]$ with $l_1=i$.  Then $C(u,w)$ is an MSI of $C$.
\end{Prop}

\begin{proof}
Using the same argument from the beginning of the proof of Proposition~\ref{gprincipal}, we conclude there exist chains that are lexicographically earlier than $C$.  Also from this proof, if $C'$ is an arbitrary maximal chain that is lexicographically earlier than $C$, then $l'_1<i$ and $v'_1$ must contain an embedding of $u$ which is not prefix.

Furthermore, since $C$ is the lexicographically first chain ending at the prefix embedding of $u$, we reduce the minimal letters at the end last, implying $v_{n-1}=um$ for some minimal element $m$.  If $v_1'$ contained $um$, $u$ would be contained in a longer prefix with multiple embeddings in $w$, contradicting the fact that $u$ is a principal factor of $w$.  Thus, the only words common to $C$ and $C'$ are $w$ and $u$.   Since $C'$ was an arbitrary maximal chain in $[u,w]$ with $l'_1<i$, and $C$ is the first maximal chain in $[u,w]$ with $l_1=i$, we conclude $C(u,w)$ is an MSI.
\end{proof}

If $u$ is not flat, the \emph{first admissible} chain ending at $\eta$, $C_{\eta}$, is the maximal chain whose chain id is the lexicographically first permutation of $M_{\eta}$ that is the chain id of a maximal chain.  This chain has the same structure it did in the case of $\P$.

\begin{Lem}
\label{Fflc}
Suppose $[u,w]\subset\P^*$ and $\eta$ is an embedding of $u$ into $w$.  Let $\ell$ be the index of the largest non-zero number in $\eta$.  If $u$ is not flat, $C_{\eta}$ has as its chain id the unique unimodal permutation of $M_{\eta}$ with decreasing suffix $|w|, |w|-1, \ldots, \ell+1$. \qed
\end{Lem}

The following Theorem completes the characterization of the MSIs and states they are once again caused by strong descents or principal factors.

\begin{Thm}
\label{FMSIchar}
Suppose $[u,w]\subset F^*$, $u$ is not the empty word, and $C:w=v_0\stackrel{l_1}{\rightarrow}v_1\stackrel{l_2}{\rightarrow}\ldots\stackrel{l_{n-1}}{\rightarrow} v_{n-1}\stackrel{l_n}{\rightarrow}v_n=u$ is a maximal chain in $[u,w]$.  Then $C(v_i,v_j)$ is an MSI of $C$ if and only if $C(v_i,v_j)$ consists of a strong descent, meaning $v_i\neq v_j mn$ for any minimal elements $m$ and $n$, or $v_{j}$ is a principal factor of $v_i$, $l_{i+1}$ is the principal index of $v_j$ with respect to the embedding $\eta_{v_j}$,  and $C[v_i,v_j]$ is the first admissible chain in $[v_j,v_i]$ ending at the prefix-embedding of $v_j$.
\end{Thm}

\begin{proof}
The reverse implication follows from Propositions~\ref{Fdescent} and~\ref{Fprincipal}.

Suppose $C(v_i,v_j)$ is an MSI of $C$.  By the definition of a poset lexicographic order, $C(v_i,v_j)$ is an MSI of $C$ if and only if it is an MSI of $C[v_i,v_j]$ in the interval $[v_j,v_i]$. Thus, it suffices to consider the case $w=v_i$ and $u=v_j$.  

Since $|u|=|w|$ implies $[u,w]$ is direct product of chains, Corollary~\ref{msisame} can be translated to the forest case.  This implies strong descents are the only MSIs when $|u|=|w|$. However, strong descents are always length 1 MSIs.  Therefore, it suffices to consider the case when $|u|<|w|$ and $C(w,u)$ is an MSI of $C$ that does not consist of a strong descent.  Our first goals are to establish that $w$ has a prefix embedding of $u$ and  $C[w,u]$ is the first admissible chain ending at the prefix embedding.

Let $\eta$ be the embedding of $u$ into $w$ at the end of the chain $C$.  Note that any descent $v_k$ in $C$ is a weak descent since otherwise, by Proposition~\ref{Fdescent}, $C(v_{k-1},v_{k+1})$ would be an MSI, contradicting the minimality of $C(w,u)$.  If $k+1\neq n$, this forces $l_{k+2}<l_{k+1}=l_k-1$,  implying that $v_{k+1}$ is also a weak descent.  By continually applying this idea, we find that any descents contained in $C(w,u)$ occur in a single sequence of weak descents at the end of this interval, and the corresponding labels form a decreasing sequence of consecutive numbers.  

To see $\eta$ is the prefix embedding of $u$ into $w$, suppose for a contradiction that the set $M_{\eta}$ contains a $1$.  Since any descents occur in a sequence at the end of the interval, it follows that $l_1=1$ or $l_n=1$.  If $l_n=1$, then since $v_{n-1}$ is a weak descent, $u$ is the empty word.  This contradicts our assumptions.  Suppose $l_1=1$.  Then any chain $C'$ lexicographically earlier than $C$ has $l'_1=1$ and thus contains $v_1$, contradicting the fact that $C(w,u)$ is an MSI.  Therefore, $\eta$ is the prefix embedding of $u$ into $w$, allowing us to write $\eta=u$.

Since $|u|<|w|$, we must reduce position $|w|$ by the end of the chain. Note that the label $|w|$ can only be followed by another $|w|$ or the sequence of labels $|w|-1,\ldots,|u|+1$, which leads to a sequence of weak descents. It follows that $C$ must contain the word $um$ for some minimal element $m$.  Since $l_n=|u|+1$, $C$ reduces the $m$ at the end of $v_{n-1}=um$ to get $v_n=u$. This implies $um$ (and hence $w$) is not flat.  So by Lemma~\ref{Fflc}, $C[w,um]$ is the first admissible chain ending at the prefix-embedding of $um$.  Since minimal elements can only be removed from the front or back of a word, $C$ is the first admissible chain ending at the prefix-embedding of $u$. 

Since $C(w,u)$ is an MSI, $C$ cannot be the lexicographically first chain in $[u,w]$.  Therefore, $w$ must contain another embedding of $u$ in addition to the prefix embedding.  This implies that $u$ has a principal index in $w$.

Next, we will show $l_1$ is the principal index of $u$ in $w$.   Since $C$ is the first admissible chain ending at the prefix embedding, $w(l_1)$ is the first letter that is reducible and satisfies $w(l_1)>u(l_1)$.

We also need to show that the word $v_1$ contains only the prefix embedding of $u$. For a contradiction, suppose $v_1$ contains another embedding $\rho$ of $u$ besides the prefix embedding. Then there is a chain $C'$ ending at $\rho$ whose chain id begins $l_11\ldots$ and has $v'_1=v_1$.  This chain is thus lexicographically earlier than $C$ and satisfies $C-C(v_1,u)\subset C'$, contradicting the fact that $C(w,u)$ is an MSI.  So $v_1$ contains only the prefix embedding.

It remains to show that there is not a longer prefix containing $u$ with another embedding in $w$.  For a contradiction, suppose there is a longer prefix of $w$ containing $u$ that has multiple embeddings in $w$.  Then $um$ is a prefix of $w$ for some unique minimal element $m$ and has another embedding in $w$.  Let $C'$ be the lexicographically first chain in the interval containing $um$.  Note $C$ also contains $um$ because as the first admissible chain ending at the prefix embedding, $l_n=|u|+1$.  Thus, $C-C(w,um)\subset C'$, implying $C(w,um)$ is a skipped interval.  This contradicts the fact that $C(w,u)$ is an MSI.
\end{proof}

We can now easily describe the critical chains that consist entirely of strong descents.  Since the proof is very similar to that of Proposition~\ref{alldes}, we omit it.

\begin{Prop}
\label{Falldes}
Suppose $[u,w]\subset\P^*$, $(u,w)$ is non-empty, and $w$ is not flat.  Suppose $\eta$ is an embedding of $u$ into $w$.  Then there is a critical chain $C$ in $[u,w]$ ending at $\eta$ that consists entirely of strong descents if and only if $w(i)=\eta(i)$ or $w(i)\rightarrow\eta(i)$ for all $i$, $\eta(2)\neq\hat{0}$, and $\eta(|w|-1)\neq\hat{0}$.  Furthermore, these conditions imply $|w|-|u|\leq2$ and that $[u,w]$ has at most two critical chains consisting entirely of strong descents.\qed
\end{Prop}

It should be noted that in the antichain case, the fact that $\mu(i(w),w)=1$ for the inner word $i(w)$ when $w$ is not flat follows directly from this proposition.

As in the case of $\P$, the remaining critical chains must contain at least one MSI caused by a principal factor.  So these chains contain a principal factor  of $w$ or a principal factor $p_{v_i}$ of some $v_i$ with the property that $C(w,v_{i+1})$ consists entirely of strong descents.

In order to have a critical chain in $[u,w]$ involving an MSI resulting from a principal factor $p_{v_i}$, $p_{v_i}$ needs to be contained in a different MSI or $p_{v_i}$ must equal $u$. By Theorem~\ref{FMSIchar}, $p_{v_i}$ could be contained in one of three types of MSI: an MSI caused by a strong descent, an overlapping MSI caused by a principal factor, or an adjacent MSI caused by a principal factor.  This is where the forest case becomes more complex than the case of $\P$ because the third possibility can happen.  However, it can only happen when $v_i$ is a rooted word.

Moving forward, we will have few results that apply to both rooted and unrooted words.  However, unrooted words still behave much like they did in the positive integer case, and our proof of Bj\"orner's formula will help us understand rooted words in the new context.

\begin{Prop}
\label{Fpfover}
Suppose $C(v_i,p_{v_i})$ is an MSI of a critical chain $C$ caused by the principal factor $p_{v_i}$.  If $v_i$ is unrooted, then $p_{v_i}$ is unrooted and is contained in an MSI caused by a strong descent or an overlapping MSI caused by a principal factor.
\end{Prop}
\begin{proof}
If $p$ is rooted prefix of an unrooted rooted $v_i$, then the principal index of $p$ in $v_i$ contains a nonminimal element.  So reducing the letter at the principal index cannot eliminate any embedding of $p$.  Thus, a rooted word cannot be a principal factor of an unrooted word.

Let $v_{j+1}=p_{v_i}$.  Suppose for a contradiction that $C(v_j,p_{v_j})$ is an MSI caused by a principal factor.  Since $C(v_i,v_{j+1})$ is an MSI caused by a principal factor, Theorem~\ref{FMSIchar} implies $C[v_i,v_{j+1}]$ is the first admissible chain ending at the corresponding embedding.  So by Lemma~\ref{Fflc}, $v_{j}=v_{j+1}m$ for some minimal element $m$.  However, this implies the value at the principal index $l_{j+1}=|v_j|$ of $p_{v_j}$ in $v_j$ is $m$. But then $v_j$ and $v_{j+1}=p_{v_i}$ are rooted words because the principal index of an unrooted principal factor must contain a nonminimal letter.  Since $p_{v_i}$ is not rooted, this is a contradiction.
\end{proof}

As stated above, this result does not apply to rooted words.  For example, in $[1,1a1a]$, the rooted principal factor $1a$ is contained in the adjacent MSI caused by the rooted principal factor $1$ of $1a1$.

The previous result essentially separates the unrooted words from the rooted ones.

\begin{Cor}
\label{notrooted}
If $w$ is not a rooted word and $C(w,u)$ is a critical maximal chain of $[u,w]$, then $C(w,u)$ has no MSIs caused by rooted principal factors.
\end{Cor}
\begin{proof}Suppose for a contradiction $C(w,u)$ did contain an MSI caused by an rooted principal factor.  By Proposition~\ref{Fpfover}, a principal factor of an unrooted word is also not rooted.  Thus, the last unrooted word in $C(w,u)$ must be a strong descent $v_k$.  Since no rooted word is a principal factor of $v_k$, $v_{k+1}$ must also be a strong descent for it to be contained in an MSI.  This implies $l_{k+2}=1$, so that $v_{k+2}$ cannot be a strong descent MSI.  However, the principal index of a principal factor cannot be $1$.  Thus, $v_{k+2}$ cannot be in an MSI caused by a principal factor.  So by Theorem~\ref{FMSIchar}, if $u\neq v_{k+2}$,  $v_{k+2}$ is not in any MSI, contradicting the fact that $C(w,u)$ is a critical maximal chain.  If $u=v_{k+2}$, then we have shown $C(w,u)$ has no MSIs caused by rooted principal factors.
\end{proof}

This allows us to conclude that if $w$ is not rooted, any overlapping MSIs in the set $I(C)$ of a critical chain must come in pairs.

\begin{Prop}
\label{Foverlap}
If $w$ is not rooted, an MSI of a critical chain can overlap with at most one other MSI.
\end{Prop}
\begin{proof}
This proof is entirely analogous to that of Proposition~\ref{overlap}, with two notable exceptions.  First, we need to point out that by Corollary~\ref{notrooted}, all words involved in the overlapping intervals $C(v_{i_1},v_{j_1}),  C(v_{i_2},v_{j_2}),\ldots, C(v_{i_k},v_{j_k})$ are not rooted.  Second, this implies each $\eta_{v_{i_\ell}}(l_{i_\ell+1})$ is nonminimal because $l_{i_\ell+1}$ corresponds to a principal index of an unrooted word.
\end{proof}

Notice again that this restriction on overlapping MSIs does not apply to rooted words.  For example, in the last maximal chain of $[a,a11aa11a]$, the word $a11aa$ is contained in three MSIs.  Nevertheless, this chain is critical.

At this point, it is clear that the $J(C)$ structure of the critical chains is essentially the same as in section 3 when a word is not rooted, while the forest case reduces to the ordinary factor order case of section 2 when a word is rooted.  To establish the formula,  we will also need to update the definitions of primary prefixes and the $\nu$ function, and use them to establish the formula when $w$ is not rooted.  We have little choice but to establish the formula separately for rooted words.

\begin{Thm}
\label{FJCstruc}
Suppose $w$ is not a rooted word.  Let $C:w=v_0\stackrel{l_1}{\rightarrow}v_1\stackrel{l_2}{\rightarrow}\ldots\stackrel{l_{n-1}}{\rightarrow} v_{n-1}\stackrel{l_n}{\rightarrow}v_n=u$ be a maximal chain of $[u,w]$.  Then $C$ is a critical chain if and only if $C(w,u)$ can be written as a sequence of intervals $$C[v_{i_1}=v_1,v_{i_2})\cup C[v_{i_2},v_{i_3}) \cup \ldots \cup C[v_{i_{k-1}},v_{i_k}=u)$$ where each interval $C[v_{i_j},v_{i_{j+1}})$ is one of the following three types:
\begin{enumerate}
\item $C[v_{i_j},v_{i_{j+1}})$ is an MSI caused by the strong descent $v_{i_{j}}$.
\item $C[v_{i_j},v_{i_{j+1}})$ is an MSI caused by the principal factor $v_{i_{j+1}}$ of the word $v_{i_j-1}$.
\item The word $v_{i_{j+1}}$ is a principal factor of a word in  $C[v_{i_{j-1}},v_{i_{j}})$ and satisfies $v_{i_{j+1}}m_1\ldots m_k=v_{i_{j}}$, where each $m_i$ is a minimal element and $k=|v_{i_{j}}|-|v_{i_{j+1}}|>0$. The value $k$ is unique in the sense that no other word satisfies the description of $v_{i_{j+1}}$ for another value $k$.
\end{enumerate}
Furthermore, type (1) intervals are followed by intervals of type (1) or (2), type (2) intervals are followed by intervals of type (1) or (3), and type (3) intervals are followed by intervals of type (1).  Finally, only intervals of type (1) or (2) can begin the decomposition.
\end{Thm}

\begin{proof}
Since $w$ is not rooted, by Corollary~\ref{notrooted}, no rooted word can be a principal factor which causes an MSI. 

Thus, both implications can be proved as they were in the proof of Theorem~\ref{JCstruc} by updating the relationship between between $v_{i_j}$ and $v_{i_{j+1}}$.  For example, in the forward implication, we have $v_{i_j}=v_{i_{j+1}}m_1\ldots m_k=v_{i_{j}}$, where each $m_i$ is a minimal element and $k=|v_{i_{j}}|-|v_{i_{j+1}}|>0$, instead of $v_{i_j}=v_{i_{j+1}}1^m$ for $m=|v_{i_{j}}|-|v_{i_{j+1}}|$.
\end{proof}

Since Theorem~\ref{FJCstruc} gives essentially the same result as Theorem~\ref{JCstruc}, when $w$ is not rooted, we need only update the appropriate definitions to get the desired formula.  However, we will need to handle the case when $w$ is rooted separately before stating the formula.

Define a word $v$ to be a \emph{base of $w$} if $v(j)=w(j)$ or $w(j)\rightarrow v(j)$ for all $j$ and $|v|=|w|$ or $|w|-1$.  Define the \emph{degree} of a base to be the number of indices $j$ for which $w(j) \rightarrow v(j)$.

Suppose $v_i$ is a base of $w$ of degree $i$.  Let $l$ be the index of the smallest position satisfying $w(l) \rightarrow v_i(l)$, or $|w|+1$ if $i=0$. We define any word $p_{v_i}$ that is a principal factor of $v_i$ and whose principal index takes a value less than $l$ to be a \emph{principal factor of $w$ of degree $i$}. As in the previous section, when a degree is not noted in the language or the notation, the assumption will be that the principal factor has degree $0$.  

Define $w\setminus k$ to be the word that results when $k$ minimal elements are removed from the suffix of $w$, or as undefined if $w$ ends in less than $k$ minimal elements.  Note that unlike the integer case when $1$ was the only minimal element, there are multiple minimal elements which could be reduced.

Suppose $p_{v_i}$ is a principal factor of $v_i$.  Let $v_i'$ be the word that results when the letter in $v_i$ at the principal index of $p_{v_i}$ is reduced by $1$ rank.  That is, $v_i(j)\rightarrow v_i'(j)$ when $j$ is the principal index of $p_{v_i}$ and $v_i'(j)=v_i(j)$ for all other indices $j$.  Define the \emph{primary prefix} $x({p_{v_i}})$ of a principal factor $p_{v_i}$ to be the longest proper prefix that has at least 2 embeddings in $v_i'$ and satisfies $x({p_{v_i}})=p_{v_i}\setminus k$ for some $k$. So if no $k>0$ satisfies the restriction, or no such word has at least two embeddings in $v_i'$, the primary prefix is undefined

Notice the primary prefix definition still makes sense when $p_{v_i}$ is rooted, in which case $i=0$ and $p_{v_0}$ is an outer word not contained in the inner word.  In this case, $x({p_{v_0}})=p_{v_0}\setminus 1$.  As was seen in our proof of Bj\"orners result, this implies that whenever $u\leq x({p_{v_0}})$, $p_{v_0}$ is contained in a length 1 MSI.  

The following proposition asserts that the primary prefix is the only word that can cause a type (3) interval after the type (2) interval $C(v_i,p_{v_i})$.  While the spirit of the proof is similar to that of Proposition~\ref{pmpfx}, the definition of a primary prefix has changed enough to warrant stating the entire proof.

\begin{Prop}
\label{Fpmpfx}
Suppose $w$ is not rooted.  Let $C$ be a maximal chain of $[u,w]$ and suppose $C(v_i,p_{v_i})$ is a type (2) interval in the set $J(C)$.  Then $C(v_i,p_{v_i})$ is followed by a type (3) interval $C[p_{v_i},x)$ in $J(C)$ if and only if $x$ is the primary prefix of $p_{v_i}$ and $x$ appears in $C$.
\end{Prop}
\begin{proof} First suppose $x$ is the primary prefix of $p_{v_i}$ and that it appears in $C$.  Let $v_\ell$ be the last word in the interval $C(v_i,p_{v_i})$ that contains at least two embeddings of $x$.  We will show $x$ is a principal factor of $v_\ell$ with principal index $l_{\ell+1}$.  Since $p_{v_i}$ is a prefix of $v_{\ell+1}$, $x$ is as well.  So $v_{\ell+1}$ only contains the prefix embedding of $p_{v_i}$.  Furthermore, since  $C[v_i,p_{v_i}]$ is the first admissible chain in $[p_{v_i},v_i]$ ending at the prefix embedding of $p_{v_i}$, for all $l_{i+1}<k<l_{\ell+1}$, either we have $v_\ell(k)=p_{v_i}(k)$ or we have $v_\ell(k)$ minimal and $p_{v_i}(k)=0$.  Therefore, for all $l_{i+1}<k<l_{\ell+1}$, either we have $v_\ell(k)=x(k)$ or we have $v_\ell(k)$ minimal and $x(k)=\hat{0}$.  This implies $C[v_\ell,x]$ is the first admissible chain in $[x,v_\ell]$ ending at the prefix embedding of $x$ and $v_\ell(l_{\ell+1})$ is the first reducible letter in $v_\ell$ greater than the corresponding position in the prefix embedding of $x$.  Thus, $l_{\ell+1}$ satisfies the definition of a principal index, and since $v_{\ell+1}$ only contains the prefix embedding of $p_{v_i}$, $x$ is a principal factor of $v_\ell$ and $C(v_\ell,x)$ is an MSI of $C$.  Since $x$ is the longest proper prefix of $p_{v_i}$ with two embeddings in  $v_i'$ that satisfies $x=p_{v_i}\setminus k$ for some $k$, and $C(v_i,p_{v_i})$ is a type (2) interval, no word between $p_{v_i}$ and $x$ can be a principal factor of a word $v_k$ in $C$.  Therefore, $C(v_\ell,x)$ is reduced to the type (3) interval $C[p_{v_i},x)$ in $J(C)$, completing the reverse implication.

Now suppose that $C(v_i,p_{v_i})$ is followed by a type (3) interval $C[p_{v_i},x)$ in the set $J(C)$.  Then by Theorem~\ref{FJCstruc}, $x$ is a principal factor of a word $v$ in $C(v_i,p_{v_i})$ and $x=p_{v_i}\setminus k$ for some $k$.  From the proof of Theorem~\ref{FJCstruc}, we know $|v|=|v_i'|$ and $v$ has at least two embeddings in $v_i'$.  By Theorem~\ref{FJCstruc} part (3), it suffices to show that no longer prefix of $p_{v_i}$ containing $x$ has two embeddings in $v_i'$.  For a contradiction, suppose $y$ is such a prefix.  Then $y=xm_1\ldots m_\ell$ has at least two embeddings in $v_i'$, and by the argument in the paragraph above, is a principal factor of some word in $C(v_i,p_{v_i})$.  Thus, $C(v,y)$ would be an MSI in $I(C)$.  This would be reduced to the interval $C[p_{v_i},y)$, which is contained in $C[p_{v_i},x)$, implying that $C[p_{v_i},x)$ could not be in $J(C)$.  This is a contradiction.  Thus, $x$ is the longest proper prefix of $p_{v_i}$ with two embeddings in $v_i'$, implying it is the primary prefix of the word $p_{v_i}$. 
\end{proof}

Let $\mu(u,v)$ be the normal M\"obius function if $u$ and $v$ are both elements of $F^*$, or zero if either is undefined.  Define the function $\nu(u,v)$ to be 
$$\nu(u,v)=\displaystyle\sum_{i\geq0} \mu(u,v\setminus i).$$
Notice all the terms in the summation will be zero beyond the largest value $i=k$ for which $v\setminus k$ is defined, or the smallest value $i=k$ for which $v\setminus k\leq u$.

We are now ready to state the contribution of critical chains whose first type (2) interval is a specific interval.  By replacing the words $p_{v_i}\setminus 1^k$ in the proof of Proposition~\ref{type2} by $p_{v_i}\setminus k$, it is easy to adapt that proof to work for the next proposition.

\begin{Prop}
\label{Ftype2}
Suppose $w$ is not rooted and $C(v_i,p_{v_i})$ is the first type (2) interval of some critical chain $C$ of $[u,w]$.  Then $p_{v_i}$ is a principal factor of $w$ of degree $i$ and $v_i$ is based in $C$.  Furthermore, the contribution to the M\"obius value $\mu(u,w)$ of all critical chains in $[u,w]$ that have $C(v_i,p_{v_i})$ as the first type (2) interval  is
$$\displaystyle (-1)^i\left(\nu(u,p_{v_i})-  \nu(u,x(p_{v_i}))\right),$$
where $x(p_{v_i})$ is the primary prefix of $p_{v_i}$.\qed
\end{Prop}

Our next goal is to consider the intervals $[u,w]$ for $w$ rooted.  Note that when $w$ is rooted, the forest case reduces to the antichain case.  Therefore, the formula for $\mu(u,w)$ is given by Theorem~\ref{muofo}, which the reader may wish to refresh at this time.  Thus, we must show this theorem is now a special case of our formula from Section 3.  The key fact in the proof is that the primary prefix of a rooted principal factor $p_{v_i}$ is $p_{v_i}\setminus 1$, which was also the key in proving Bj\"orner's formula using discrete Morse theory.

Let $\rho$ denote the rank function in $F^*$, and $\rho(u,w)=\rho(w)-\rho(u)$.  For simplicity, let $0\leq t\leq2$ be the number of critical chains in $[u,w]$ that consist entirely of strong descents, and define
$$d(u,w)=\begin{cases}
t(-1)^{\rho(u,w)} & \mbox{if } \rho(u,w)>1\\
(-1)^{\rho(u,w)} & \mbox{if } \rho(u,w)\leq1.\\
\end{cases}$$

\begin{Prop}
\label{Muord}
Suppose $u\leq w$ in the poset $F^*$ and $w$ is a rooted word.  Then 
$$\displaystyle\mu(u,w)=d(u,w)+\sum(-1)^i\left(\nu(u,p_{v_i})-\nu(u,x(p_{v_i}))\right)=d(u,w)+\mu(u,o(w)),$$
where $o(w)$ is the outer word as defined on page 1 and the sum is over all triples $v_i,p_{v_i},x(p_{v_i})$ such that $p_{v_i}$ is a principal factor of $w$ of degree $i$ with base $v_i$ and primary prefix $x(p_{v_i})$.
\end{Prop}

\begin{proof} First we show rooted words can only have principal factors of degree $0$.  Since any rooted strong descent $v_i$ satisfies $v_{i+1}=mv_{i-1}n$ for some minimal elements $m$ and $n$, and every principal factor of $v_i$ has principal index $|v_i|\neq1$, rooted words do not have bases of positive degree. Thus, they cannot have principal factors of positive degree either.

We need to show the formula in Theorem~\ref{muofo} agrees with the one in the statement of this proposition.  Suppose $|w|-|u|\leq2$ and $u\neq o(w)$.  Then by Proposition~\ref{Falldes} and Theorem~\ref{muofo},
$$\mu(u,w)=d(u,w).$$
Furthermore, by the definition of principal factor, the only word that can be a principal factor is the outer word $o(w)$.  Note $|w|-|o(w)|>1$ when $o(w)$ is not flat.  So $u\not\leq o(w)$, implying the summation is $0$ because $\nu(u,o(w))=0$ and $\mu(u,o(w))=0$ by definition. This completes the proof in this case.

Next, suppose $|w|-|u|> 2$, $u\leq o(w)\setminus 1$, and $o(w)\not\leq i(w)$.  Note $o(w)$ is the unique principal factor of $w$.  From the discussion following the definition of the primary prefix, we know $o(w)\setminus 1$ is the primary prefix of $o(w)$.  Furthermore, by Proposition~\ref{Falldes}, $d(u,w)$=0 because $|w|-|u|>2$.  Thus,
\begin{align*}
\sum(-1)^i\left(\nu(u,p_{v_i})-\nu(u,x(p_{v_i}))\right)=&\nu(u,o(w))-\nu(u,x(o(w)))\\
=&\mu(u,o(w))+\nu(u,o(w)\setminus 1)-\nu(u,o(w)\setminus 1)\\
=&\mu(u,o(w)),
\end{align*}
completing the proof in this case.

Suppose $|w|-|u|\geq 2$, $u\leq o(w)$, $u\not\leq o(w)\setminus 1$ and $o(w)\not\leq i(w)$.  Note $o(w)$ is the unique principal factor of $w$, but the primary prefix of $o(w)$, $o(w)\setminus 1$, does not contain $u$.  Furthermore, by Proposition~\ref{Falldes}, $d(u,w)$=0 because $|w|-|u|\geq 2$ and $u\neq i(w)$.  Thus, 
\begin{align*}
\sum(-1)^i\left(\nu(u,p_{v_i})-\nu(u,x(p_{v_i}))\right)=&\nu(u,o(w))-0\\
=&\mu(u,o(w))+\nu(u,o(w)\setminus 1)\\
=&\mu(u,o(w)),
\end{align*}
completing the proof in this case.

Finally, in all other cases, $|w|-|u|> 2$ and $o(w)\leq i(w)$.  Since $|w|-|u|> 2$, $d(u,w)=0$.  Since $o(w)\leq i(w)$, $o(w)$ is not a principal factor of $w$, meaning the summation is $0$ as well.  This agrees with Theorem~\ref{muofo}, which states $\mu(u,w)=0$ in all other cases, completing the proof.
\end{proof}

Using Propositions~\ref{Falldes},~\ref{Ftype2}, and~\ref{Muord} we are able to write down a formula for $\mu(u,w)$ for $u\leq w$ in $F^*$.  Thus, this formula applies to all posets $P^*$ ordered by generalized factor order in which each element of the base poset $P$ covers a unique element.

\begin{Thm}
\label{FMu2}
Suppose $u\leq w$ in the poset $F^*$.  Then 
$$\displaystyle\mu(u,w)=d(u,w)+\sum(-1)^i\left(\nu(u,p_{v_i})-\nu(u,x(p_{v_i}))\right),$$
where the  sum is over all triples $v_i,p_{v_i},x(p_{v_i})$ such that $p_{v_i}$ is a principal factor of $w$ of degree $i$ with base $v_i$ and primary prefix $x(p_{v_i})$.
\end{Thm}

\begin{proof} 
If $w$ is rooted, the result follows from Proposition~\ref{Muord}.  

If $w$ is unrooted, the proof of the desired result is an easy adaptation of the proof of Theorem~\ref{Mu2}.
\end{proof}
%---------------------------------------------------------------------------------------------------------------

\section{Future Research and Open Problems}
\label{open}

\subsection{Generalizing and Simplifying this Formula}

As noted at the end of Section 3, many of the coefficients in our formula for generalized factor order on the integers are zero.  This phenomenon is even more pronounced in the rooted forest case.  Since our formula simplifies considerably in the case of rooted words, it is natural to wonder whether the general formula can be simplified as well.

Should this formula be simplified, it will likely be done in one of two ways.  It may be possible to find a formula which applies to a more general class of posets ordered by generalized factor order.  Sagan and McNamara~\cite{MS11} were able prove a formula that works for any poset ordered by generalized subword order using a different approach than the one in~\cite{SV06}.  Thus, it is possible a different approach could result in a more general formula, which may be simpler than the one given here.

To investigate generalized factor order on other posets $P^*$, one needs to consider words which cover multiple elements, complicating the poset lexicographic order we used in this investigation.  One way to resolve these complications is to place an order on the children of each element of $P$.  This would generalize our current chain ids by including a subscript on the label indicating which child each letter is reduced to.  In the context of $F^*$, such subscripts would always be $1$ because each element has a unique child, making this new type of chain id a clear generalization of the current one. Initial data suggests it is worth pursuing this line of thought to see if a formula can be found in more general cases.

It may also be possible to simplify the formula at the level of critical simplices.  In particular, it is often the case that critical simplices of dimension $d$ and $d+1$ are present when a coefficient of $\nu(u,p_{v_i})$ is zero.  This leads one to suspect it may be possible to use the discrete analogue of ``The First Cancellation Theorem'' from smooth Morse Theory to cancel critical simplices.

However, it is not clear whether reducing the number of critical simplices in such a manner would result in a simplified formula.  For example, it is often the case that a principal factor $p$ will have have unique minimal degree base $b_i$ (of degree $i$) and maximal base degree base $b_j$ (of degree $j$) such that if $b_k$ (of degree $k$) is another base, then $i<k<j$.  When this happens, a binomial sum zeros out the coefficient of $\nu(u,p)$. But this is not always the case - the smallest counterexample we found was the word $u=2111222$ in the interval $[2111222,2111222112221222]$, in which $u$ occurs as a principal factor of degree $1$ twice, degree $2$ once, but is not a principal factor of degree $0$ because $21112221111$ is a longer outer factor.  We were not able to reconcile the previous observation with this exception to it in a desirable manner.

\subsection{The Topology of Posets ordered by Generalized Factor Order.}

Besides being useful results in proving the formula given for the M\"obius function of $\P^*$ and $F^*$, Theorems~\ref{JCstruc} and~\ref{FJCstruc} can be used to get a detailed description of the critical simplices of intervals in $\P^*$ and $F^*$.  Thus, it is a first step into investigating the homotopy type of these posets.  The next step is again checking whether there are critical simplices of dimension $d$ and $d+1$ which cancel each other out.

\subsection{The Consecutive Pattern Poset and Ordinary Factor Order}

In a paper submitted to the arXiv in 2011~\cite{BFS11}, Bernini, Ferrari and Steingr\'imsson calculate the M\"obius function of the consecutive pattern poset.  Let $S_d$ be the set of all permutations of the first $d$ positive integers. A consecutive pattern $\sigma=a_1a_2\ldots a_k$ appears in a permutation $\tau=b_1b_2\ldots b_n$ if the letters of some subsequence $b_ib_{i+1}\ldots b_{i+k-1}$ of $\tau$ appear in the same order of size as the letters in $\sigma$.  The consecutive pattern poset is $\cup_{d\geq0}S_d$ ordered with respect to consecutive pattern containment.

One of their results is a formula for M\"obius function of the consecutive pattern poset which has many similarities to Bj\"orner's formula for the M\"obius function of ordinary factor order.  This suggests there may be some common generalization of these posets, and this is the topic of a paper of Sagan and Willenbring~\cite{SW11}.

\appendix
\section{The Matching of Babson and Hersh}
\label{BHmatch}

In~\cite{BH05}, Babson and Hersh give the acyclic matching of simplices in $\Delta(u,w)$ based on whether the sets $I(C)$ and $J(C)$ cover $C$. Although knowledge of this matching is not required to apply Theorem~\ref{chainmu}, we record the matching below for completeness.

The Matching:

$\bullet$ If $I(C)$ does not cover $C(w,u)$, let $\rho_0$ be the lowest rank (that is, the last) vertex not covered by $I(C)$.

Match each new simplex $h$ with $h\bigtriangleup\{\rho_0\}$, where $\bigtriangleup$ is the symmetric difference operator (that is, $h\setminus\{\rho_0\}$ if $\rho_0\in h$ and $h\cup\{\rho_0\}$ if $\rho_0\not\in h$.)  This matching matches all new simplices of $C$.

Note that $h\bigtriangleup\{\rho_0\}$ is always in $ C\setminus(\bigcup_{C'<C} C')$ because the inclusion/exclusion of $\rho_0$ does not affect whether the simplex hits every MSI.  Also, this matching works for the very first maximal chain since the empty set is consider a simplex in $\Delta(u,w)$.
\vskip1\baselineskip

$\bullet$ Otherwise, $I(C)$ covers $C$ and we base the matching on $J(C)=\{J_1,\ldots,J_r\}$.  Let $\rho_i$ be the lowest rank vertex of $J_i\in J(C)$.  Let $J_{r+1}=$ the set all vertices not in $J(C)$, and let $\rho_{r+1}$ be the lowest rank vertex in $J_{r+1}$.  Define a map $\tau$ that associates an integer with each new simplex $h$ based on the first set $J_i$ which $h$ intersects in more than the lowest rank element.  That is,
\begin{align*}
\tau: C\setminus(\bigcup_{C'<C} C') &\rightarrow [r]\cup\{\infty\}\\
h&\mapsto \min_{1\leq i\leq r} \{i|h\cap J_i\neq\{\rho_i\}\},
\end{align*}
setting $\tau h=\infty$ when $h$ intersects each $J_i$ in exactly the lowest rank element (that is, when the set $\{i|1\leq i\leq r, h\cap J_i\neq\{\rho_i\}\}$ is empty.)

First, if $\tau h\neq \infty$, match the new simplex $h$ with $h\bigtriangleup\{\rho_{\tau h}\}$. This matches all simplices for which $\tau h\neq \infty$.

If $J_{r+1}=\emptyset$, then $J(C)$ covers $C$ and there is one simplex satisfying $\tau h= \infty$.  This simplex contains each $\rho_i$.  Thus,  there is one simplex unmatched and it is a critical simplex of this matching.  

If $J_{r+1}\neq\emptyset$, then $J(C)$ does not cover $C$ even though $I(C)$ does.  In this case, we match each simplex $h$ satisfying $\tau h= \infty$ with $h\bigtriangleup\{\rho_{r+1}\}$.  This matches all simplices for which $\tau h= \infty$.

The matching is more complicated when $I(C)$ covers $C$ but $J(C)$ does not because $\rho_{r+1}$ is in $I(C)$.  Thus, removing it from a simplex does not guarantee every MSI is still hit, meaning that $h\setminus\{\rho_{r+1}\}$ might not be in $C\setminus(\bigcup_{C'<C} C')$.  To guarantee $h\setminus\{\rho_{r+1}\}$ is a new simplex, we need the lowest rank element of each $J_i$, $\rho_i$, in $h$.  Indeed, any interval from $I(C)$ reduced and not included in $J(C)$ intersects some interval $J_i$ in at least the element $\rho_i$.  This issue led to a mistake in the published version of Babson and Hersh's article, but this mistake is corrected in newer versions.

Please refer to Tables~\ref{table:MSIs} and~\ref{table:JC} for the context of the below examples.

$\bullet$ $I(C)$ does not cover $C$.  Example: Chain $1-2-5-3$ in the interval $[b,bbabb]$.  Match new simplices based on inclusion/exclusion of vertex $abb$.  That is, match $ab$ with $abb-ab$ and $babb-ab$ with $babb-abb-ab$.

$\bullet$ $J(C)$ covers $C$.  Example:  Chain $5-4-3-1$ in $[b,bbabb]$.  Match 2 of 3 new simplices, based on inclusion/exclusion of vertex $bba$.  That is, $bba-bb$ is an unmatched, critical simplex, while $bbab-bb$ is matched with $bbab-bba-bb$.

$\bullet$ $I(C)$ covers $C$, but $J(C)$ does not.  Example: Chain $6-5-4-3-2$ in $[a,abbabb]$.  We match all new simplices based on the matching rules above.  In particular, match $abbab-abb$ with $abbab-abba-abb$, $abbab-abb-ab$ with $abbab-abba-abb-ab$, and $abba-abb$ with $abba-abb-ab$.

\bibliographystyle{acm}
\bibliography{factorOrder}

\begin{thebibliography}{10}

\bibitem{BH05}
{\sc Babson, E., and Hersh, P.}
\newblock Discrete {M}orse functions from lexicographic orders.
\newblock {\em Trans. Amer. Math. Soc. 357}, 2 (2005), 509--534 (electronic).

\bibitem{BFS11}
{\sc Bernini, A., Ferrari, L., and Steingr\'imson, E.}
\newblock The {M}\"obius function of the consecutive pattern poset.
\newblock Preprint {\texttt{arXiv:1103.0173}}.

\bibitem{aB90}
{\sc Bj{\"o}rner, A.}
\newblock The {M}\"obius function of subword order.
\newblock In {\em Invariant theory and tableaux ({M}inneapolis, {MN}, 1988)},
  vol.~19 of {\em IMA Vol. Math. Appl.} Springer, New York, 1990, pp.~118--124.

\bibitem{aB93}
{\sc Bj{\"o}rner, A.}
\newblock The {M}\"obius function of factor order.
\newblock {\em Theoret. Comput. Sci. 117}, 1-2 (1993), 91--98.
\newblock Conference on Formal Power Series and Algebraic Combinatorics
  (Bordeaux, 1991).

\bibitem{rF95}
{\sc Forman, R.}
\newblock A discrete {M}orse theory for cell complexes.
\newblock In {\em Geometry, topology, \& physics}, Conf. Proc. Lecture Notes
  Geom. Topology, IV. Int. Press, Cambridge, MA, 1995, pp.~112--125.

\bibitem{rF02}
{\sc Forman, R.}
\newblock A user's guide to discrete {M}orse theory.
\newblock {\em S\'em. Lothar. Combin. 48\/} (2002), Art.\ B48c, 35 pp.
  (electronic).

\bibitem{MS11}
{\sc McNamara, P. R.~W., and Sagan, B.~E.}
\newblock The {M}\"obius function of generalized subword order.
\newblock Preprint {\texttt{arXiv:1107.5070v1}}.

\bibitem{SV06}
{\sc Sagan, B.~E., and Vatter, V.}
\newblock The {M}\"obius function of a composition poset.
\newblock {\em J. Algebraic Combin. 24}, 2 (2006), 117--136.

\bibitem{SW11}
{\sc Sagan, B.~E., and Willenbring, R.}
\newblock Discrete {M}orse theory and the consecutive pattern poset.
\newblock Preprint {\texttt{arXiv:1107.3262v2}}.

\bibitem{rS96}
{\sc Stanley, R.~P.}
\newblock {\em Enumerative combinatorics. {V}ol. 1}, vol.~49 of {\em Cambridge
  Studies in Advanced Mathematics}.
\newblock Cambridge University Press, Cambridge, 1997.
\newblock With a foreword by Gian-Carlo Rota, Corrected reprint of the 1986
  original.

\end{thebibliography}
%NOTE: Forced capitalization of proper names in .bib file.
%Verify your references using math sci net! (done)
\end{document}